\documentclass[10pt]{amsart}
\usepackage{amssymb,amsmath,mathtools}
\usepackage{enumerate}
\usepackage{xcolor}
\usepackage{caption}

\newtheorem{theorem}{Theorem}[section]
\newtheorem{proposition}[theorem]{Proposition}
\newtheorem{lemma}[theorem]{Lemma}
\newtheorem{corollary}[theorem]{Corollary}
\theoremstyle{definition}
\newtheorem{definition}[theorem]{Definition}
\newtheorem{assumption}[theorem]{Assumption}
\newtheorem{example}[theorem]{Example}
\newtheorem{remark}[theorem]{Remark}

\newcommand{\E}{\mathbb{E}}
\renewcommand{\P}{\mathbb{P}}
\newcommand{\C}{\mathrm{\mathbb{C}ov}}
\newcommand{\V}{\mathrm{\mathbb{V}ar}}

\numberwithin{equation}{section}

\begin{document}

\title{On Monte-Carlo methods in convex stochastic optimization}

\author{Daniel Bartl}
\address{Department of Mathematics, Vienna University, Austria}
\email{daniel.bartl@univie.ac.at}
\author{Shahar Mendelson}
\address{Centre for Mathematics and its Applications, The Australian National University Canberra, Australia}
\email{shahar.mendelson@anu.edu.au}
\date{\today}

%90C15  	Stochastic programming
%90B50  	Management decision making, including multiple objectives 
%62J02  	General nonlinear regression

\begin{abstract}
	We develop a novel procedure for estimating the optimizer of general convex stochastic optimization problems of the form $\min_{x\in\mathcal{X}} \E[F(x,\xi)]$, when the given data is a finite independent sample selected according to $\xi$.
	The procedure is based on a median-of-means tournament, and is the first procedure that exhibits the optimal statistical performance in heavy tailed situations:
	we recover the asymptotic rates dictated by the central limit theorem in a non-asymptotic manner once the sample size exceeds some explicitly computable threshold.
	Additionally, our results apply in the high-dimensional setup, as the threshold sample size exhibits the optimal dependence on the dimension (up to a logarithmic factor).
	The general setting allows us to recover recent results on multivariate mean estimation and linear regression in heavy-tailed situations and to prove the first sharp, non-asymptotic results for the portfolio optimization problem.
\end{abstract}

\maketitle
\setcounter{equation}{0}
\setcounter{tocdepth}{2}
\tableofcontents

\section{Introduction and appetizer}
\label{sec:intro}

\emph{Stochastic optimization} is widely used as a way of solving certain problems numerically. It appears in diverse areas of mathematics, with a generic convex stochastic optimization problem taking the following form:
One is given a random variable $\xi$ whose range is a measurable space $\Xi$, a convex set of actions $\mathcal{X}\subseteq\mathbb{R}^d$, and a function $F\colon \mathcal{X}\times \Xi\to\mathbb{R}$ that is convex in its first argument. The objective is to solve the optimization problem
\begin{align}
\tag{SO}
\label{eq:SO}
 \min_{x\in\mathcal{X}} f(x)
\quad\text{where}\quad
 f(x):=\E[F(x,\xi)].
\end{align}
In typical situations, however, one does not have access to the function $f$ directly. Rather, the information one is given is the set of values $F(\cdot,\xi_i)_{i=1}^N$, where $(\xi_i)_{i=1}^N$ is an independent \emph{sample}, selected according to $\xi$ and of cardinality $N$.
This type of random data is natural, for example, if the distribution of $\xi$ is not known and the only possibility is to sample it;  or when an exact computation of $f$ is unfeasible and one relies on Monte-Carlo methods to evaluate it instead. We refer the reader to  Shapiro, Dentcheva, and Ruszczy{\'n}ski \cite{shapiro2014lectures} or to Kim, Pasupathy, and Henderson \cite{kim2015guide} for introductions on such aspects of stochastic optimization.

Regardless of the reason why one uses a random sample, the fundamental question remains unchanged: to what degree \eqref{eq:SO} can be recovered when the given data is a random sample?

To be more accurate, assume that \eqref{eq:SO} admits a unique optimal action $x^\ast$ and denote by $\widehat{x}_N^\ast$ a candidate for the optimal action that is selected using some sample-based procedure. 
Given a prescribed error $r>0$, one seeks to bound the probability that the \emph{estimation error} $\|\widehat{x}_N^\ast-x^\ast\|$ or the  \emph{prediction error} / \emph{optimality gap} $f(\widehat{x}_N^\ast)-f(x^\ast)$ exceeds  $r$ in terms of the sample size $N$.

It should be stressed that the norm $\| \cdot \|$ need not be the Euclidean norm. 
Rather, the right choice of $\|\cdot\|$ turns out to be a natural Hilbertian structure endowed by the Hessian of $f$. 
The reason for that is clarified in what follows.

The typical approach used to produce $\widehat{x}_N^\ast$ is called \emph{sample average approximation} (SAA) and is denoted in what follows by $\widehat{x}^{\text{SAA}}_N$. 
The choice is very natural: $\widehat{x}^{\text{SAA}}_N$ is a minimizer of the empirical mean $\widehat{f}_N:=\frac{1}{N}\sum_{i=1}^N F(\cdot,\xi_i)$.

Asymptotic properties of the SAA solution have been thoroughly investigated. 
Roughly and inaccurately put, the SAA solution behaves asymptotically as one would expect based on the central limit theorem.
However, as we shall explain immediately, these asymptotic results  can be misleading. 
In fact, unless one imposes \emph{highly restrictive integrability assumptions}, when given a finite sample the SAA solution behaves poorly: it exhibits drastically weaker rates than what one may expect based on the asymptotic behaviour.

%\vspace{1em}
In contrast to the SAA solution, we propose a \emph{novel procedure} that aims at selecting the optimal action when given a finite random sample.
This procedure exhibits the best possible performance regarding the estimation and prediction errors, and it does so in completely heavy tailed situations.
Our results are based on the methods developed by G.~Lugosi and the second-named author in \cite{lugosi2019sub} for mean estimation in $\mathbb{R}^d$ w.r.t.\ the Euclidean norm; in \cite{lugosi2019near} for mean estimation w.r.t.\ to a general norm; and in \cite{lugosi2016risk} for general regression / learning problems using the squared loss.

In Section \ref{sec:applications} we shall show that the formulation of the general problem allows one to recover the results of \cite{lugosi2019sub} as well as parts of \cite{lugosi2016risk} in the context of linear regression.
\vspace{1em}

Before explaining what we mean by ``highly restrictive integrability assumptions" and indicating the very poor behaviour of SAA in their absence, let us interject with an example of an important convex stochastic optimization problem. 
This example will accompany us throughout the article and will help concretize the results and assumptions of this article.

\begin{example}[Portfolio optimization]
\label{ex:PO}
Modern portfolio theory was initiated by Markowitz \cite{markowitz1952} and is among the central optimization problems in mathematical finance.
Without going into details, the \emph{portfolio optimization problem} has three ingredients:
\begin{itemize}%[leftmargin=2.5em]
\item a $d$-dimensional random vector $X$, which is interpreted as (discounted) future prices of some stocks/goods;
\item a random variable $Y$, which is interpreted as the (random) future payoff;
\item a concave, increasing utility function $U\colon\mathbb{R}\to\mathbb{R}$.
\end{itemize}
We assume that the stocks $X$ are available for buying and selling today at the prices $\pi$ so that a trading strategy $x\in\mathbb{R}^d$ bears the cost $\langle \pi,x\rangle:=\sum_{i=1}^d \pi_i x_i$ and gives the (random, future) payoff $\langle X,x\rangle$.
After trading according to the strategy $x$, the investor's terminal wealth is $Y+\langle X-\pi,x\rangle$ and her goal is to maximize the expected utility, namely
\[ \max_{x\in\mathcal{X}} u(x)
\quad\text{where}\quad
 u(x):=\E\left[ U \left( Y +  \langle X-\pi,x\rangle \right) \right], \]
where $\mathcal{X}\subseteq\mathbb{R}^d$ is closed and convex.
(For instance $\mathcal{X}=[0,\infty)^d$ corresponds to short-selling constraints.)
We refer, e.g., to F\"ollmer and Schied \cite[Chapter 3]{follmer2011stochastic} or to Shapiro, Dentcheva, and Ruszczynski \cite[Section 1.4]{shapiro2014lectures} for a more elaborate introduction to the portfolio optimization problem.

Setting $F(x,\xi):=\ell ( -Y - \langle X-\pi,x\rangle)$ with $\xi:=(X,Y)$ and $\ell:=-U(-\, \cdot\,)$, the portfolio optimization problem is indeed a special instance of \eqref{eq:SO}.
\end{example}

It is important to stress that the portfolio optimization problem highlights the natural presence of \emph{heavy tails}:
for one, even if the input $X$ is light-tailed (e.g.\ Gaussian, as is the case in the Bachelier model), the composition with the utility function $U$ can render the problem heavy-tailed.
This is particularly true for the exponential utility function $U=-\exp(-\,\cdot\,)$, which is arguably (among) the most important utility functions.
Additionally, $X$ itself is often heavy tailed; for instance, in the famous Black-Scholes model $X$ is log-normal.

\vspace{1em}
Before explaining the new procedure, it is worthwhile to outline what is known on the statistical performance of the sample average approximation method. As we already mentioned, the vast majority of known results are of an \emph{asymptotic} nature. 
To ease notation we assume here and for the rest of this section that $\nabla^2 f(x^\ast)=\mathrm{Id}$ (recall that $x^\ast$ is the unique minimizer of $f$).

One can show that under suitable regularity and mild integrability assumptions, $\sqrt{N}( \widehat{x}^{\mathrm{SAA}}_N - x^\ast )$ is asymptotically a multivariate Gaussian with zero mean and covariance matrix $\C[\nabla F(x^\ast,\xi)]$, see e.g.\ \cite[Chapter 5]{shapiro2014lectures}.
In particular, if we set
\[\sigma^2:=\lambda_{\max}( \C [\nabla F(x^\ast,\xi)])\]
to be the largest eigenvalue of the covariance matrix of the gradient of $F$ at $x^\ast$  and denote by  $\|\cdot\|_2$ the Euclidean norm, it follows that
\begin{align*}
%\label{eq:intro.saa.estimaton.error.asy}
\P\left[ \| \widehat{x}^{\mathrm{SAA}}_N - x^\ast \|_2 \geq r \right]
\approx 2\exp \left( - cN\frac{r^2}{\sigma^2} \right)
\quad\text{asymptotically as } N\to\infty.
\end{align*}
This error rate is often used to calculate the minimal sample size $N$ required to guarantee that the estimation error is below the wanted threshold $r$ with some prescribed confidence $1-\delta$.
However, as we shall explain in Section \ref{sec:non.gaussian} below, unless (restrictive) integrability assumptions are imposed, the asymptomatic exponential decay really does hold \emph{only asymptotically}.
Indeed, it may very well be possible that the \emph{non-asymptotic} (or, finite sample) rate
\begin{align*}
%\label{eq:intro.saa.estimaton.error.non.asy}
\P\left[ \| \widehat{x}^{\mathrm{SAA}}_N - x^\ast \|_2 \geq r \right]
\leq \frac{c\sigma^2}{Nr^2}
\quad\text{for all } N\geq 1
\end{align*}
cannot be improved, meaning that the finite sample rate decays linearly in $N$ rather than exponentially.

The meaning of this significant gap between asymptotic and finite sample behaviour is that the asymptotic estimate is misleading: while it suggests that $\frac{\sigma^2}{r^2} \log(\frac{2}{\delta})$ samples are enough to guarantee a confidence of $1-\delta$, one actually needs $\frac{\sigma^2}{r^2} \frac{  1 }{\delta}$ samples. For small values of $\delta$ this gap in the required sample size is significant.

\vspace{1em}
\emph{Our contribution} is the following.
We construct a procedure that estimates the optimal action $\widehat{x}_N^\ast$ --- but it is not SAA. 
Recalling that  $\nabla^2f(x^\ast)=\mathrm{Id}$ throughout this section for notational simplicity, given a finite sample, the procedure recovers the optimal Gaussian rate under modest integrability assumptions:
for (small) $r>0$ we have
\begin{align*}
\P\left[
\| \widehat{x}^\ast_N - x^\ast  \|_2 \geq r \right] %\text{ or }
%f(\widehat{x}_N^\ast)\geq f(x^\ast) + r^2 \right]
\leq 2\exp\left(-C N \frac{r^2}{\sigma^2} \right)
\quad\text{whenever } N\geq N_0(r),
\end{align*}
where $N_0(r)$ can be controlled explicitly and depends only on certain low-order moments of the random variables involved (see Theorem \ref{thm:main.convex.intro} for the precise statement).
As a matter of fact, in typical situations
\[N_0(r) \lesssim \max\left\{d\log (d) , \frac{\mathrm{trace}(\C[\nabla F(x^\ast,\xi)])}{ r^2} \right\}.\]
It is worthwhile mentioning that the prediction error is one order of magnitude smaller than the estimation error, namely, for $N\geq N_0(r)$
\[\P\left[ f(\widehat{x}_N^\ast)\geq f(x^\ast) + r^2 \right]
\leq  2\exp\left(-C N \frac{r^2}{\sigma^2} \right). \]

Although this is a trivial consequence of the first order condition for optimality if $\mathcal{X}=\mathbb{R}^d$, it is far less obvious if $\mathcal{X}$ is a proper subset of $\mathbb{R}^d$ and $x^\ast$ lies on the boundary of $\mathcal{X}$.

\begin{remark}
\label{rem:computational}
	While our procedure showcases the possibility of drastically improving the statistical properties of the SAA, this improvement does not come for free.
	A major advantage of SAA is its computational simplicity,
	and unfortunately, our procedure is the outcome of a (rather complex) tournament that takes place between the actions in $\mathcal{X}$ (see Section \ref{sec:the.procedure}).
	In particular, its computational cost is quite high---in fact, it is likely to be computationally intractable (i.e.\ NP-hard).
	Tournament based procedures are used in other natural statistical problems and there are ongoing attempts of identifying alternative procedures that maintain the tournament's optimal statistical performance in a computationally feasible way.
	For example, in the context of multivariate mean-estimation w.r.t.\ the Euclidean norm, Hopkins \cite{hopkins2020mean} defines a semi-definite relaxation of the tournament procedure that can be computed in $O(Nd+ d\log(\frac{1}{r})^c)$-time (where $r$ is the wanted accuracy and $c$ is a dimension dependent constant).
	Cherapanamjeri, Flammarion, and Bartlett \cite{cherapanamjeri2019fast} improved Hopkins' result and introduced a relaxation running in $O(Nd+ d\log(\frac{1}{r})^2 + \log(\frac{1}{r})^4 )$-time.
	
	We believe that these results are not the end of the story and that similar ideas can be applied in the preset setting as well.
	We do think that a detailed investigation of the tradeoff between statistical optimality and computational feasibility is a challenge and of central importance.
	However, we will not pursue this aspect further in what follows.
\end{remark}

%\vspace{0.5em}
\noindent
{\bf Plan of the article.}
We begin Section \ref{sec:main} with a detailed explanation of the devastating effect heavy-tailed random variables can have on SAA; we then formulate our main result (Theorem \ref{thm:main.convex.intro}), discuss its application to the portfolio optimization problem, and survey related literature. Section \ref{sec:applications} contains a description of several other applications of our main result. In Section \ref{sec:mom.singular.value} we lay the groundwork for the proof of Theorem \ref{thm:main.convex.intro} by establishing a high probability lower bound on the smallest singular value of a general random matrix ensemble (see Theorem \ref{thm:singular.value} and Corollary \ref{cor:smallest.singular.value})---a result that is of independent interest.
This lower bound will be used in Section \ref{sec:proof.main}, where we prove our main result.
Proofs related to the portfolio optimization problem are presented in Section \ref{sec:PO.proof}.
Finally, Section \ref{sec:concluding.remark} contains two concluding remarks.

\section{Main results}
\label{sec:main}

\subsection{Difficulties caused by non-Gaussian tails}
\label{sec:non.gaussian}
Let us revisit our claim that heavy tails drastically change the statistical performance of the sample average approximation.
As a starting point, consider the more basic problem of estimating the mean $\mu:=\E[\xi]$ of a one-dimensional, square integrable random variable $\xi$.
It should be noted that by setting \ $F(x,\xi):=\frac{1}{2}x^2-\xi x$, one-dimensional mean estimation becomes a stochastic optimization problem.

Following the SAA approach, the corresponding estimator for the mean is $\widehat{\mu}_N:=\frac{1}{N}\sum_{i=1}^N \xi_i$.
The central limit theorem then guarantees that
\begin{align}
\label{eq:mean.asy}
\P[ | \widehat{\mu}_N - \mu | \geq r] \leq 2\exp\left(-N\frac{ r^2 }{ 2\sigma^2 }\right)
\quad\text{asymptotically as } N \to\infty,
\end{align}
where $\sigma^2:=\V[\xi]$ denotes the variance of $\xi$.
On the other hand, invoking Markov's inequality to bound the probability in \eqref{eq:mean.asy} for finite $N$ only implies that
\begin{align}
\label{eq:mean.non.asy}
\P[ | \widehat{\mu}_N - \mu| \geq r] \leq \frac{\sigma^2}{N r^2}
\quad\text{for every } N\geq 1
\end{align}
and that dictates a much slower rate than the bound obtained in \eqref{eq:mean.asy}.
It should be stressed that the weaker bound is not caused by looseness in Markov's inequality; in fact, there are examples where \eqref{eq:mean.non.asy} is sharp (up to a constant).
Indeed, let $r$ and $N$ such that $Nr^2\geq 1$, and let $\xi$ be the symmetric random variable taking the values $\pm Nr$ with probability $\frac{1}{2(Nr)^2}$ and $0$ with probability $1-\frac{1}{(Nr)^2}$. 
Then $\mu=0$ and $\sigma^2=1$. 
Moreover, given a sample of cardinality $N$, a straightforward computation shows that there is an absolute constant $C$  such that the following holds: with probability at least $\frac{C}{Nr^2}$, exactly one of the sample points is nonzero.
On that event we clearly have $|\widehat{\mu}_N-\mu| =r$, showing that the estimate in Markov's inequality \eqref{eq:mean.non.asy} is sharp (up to the absolute multiplicative constant $C$).

In order to improve \eqref{eq:mean.non.asy} one needs to impose a stronger integrability assumption and, eventually, one can show that \eqref{eq:mean.asy}  holds for all $N\geq 1$ if and only if $\xi$ has sub-Gaussian tails in the sense that $\P[ |\xi-\E[\xi]| \geq t]$ is at most of the order $\exp(-c\frac{t^2}{\sigma^2})$ for $t \geq c^\prime \sigma$.

At this point, sub-Gaussian tails seem unavoidable if one's goal is to have finite sample estimates that match the asymptotic ones (as dictated by the central limit theorem).
There is, however, one important possibility that so far has been neglected:
we are free to come up with an alternative estimator instead of the empirical average.
To explain, at an intuitive level, how this might be a way out of our predicament, note that the non-optimal performance of the empirical mean stems from the fact that, in the presence of heavy tails, some of the observations will have untypically large values.
These observations, while few in numbers, offset the empirical mean from its true counterpart, and the hope is that getting rid of those outliers would lead to a better statistical performance.
The so-called \emph{median-of-means} estimator is a simple yet powerful estimator that does just that. It goes back at least to Nemirovsky and Yudin \cite{nemirovsky1983problem}.

\vspace{1em}
Partition the sample $\{1,\dots,N\}=\cup_{j=1}^n I_j$ into $n$ disjoint blocks $I_j$ of cardinality $m:=3\frac{\sigma^2}{r^2}$ (w.l.o.g.\ assume that $n$ and $m$ are integers).
By \eqref{eq:mean.non.asy}, we have that
\[\P[| \widehat{\mu}_{I_j} - \mu| \geq r]\leq \frac{1}{3}
\quad\text{where}\quad
\widehat{\mu}_{I_j}:=\frac{1}{m}\sum_{i\in I_j} \xi_i \]
and a basic Binomial calculation reveals that the probability that the majority of the $n$ blocks satisfy $| \widehat{\mu}_{I_j} - \mu| \geq r$ is of the order of  $\exp(-cn)$.
Since $n= N\frac{r^2}{3\sigma^2}$, we conclude that
\[ %\P\left[ \left| \E[\xi]- \mathop{\mathrm{median}}_{j=1,\dots,n} \frac{1}{m}\sum_{i\in I_j} \xi_i  \right| \geq r  \right]
\P\left[ \left| \mathop{\mathrm{median}}_{j=1,\dots,n} \,\widehat{\mu}_{I_j} -\mu \right| \geq r  \right]
\leq 2\exp\left(- C N\frac{r^2}{\sigma^2} \right)
\quad\text{for all }N\geq 1. \]
In other words, the median-of-means estimator exhibits the best possible performance \eqref{eq:mean.asy} (up to a multiplicative constant) under the sole assumption that $\xi$ has a finite second moment.

\vspace{0.5em}
Appealing as this sounds, it is important to stress that the median is a one-dimensional object and has no simple vector-valued analogue. 
In fact, the question of an optimal multivariate mean estimation procedure, assuming only that the vector has a finite mean and covariance, remained open until it was resolved recently in \cite{lugosi2019sub}.
In contrast, stochastic optimization is a multi-dimensional problem, and just like multivariate mean estimation, simply minimizing the functional $\mathop{\mathrm{median}}_j \frac{1}{m}\sum_{i\in I_j} F(\cdot,\xi_i)$ does not lead to an optimal estimator. What has a better chance of success is the \emph{tournament procedure} which happens to be a powerful extension of the idea of median-of-means.
We will explain the procedure in Section \ref{sec:the.procedure} below.

\subsection{What to expect when tails are Gaussian}
\label{sec:gaussian}

Ignoring for a second the difficulties caused by non-Gaussian tails, let us explain the kind of result one could hope for in general convex stochastic optimization problems and how the underlying dimension $d$ enters the picture. This will serve as our benchmark in what follows.

To that end, consider the case where $F$ is a \emph{quadratic function} defined on the whole of $\mathbb{R}^d$, that is,
\begin{align}
\label{eq:intro.F.quadratic}
F(x,\xi) := \langle b,x\rangle + \frac{1}{2}\langle Ax,x\rangle
\qquad\text{for } x\in\mathcal{X}:=\mathbb{R}^d,
\end{align}
where $b=b(\xi)$ is a $d$-dimensional Gaussian vector with zero mean and $A=A(\xi)$ is a random positive definite symmetric $(d\times d)$-matrix specified in what follows.
Although this example appears to be very special, it should not be considered as a toy example: by a second order Taylor expansion, every convex function is approximately a quadratic function, at least locally, around the minimizer.

By \eqref{eq:intro.F.quadratic}, it follows that
\[\nabla^2 f(x^\ast)=\E[A],
\quad \quad
\nabla F(x^\ast,\xi)=b + A x^\ast
\quad\text{and}\quad
\nabla^2 F(x^\ast,\xi)= A,\]
and we assume throughout that $\E[A]$ is non-degenerate (i.e.\ $\E[A]$ has full rank). 
Setting $\|z \|:=\langle \nabla^2 f(x^\ast) z,z\rangle^{\frac{1}{2}}$ for $z\in\mathbb{R}^d$ and recalling that $b$ has zero mean, it is evident that $f=\frac{1}{2}\|\cdot\|^2$; thus, the optimal action is given by $x^\ast=0$.

\begin{remark}
	To get a clearer picture of this setup, it might help the reader to first consider the case $\nabla^2 f(x^\ast)=\mathrm{Id}$, and then $\|\cdot\|$ is just the Euclidean norm.
\end{remark}

The advantage of using the quadratic function considered here is that the sample average approximation optimizer has a particularly simple form: $\widehat{x}_N^{\text{SAA}}$ is any element satisfying
\begin{align}
\label{eq:intro.quad.function.saa.solution}
\left( \frac{1}{N}\sum_{i=1}^N A_i\right) \widehat{x}_N^{\text{SAA}}=  -\frac{1}{N}\sum_{i=1}^N b_i.
\end{align}
To explain the statistical behavior of $\widehat{x}_N^{\text{SAA}}$, let us first focus on the gradient and assume for the sake of simplicity that the \emph{Hessian is deterministic}, that is, $A=\E[A]$.
In that case, $\nabla^2 f(x^\ast)\widehat{x}_N^{\text{SAA}}$ is Gaussian with mean $x^\ast=0$ and covariance $\frac{1}{N}\C[\nabla F(x^\ast,\xi)] $. A straightforward computation (noting that $\|\cdot\|=\|\nabla^2 f(x^\ast)^{\frac{1}{2}}\cdot\|_2$) reveals the following:
 for the estimation error $ \|\widehat{x}_N^{\text{SAA}} -x^\ast \| $ to be smaller than $r$ with constant probability (say, with probability at least $\frac{1}{2}$), it is necessary to have a sample size of cardinality larger than $N\geq N_{\mathrm{G}}(r)$, where
\begin{align}
\label{eq:def.N.G}
N_{\mathrm{G}}(r)
&:=\frac{1}{r^2}\mathop{\mathrm{trace}}\left( \nabla^2f(x^\ast)^{-1}  \C[\nabla F(x^\ast,\xi)] \right).
\end{align}
On the other hand, for large $N$,  it follows from the concentration of a Lipschitz function of the Gaussian vector that the probability that the estimation error exceeds $r$ is of the order $2\exp(-c N\frac{ r^2 }{ \sigma^2 } )$. And the variance parameter is
\begin{align}
\label{eq:def.sigma}
\sigma^2
%&:= \max_{ z\in\mathbb{R}^d \text{ s.t. } \|z\|\leq 1} \langle \C[\nabla F(x^\ast,\xi)]z,z\rangle \\
&:= \lambda_{\max}\left( \nabla^2f(x^\ast)^{-1}  \C[\nabla F(x^\ast,\xi)]  \right).
\end{align}
To summarize, when the quadratic function has a deterministic Hessian, the minimal sample  size needed to guarantee that the estimation error does not exceed $r$ with constant probability is $N_{\mathrm{G}}(r)$, whereas the correct variance parameter (namely $\sigma^2$) dictates the high-probability regime.
Note that the latter, of course, matches the variance parameter of \cite[Chapter 5]{shapiro2014lectures}, as stated in Section \ref{sec:intro}.

In a next step, still within the setting of the quadratic function \eqref{eq:intro.F.quadratic}, let us remove the assumption that the Hessian is deterministic.
In that case, if one wishes to make any statement regarding the estimation (or, prediction) error, the empirical Hessian $\frac{1}{N}\sum_{i=1}^N A_i$ on the left hand side of \eqref{eq:intro.quad.function.saa.solution} must not be degenerate.
One can readily verify that, unless some specific assumptions are made, the empirical Hessian is singular with probability 1 whenever $N<d$.
Thus, $N\geq d$ is another restriction on the minimal sample size (though, at this point, it is far from obvious that a sample of size $d$ or proportional to $d$ would suffice to guarantee a non-degenerate Hessian with, say, constant probability).

\vspace{1em}
Following these observations one can make a very \emph{optimistic guess} on the estimate one can hope to obtain:
that there exists a procedure $\widehat{x}^N_\ast$ such that for
\[ N\geq \max\{ N_{\mathrm{G}}(r) ,d \}, \]
with probability at least
\[ 1-2\exp\left( -cN\frac{ r^2}{\sigma^2}\right), \]
we have that
\[ \|\widehat{x}_N^\ast-x^\ast\|\leq r.\]

\begin{remark}
This is indeed an optimist guess, and is ``very Gaussian". The minimal sample size $\max\{ N_{\mathrm{G}}(r),d\}$ is the result of rather trivial obstructions; their removal is necessary if the estimation error is to have any chance of being smaller than $r$ with \emph{constant} probability.
Furthermore, the variance term $\sigma^2$, which dictates the \emph{high} probability regime (for large $N$), is effectively one dimensional:  it corresponds to the worst direction of the gradient (w.r.t.\ the norm $\|\cdot\|$).
\end{remark}

%\vspace{1em}
Before we proceed with the main (affirmative) result of this article, let us conclude this section with a comment regarding the relation between the estimation error and the prediction error / the optimality gap, in a more general setup than the simple example we presented previously.

If $\mathcal{X}=\mathbb{R}^d$, a Taylor expansion and the first order condition for optimality of $x^\ast$ immediately implies that
\begin{align*}
f(x) -f(x^\ast)
&= \frac{1}{2}\langle \nabla^2 f(y)(x-x^\ast),x-x^\ast\rangle
\end{align*}
where $y$ is some mid-point between $x^\ast$ and $x$.
In particular, setting $\|B\|_{\mathrm{op}}:=\sup_{z\in\mathbb{R}^d \text{ s.t.\ } \|z\|\leq  1} \langle Bz,z\rangle$ to be the operator norm\footnote{
	Note that $\|\cdot\|_{\mathrm{op}}$ is indeed the operator norm from $(\mathbb{R}^d,\|\cdot\|)$ to $(\mathbb{R}^d,\|\cdot\|_\ast)$, where $\|\cdot\|_\ast$ is the dual norm of $\|\cdot\|$, i.e.\ $\|y\|_\ast:=\sup_{\|x\|\leq 1} \langle x,y\rangle$.
	}
of a positive semi-definite $(d\times d)$-matrix $B$, and 
\begin{align}
\label{eq:def.c.H}
c_{\mathrm{H}}:= \sup_{y\in\mathcal{X} \text{ s.t.\ } \|y-x^\ast\|<  1} \,\,  \frac{1}{2} \|\nabla^2 f(y) \|_{\mathrm{op}}
\end{align}
 it is clear that
\[f(x)- f(x^\ast)  \leq c_{\mathrm{H}} r^2
\quad\text{whenever }\|x-x^\ast\|\leq r\]
and $r<1$.
Thus, one can make another \emph{highly optimistic guess}:
that there exists a procedure $\widehat{x}_N^\ast$ for which the prediction error / optimality gap is smaller than the estimation error by at least one order of magnitude.

What makes this guess optimistic (and nontrivial), is that the above argument crucially relies on the fact that $\mathcal{X}=\mathbb{R}^d$; or, more generally, that $x^\ast$ lies in the interior of $\mathcal{X}$. That need not be the case.

\subsection{Recovering Gaussian rates without sub-Gaussian tails}
\label{sec:main.results}

This section contains the main result of the article, formulated in Theorem \ref{thm:main.convex.intro} below. It provides affirmative answers to the optimistic guesses made in the previous section (under some mild assumptions, of course).
The assumptions might appear technical at first glance, and to help the reader put them in context, each assumption will be explained in the case of portfolio optimization, and in a heuristic manner; the detailed analysis will be presented in Section \ref{sec:PO.proof}.

\begin{assumption}[Convexity and coercivity]
\label{ass:on.F.diff.convex.etc}
	The following hold:
\hfill
	\begin{enumerate}[(a)]
	\item
	$\mathcal{X}\subseteq\mathbb{R}^d$ is closed and convex;
	\item
	$x\mapsto F(x,\gamma)$ is convex and twice continuously differentiable\footnote{
		If $\mathcal{X}$ is not open, we mean by ``continuously differentiable" that there is a continuous function $\nabla F(\cdot,\gamma)$ such that $F(y,\gamma)-F(x,\gamma)=\int_0^1 \langle \nabla F(x+t(y-x),\gamma),y-x\rangle \,\mathrm{d}t$. 
		A similar notion holds for $\nabla^2 F$ for the ``twice continuously differentiable".}
		for every $\gamma\in \Xi$;
	\item
	$F(x,\xi)$, and $\nabla^2 F(x,\xi)$ are integrable and $\nabla F(x,\xi)$ is square integrable for every $x\in\mathcal{X}$.	
	\end{enumerate}	
	Further, there exists an \emph{optimal action} $x^\ast\in\mathcal{X}$ that satisfies $f(x^\ast)=\inf_{x\in\mathcal{X}} f(x)$, and the seminorm induced by the Hessian of $f$ at $x^\ast$ given by
	\[ \|z\|:= \langle \nabla^2 f(x^\ast) z,z\rangle^{\frac{1}{2}}
	\quad\text{for } z\in\mathbb{R}^d\]
	is a true norm (i.e.\ $\|z\|=0$ implies $z=0$).
\end{assumption}

\begin{remark}
	While $f$ clearly inherits convexity from $F$, it is not clear a priori that $f$ is twice continuously differentiable (in the same sense as $F$ if $\mathcal{X}$ is not open).
	This follows once Assumption \ref{ass:hessian.midpoints} below is imposed, as we shall explain in the beginning of Section \ref{sec:proof.main}.
\end{remark}

Note that the minimizer $x^\ast$ in Assumption \ref{ass:on.F.diff.convex.etc} is unique, as $\|\cdot\|$ is a norm.
In fact, the latter relates to a standard assumption in stochastic optimization---the so called \emph{quadratic growth condition}: that there is a constant $\kappa>0$ such that
\begin{align*}
%\label{eq:quad.growth}
f(x)\geq f(x^\ast) + \kappa  \| x-x^\ast \|_2^2
\end{align*}
for all $x$ close to $x^\ast$.
Indeed, denoting $\tilde{\kappa}:=\lambda_{\min}(\nabla^2 f(x^\ast))$, the smallest eigenvalue of the Hessian of $f$ at $x^\ast$, we have that $\tilde{\kappa}>0$ whenever $\|\cdot\|$ is true norm.
Moreover, a Taylor expansion shows that the quadratic growth condition holds with  constant $\kappa=\tilde{\kappa}$ for all $x$ in an infinitesimal neighbourhood of $x^\ast$ (or with  constant  $\kappa=\frac{\tilde{\kappa}}{2}$ in a sufficiently small neighbourhood).
Conversely, at least when $x^\ast$ lies in the interior of $\mathcal{X}$, the quadratic growth condition readily implies $\tilde{\kappa}\geq \kappa$ and in particular that $\|\cdot\|$ is a true norm.

\vspace{0.7em}
%\par
%\begingroup
%\leftskip=3.5em % Parameter anpassen
%\noindent
\begin{quote}
Let us now give an intuitive interpretation of Assumption \ref{ass:on.F.diff.convex.etc} in the context of the portfolio optimization example.
To ease notation, we shall assume that $\pi=\E[X]=0$ and that $\C[X]=\mathrm{Id}$.
The convexity and differentiability parts of the assumption are clear, and it is straightforward to verify that
\begin{align*}
	\nabla F(x,\xi)
	&= -\ell' ( -Y - \langle X,x \rangle) X,\\
	\nabla^2 F(x,\xi)
	&= \ell'' ( -Y- \langle X,x \rangle) X \otimes X.
\end{align*}
In particular
\[\|z\|^2 = \E[\ell''( -Y - \langle X,x^\ast \rangle) \langle X,z\rangle^2].\]
Ignoring the $\ell''$-term inside the expectation for the moment, this would imply that $\|\cdot\|=\|\cdot\|_2$. In general, when accounting for the $\ell''$-term, a minor integrability assumption will be used to guarantee that $\|\cdot\|$ and $\|\cdot\|_2$ are equivalent norms.
\end{quote}
%\par
%\endgroup
\vspace{0.7em}

%\noindent
In Section \ref{sec:gaussian} we argued that unless the Hessian has a specific form, the empirical Hessian is singular whenever $N \leq d$.
However, without further assumptions, believing that a sample of cardinality $d$ is enough to guarantee a non-degenerate empirical Hessian  (say with constant probability) is too optimistic---certainly in the general setting we are interested in here. Degeneracy can happen even in dimension $d=1$: if  $\nabla^2 F(x^\ast,\xi)$ takes the value $\frac{1}{\varepsilon}$ with probability $\varepsilon>0$ and is zero otherwise, the endowed norm $\|\cdot\|$ is simply the absolute value---regardless of the choice of $\varepsilon$.
However, with probability $(1-\varepsilon)^N$, all observations in a sample of size $N$ are zero, and  for small $\varepsilon$, e.g.\ $\varepsilon=\frac{1}{N^2}$, that probability converges to 1 as $N\to\infty$.

As it happens, the following modest integrability assumption on the Hessian prevents such behavior.

\begin{assumption}[Integrability of the Hessian]
\label{ass:norm.equiv}
	There is a constant $L$ such that
	\[\E[\langle \nabla^2 F(x^\ast,\xi)z,z\rangle^2] \leq L\]
	for every $z\in\mathbb{R}^d$ with $\|z\|=1$.	
\end{assumption}

Another way of formulating Assumption \ref{ass:norm.equiv} is in the sense of \emph{norm-equivalence}: for every $z\in\mathbb{R}^d$, we have that
\begin{align}
\label{eq:norm.equiv.spelled.out}
\E[\langle \nabla^2 F(x^\ast,\xi)z,z\rangle^2]^{\frac{1}{4}}
\leq L^{\frac{1}{4}}  \E[\langle \nabla^2 F(x^\ast,\xi)z,z\rangle]^{\frac{1}{2}},
\end{align}
i.e.\ the $L_4$-norm and the $L_2$-norm of the forms $\langle \nabla^2 F(x^\ast,\xi)z,z\rangle^{\frac{1}{2}}$ are equivalent\footnote{Note that the reverse inequality to \eqref{eq:norm.equiv.spelled.out} is trivially true with constant $1$, by H\"older's inequality.}.
Therefore the constant $L$ pertains to the worst direction $z\in\mathbb{R}^d$ (and not e.g.\ the average over different directions).
As such, $L$ typically does not depend on the dimension $d$.

Under Assumption \ref{ass:norm.equiv} one can prove a lower bound on the smallest singular value of the  empirical Hessian whenever $N \gtrsim d\log(d)$.

\vspace{0.7em}
%\par
%\begingroup
%\leftskip=3.5em % Parameter anpassen
%\noindent
\begin{quote}
To check Assumption \ref{ass:norm.equiv} in the portfolio optimization example, note that
\[ \E[\langle \nabla^2 F(x^\ast,\xi)z,z \rangle^2]
=\E[\ell''( -Y - \langle X,x^\ast\rangle)^2  \langle X,z\rangle^{4}]. \]
Thus, Assumption \ref{ass:norm.equiv} is a simple consequence of H\"older's inequality and a minor integrability condition.
\end{quote}
%\par
%\endgroup
\vspace{0.7em}

\begin{remark}
Let us stress that Assumption \ref{ass:norm.equiv} is just a tractable way of ensuring that our argument works; it could be replaced by the more general assumption that $\nabla^2 F(x^\ast,\xi)$ satisfies a so-called \emph{stable lower bound}, see Remark \ref{rem:hessian.SLB}.
The stable lower bound and its role in obtaining lower bounds on the smallest singular values of rather general random matrix ensembles is described in Theorem \ref{thm:singular.value} and in Corollary \ref{cor:smallest.singular.value.norm.equiv}.
\end{remark}

\vspace{0.7em}

There is another reason why the minimal sample should be at least a (large constant) multiple of $d$, namely, because $F$ need not be quadratic.
In the example in Section \ref{sec:gaussian} $F$ was a quadratic function, and as a result the Hessian did not depend on the action $x$.
In general, when invoking a second order Taylor expansion, the Hessian does depend on some mid-point $x^\ast + t(x-x^\ast)$.
At the same time, Assumption \ref{ass:norm.equiv} only takes into account the Hessian at the optimizer; therefore, some continuity assumption is needed if one is to control the deviation from quadratic, which is governed by
\begin{align*}
\mathcal{E}_{\mathrm{H}}(x)
:=\sup_{ t\in[0,1] } \left| \left\langle \left(\nabla^2 F(x^\ast \!\! + \! t(x \! - \! x^\ast) ,\xi) -\nabla^2 F(x^\ast ,\xi) \right) (x-x^\ast), x-x^\ast \right\rangle \right|
%:=\sup_{ t\in[0,1] } \left| \left\langle \left(\nabla^2 F(x(t) ,\xi) -\nabla^2 F(x^\ast ,\xi) \right) (x-x^\ast), x-x^\ast \right\rangle \right|
\end{align*}
for $x\in\mathcal{X}$.

Note that $\mathcal{E}_{\mathrm{H}}(x)$ is likely to be of order $\|x-x^\ast\|^3$ under a suitable Lipschitz condition on the Hessian. 
Assumption \ref{ass:hessian.midpoints} is there to ensure that $\mathcal{E}_{\mathrm{H}}(x)$ is sufficiently small.

\begin{assumption}[Continuity of the Hessian]
\label{ass:hessian.midpoints}
	There exists a radius $r_0\in(0,1)$ such that the following hold.
	\begin{enumerate}[(a)]
	\item
	There is a H\"older coefficient $\alpha\in(0,1]$ and a measurable function $K\colon\Xi\to[0,\infty)$ such that $\E[K(\xi)]<\infty$ and for all $x,y\in\mathcal{X}$ with $\|x-x^\ast\|, \|y-x^\ast\| \leq  r_0$, we have that
	\begin{align*}
	\left\|  \nabla^2 F(x,\xi) - \nabla^2 F (y,\xi) \right\|_{\mathrm{op}}
	&\leq  \|x-y\|^\alpha  K(\xi).
	\end{align*}
%	for all $x,y\in\mathcal{X}$ with $\|x-x^\ast\|, \|y-x^\ast\| \leq r_0$.
	\item
	For some given constant $c_1$ (which will be specified in Theorem \ref{thm:main.convex.intro} and depends only on the parameter $L$ of Assumption \ref{ass:norm.equiv}) and for all $x\in\mathcal{X}$ with $\|x-x^\ast\|\leq r_0$, we have that
	\begin{align*}
	\P\left[ \mathcal{E}_{\mathrm{H}}(x) \leq \frac{ \|x-x^\ast\|^2}{8} \right]
	&\geq (1-c_1).
	\end{align*}
	\end{enumerate}
\end{assumption}	

Assumption \ref{ass:hessian.midpoints} implies that, setting
\begin{align}
\label{eq:def.N.H.E}
 N_{\mathrm{H},\mathcal{E}}
	&:= \frac{d}{\alpha} \log\left(r_0^\alpha \E[K(\xi)] +2\right),
\end{align}
whenever $N \gtrsim  N_{\mathrm{H},\mathcal{E}}$, the error caused by replacing $F$ by its quadratic approximation does not distort the outcome by too much.

\begin{remark}
At a first glance it might seem as if part (b) of Assumption \ref{ass:hessian.midpoints} follows from part (a). 
It is true that $\mathcal{E}_{\mathrm{H}}(x)\leq \|x-x^\ast\|^{2+\alpha}  K(\xi)$, but there is one important difference:
in typical situations, $\mathcal{E}_{\mathrm{H}}$ does not depend on the dimension $d$, while $K$ does (we shall see this phenomenon in the portfolio optimization problem).
In particular, estimating $\mathcal{E}_{\mathrm{H}}$ using $K(\xi)$ will unnecessarily force the threshold radius $r_0$ to depend on the dimension, which is something we wish to avoid.
On the other hand, the dimension-dependent term $\E[K(\xi)]$ only appears through a logarithmic factor in the minimal sample size.
\end{remark}

\vspace{0.7em}
%\par
%\begingroup
%\leftskip=3.5em % Parameter anpassen
%\noindent
\begin{quote}
Returning to the portfolio optimization problem, let us, for the sake of a simplified exposition,  pretend  that $\ell''$ is $1$-Lipschitz continuous, and recall that $\| \cdot \|_2$ and $\| \cdot \|$ are equivalent norms.
Then
\begin{align}
\label{eq:PO.intro.hessian.mid}
\begin{split}
&\left| \left\langle (\nabla^2 F(x,\xi) - \nabla^2 F (y,\xi))z,z\right\rangle\right|\\
&= |\ell''( -Y - \langle X,x\rangle) -\ell''( -Y -\langle X,y\rangle) | \, \langle X,z\rangle^2 \\
&\leq | \langle X,x-y\rangle| \, \langle X,z\rangle^2.
\end{split}
\end{align}
Thus $\mathcal{E}_{\mathrm{H}}(x)\leq | \langle X,x-x^\ast\rangle|^3$ and, under some mild integrability assumption, the latter term behaves like $\|x-x^\ast\|_2^3$ on average.
Markov's inequality and the fact that $\| \cdot \|$ and $\| \cdot \|_2$ are equivalent norms imply that
\[\P\left[ \mathcal{E}_{\mathrm{H}}(x) >\frac{1}{8}\|x-x^\ast\|^2\right] 
\leq \frac{ 8 \E[\mathcal{E}_{\mathrm{H}}(x)] }{ \|x-x^\ast\|^2 }
\leq \frac{ c8 \E[\mathcal{E}_{\mathrm{H}}(x)] }{ \|x-x^\ast\|_2^2 }
\]
is of order $\|x-x^\ast\|_2$.

\noindent
On the other hand, \eqref{eq:PO.intro.hessian.mid} implies that
\[ \|\nabla^2 F(x,\xi) - \nabla^2 F (y,\xi)\|_{\mathrm{op}}\leq K(\xi)  \|x-y\|_2 ,\]
for $K(\xi):=\|X\|_2^3$ which typically scales like $d^{\frac{3}{2}}$.
%\par
%\endgroup
\end{quote}
\vspace{0.7em}

With all the definitions set in place let us turn to the formulation of our main result. Recall that $N_{\mathrm{G}}(r)$, $\sigma^2$, $c_{\mathrm{H}}$, and $N_{\mathrm{H},\mathcal{E}}$ were defined in \eqref{eq:def.N.G},  \eqref{eq:def.sigma},  \eqref{eq:def.c.H}, and \eqref{eq:def.N.H.E} respectively.

\begin{theorem}[Estimation and prediction error]
\label{thm:main.convex.intro}
	There are constants $c_1,c_2,c_3$ depending only on $L$ such that the following holds.
	Assume that Assumptions \ref{ass:on.F.diff.convex.etc}, \ref{ass:norm.equiv}, \ref{ass:hessian.midpoints} hold, let $r\in (0,r_0)$ and consider
	\[N\geq c_2 \max\{  d\log(2d), N_{\mathrm{H},\mathcal{E}} , N_{\mathrm{G}}(r) \}.\]
	Then there is a procedure $\widehat{x}_N^\ast$ 	such that, with probability at least
	\[ 1-2\exp\left(-c_3 N \min\left\{1, \frac{r^2}{\sigma^{2}  } \right\} \right), \]
	we have that
	\begin{align*}
	\| \widehat{x}_N^\ast - x^\ast \| &\leq  r, \\\
	f(\widehat{x}_N^\ast) &\leq f(x^\ast) + 2c_{\mathrm{H}}  r^2.
	\end{align*}
	The procedure is described in Section \ref{sec:the.procedure}. 
\end{theorem}

Theorem \ref{thm:main.convex.intro} implies that our procedure recovers (up to multiplicative constants) the optimal asymptotic rates for the sample average approximation \cite[Chapter 5]{shapiro2014lectures} in a non-asymptotic fashion and when the random variables involved can be heavy tailed.

The procedure $\widehat{x}_N^\ast$ is described in Section \ref{sec:the.procedure}.
The parameters (e.g., $\sigma^2$, $\|\cdot\|$ etc.) appearing in Theorem \ref{thm:main.convex.intro} depend on the unknown optimal action $x^\ast$ so that their a priori knowledge seems unrealistic.
However, as we explain in Remark \ref{rem:parameter.dependence.optimal.action}, that is not a major problem and there are several ways of addressing it.

\vspace{0.7em}
\begin{quote}
To complete the heuristics pertaining to the portfolio optimization problem, one has to compute $N_{\mathrm{G}}(r), \sigma^2$ and $c_{\mathrm{H}}$.
Again, for the sake of simplicity we shall ignore $\ell'$ and $\ell''$ at every appearance (keeping in mind that some minor integrability assumptions guarantee that this simplification does not shift the results by too much from the truth).

In this case, $\C[\nabla F(x^\ast,\xi)]=\mathrm{Id}$ and $\nabla^2f(x^\ast)=\mathrm{Id}$, and in particular
\[
\sigma^2=\lambda_{\max}(\mathrm{Id})=1 \ \ \text{and} \ \  N_\mathrm{G}(r)=\frac{\mathrm{trace}(\mathrm{Id})}{r^2}=\frac{d}{r^2}.
\]
Moreover, ignoring $\ell''$ also clearly implies that $c_{\mathrm{H}}=1$.
\end{quote}
\vspace{0.5em}

%\noindent
In Corollary \ref{cor:PO} we specify all the assumptions that are needed to make this heuristic argument hold; but for now let us state a particularly simple case which is of interest in its own right: the exponential portfolio optimization in the Bachelier model.

Recall that in this model $U(\cdot)=-\exp(-\,\cdot)$ is the exponential utility function and $X$ is zero-mean Gaussian.
We assume that the covariance matrix of $X$ is non-degenerate and that both $Y$ and $U(2Y)$ are integrable.
Under these assumptions, we shall see that there exists a unique optimal action $x^\ast\in\mathcal{X}$.
Set
\begin{align*}
	\bar{\sigma}^2
	&:=\E\left[ \exp( -Y-\langle X,x^\ast\rangle)^2 \right]^{\frac{1}{2}},
\end{align*}
and assume that Assumption \ref{ass:norm.equiv} holds true.
While the latter assumption can be verified via some integrability conditions (we shall see this in  context of the general portfolio optimization problem in Lemma \ref{lem:PO.on.norm.equiv}), the obtained bounds may fail to be sharp.
To showcase that Assumption \ref{ass:norm.equiv} can sometimes be easily verified by other means, consider for a moment $Y=\langle X,\tilde{x}\rangle +W$ for some $\tilde{x}\in\mathcal{X}$ and $W$ that is independent of $X$.
Then $x^\ast=\tilde{x}$ and Gaussian norm equivalence (i.e.\ there is an absolute constant $C$ such that  $\E[\langle X,z\rangle^4 ]^{\frac{1}{4}} \leq C \E[\langle X,z\rangle^2]^{\frac{1}{2}}$ for every $z\in\mathbb{R}^d$) together with independence of $X$ and $W$ readily implies that Assumption \ref{ass:norm.equiv} is satisfied with
\[L = \frac{ C^4 \E[\exp(W)^2]}{\E[\exp(W)]^2}.\]

\begin{corollary}[Exponential portfolio optimization]
\label{cor:PO.exp}
	Under the above assumptions, there are constants $c_1,c_2,c_3,c_4>0$ depending only on $L$ and $\E[|Y+\langle X,x^\ast\rangle|]$ such that the following holds.
	For $r\in(0,\min\{ 1,\frac{c_1}{\bar{\sigma}^2}\})$ and
	\[ N\geq c_2 \max\left\{  \frac{ d \bar{\sigma}^2}{r^2} , \, d^{\frac{3}{2}} \, \right\},\]
	there is a procedure $\widehat{x}_N^\ast$ 	such that, with probability at least
	\[ 1-2\exp\left(-c_3 N \min\left\{1, \frac{r^2}{\bar{\sigma}^2  } \right\} \right), \]
	we have that
	\begin{align*}
	\| \widehat{x}_N^\ast - x^\ast \|
	&\leq   r, \\ %\ \ \ {\rm and} \\
	u(\widehat{x}_N^\ast)
	&\geq u(x^\ast) -  c_4 r^2.
	\end{align*}
\end{corollary}

\begin{remark}
The origin of the (somewhat unfamiliar) term $d^{\frac{3}{2}}$ appearing in the minimal sample size in Corollary \ref{cor:PO.exp} is $N_{\mathrm{H},\mathcal{E}}$. However, as our previous heuristics indicate, that term should be of order $d\log(d)$.

As it happens, the source of this difference is the exponential utility function.
Indeed, the heuristic presentation was based on a simplifying assumption: that $\ell'''$ was bounded by 1. That allowed us to conclude that $\E[K(\xi)]$ was of the order $d^3$, resulting in $N_{\mathrm{H},\mathcal{E}}$ of order $d\log(d^3)=3d\log(d)$.
Here, however, $\ell'''$ is the exponential function; the term $\E[K(\xi)]$ is actually of order $\exp(\sqrt{d})$ resulting in $N_{\mathrm{H},\mathcal{E}}$ that is of order $d\log (\exp(\sqrt{d}))= d^{\frac{3}{2}}$.

It should be stressed that although the term $d^{\frac{3}{2}}$ in the minimal sample size of Corollary \ref{cor:PO.exp} could be off by a factor of $\sqrt{d}$, Corollary \ref{cor:PO.exp} is, to the best of our knowledge, the first non-asymptotic estimate for the exponential portfolio optimization problem.
\end{remark}

In some of the examples we will present later, the Hessian additionally satisfies a deterministic lower bound:
	
\begin{assumption}[Deterministic lower bound of the Hessian]
\label{ass:Deterministic.lower.bound.Hessian}
	There is $r_0\in(0,1)$ and $\varepsilon>0$ such that 
	\[\nabla^2 F(x,\xi)\succeq \varepsilon \nabla^2 f(x^\ast)\]
	for all $x\in\mathcal{X}$ with $\|x-x^\ast\|\leq r_0$.
	
	Moreover, $\nabla f(x)=\E[\nabla F(x,\xi)]$ and $\nabla^2 f(x)=\E[\nabla^2 F(x,\xi)]$
	for all $x\in\mathcal{X}$ with $\|x-x^\ast\|<r_0$.
\end{assumption}

When Assumption \ref{ass:Deterministic.lower.bound.Hessian} holds true, the two Assumptions \ref{ass:norm.equiv} and \ref{ass:hessian.midpoints} that were imposed to control the smallest singular value of the Hessian are not needed, and Theorem \ref{thm:main.convex.intro} can be simplified as follows.

\begin{theorem}[Estimation and prediction error, simplified]
\label{thm:main.hessian.deterministic.lower.bound}
	There are constants $c_1,c_2$ depending only on $\varepsilon$ such that the following holds.
	Assume that Assumptions \ref{ass:on.F.diff.convex.etc} and \ref{ass:Deterministic.lower.bound.Hessian} hold, let $r\in (0,r_0)$, and consider
	\[N\geq c_1 N_{\mathrm{G}}(r).\]
	Then there is a procedure $\widehat{x}_N^\ast$ 	such that, with probability at least
	\[ 1-2\exp\left(-c_2 N \min\left\{1, \frac{r^2}{\sigma^{2}  } \right\} \right), \]
	we have that
	\begin{align*}
	\| \widehat{x}_N^\ast - x^\ast \| &\leq  r, \\\
	f(\widehat{x}_N^\ast) &\leq f(x^\ast) + 2c_{\mathrm{H}}  r^2.
	\end{align*}
%	The procedure is described in Section \ref{sec:the.procedure}.
\end{theorem}

%\vspace{1em}
Before presenting the procedure $\widehat{x}_N^\ast$ in detail, let us compare the outcome of Theorem \ref{thm:main.convex.intro} with the current state of the art. Our focus is on recent non-asymptotic results by Oliveira and Thompson \cite{oliveira2017sampleI,oliveira2017sampleII}; a general literature review will be presented in Section \ref{sec:literature}.

For a clearer comparison, let us restate Theorem \ref{thm:main.convex.intro} (pertaining to the prediction error) as follows:

given some fixed confidence level $\delta\in(0,1)$, small $r>0$, and
\[
N\geq c_2 \max\{ d\log(2d) , N_{\mathrm{H},\mathcal{E}} , N_{\mathrm{G}}(r) \},
\]
 the prediction error is bounded by $r^2$ with probability at least $1-\delta$, whenever the sample $N$ satisfies
\begin{align}
\label{eq:literature.our.lower.bound}
N\gtrsim \frac{ \sigma^2 }{r^2}\log\left(\frac{2}{\delta}\right).
\end{align}
In contrast, the main result of Oliveira and Thompson \cite[Theorem 3]{oliveira2017sampleII} regarding general convex stochastic optimization problems is the following. There is a random variable $\widehat{\Sigma}_N$ (which we shall not define here) that satisfies
\[
\widehat{\Sigma}_{N}^2
\gtrsim d  \left( \E[\|\nabla F(x^\ast,\xi)\|^2] +  \frac{1}{N}\sum_{i=1}^N \|\nabla F(x^\ast,\xi_i)\|^2   \right)\]
and in order to guarantee that the prediction error is bounded by $r^2$ with probability at least $1-\delta$ one should have
\begin{equation}
\label{eq:literature.oliveira.lower.bound}
\P\left[ N\geq \frac{ \widehat{\Sigma}_{N}^2 }{r^2}\log\left(\frac{2}{\delta}\right) \right]
\geq 1-\delta.
\end{equation}

There are two \emph{major} differences between these two results. The first one, which should not come as a great surprise following the discussion in Section \ref{sec:non.gaussian}, is that $\widehat{\Sigma}_{N}^2 $ does not concentrate around its mean with high probability in heavy tailed situations.
Thus, \eqref{eq:literature.oliveira.lower.bound} forces $N$ to grow like $(\frac{1}{\delta})^{\frac{1}{p}}$ for some power $p>1$ depending on the integrability of $\|\nabla F(x^\ast,\xi)\|$.
For small $\delta$ (i.e.\ if one is interested in high confidence) this is in stark contrast to the order $\log( \frac{1}{\delta})$ in \eqref{eq:literature.our.lower.bound}.
The second difference is the dependence of $\widehat{\Sigma}_{N}^2$ on the dimension $d$.
Indeed, even if we neglect the integrability issues and replace $\widehat{\Sigma}_{N}^2$ by its mean, what we end up with is not the correct variance parameter (which appears in the central limit theorem).
For instance, if $\|\cdot\|$ is the Euclidean norm, then
\[ \widehat{\Sigma}_{N}^2
\gtrsim d  \mathrm{trace}( \C[\nabla F(x^\ast,\xi)]), \]
which is at least $d$ times (possibly even $d^2$ times) larger than the true variance parameter $\sigma^2=\lambda_{\max}( \C[\nabla F(x^\ast,\xi)])$.
In high-dimensional problems such as the portfolio optimization problem, where $d$ is the number of stocks and is likely to be a large number, this difference is significant. As a result, even for moderate confidence levels $\delta$, the required sample size jumps up by a factor of $d$ (or even $d^2$) from what we would expect.

\subsection{The procedure}
\label{sec:the.procedure}

As we already explained previously, to have any hope of optimal performance, the procedure $\widehat{x}_N^\ast$ \emph{cannot} be the sample average approximation.
Instead, $\widehat{x}_N^\ast$ will be determined through median-of-mean tournaments conducted between all $x\in\mathcal{X}$, following the method introduced in \cite{lugosi2016risk}.

The first phase of the procedure returns a set of candidates, each with a small estimation error.

\vspace{0.5em}
\begin{enumerate}[(Step 1)]
\item
For some (small) tuning parameter $\theta$ to be specified later, set
\[n:=\theta N \min\left\{1,\frac{r^2}{\sigma^2}\right\}
\quad\text{and}\quad
m:=\frac{N}{n}.\]
Without loss of generality assume that $m$ and $n$ are integers. Partition
\[\{1,\dots,N\}=\bigcup_{j=1}^n I_j\]
into $n$ disjoint blocks $I_j$, each of equal cardinality $|I_j|=m$.
\item
For every $x\in\mathcal{X}$, compute the empirical mean of $F(x,\xi)$ on the $j$-th block
\[ \widehat{f}_{I_j}(x):=\frac{1}{m}\sum_{i\in I_j} F(x,\xi_i).\]
We then say that $x\in\mathcal{X}$ \emph{defeats} $y\in\mathcal{X}$ \emph{on the $j$-th block} if $\widehat{f}_{I_j}(x)<\widehat{f}_{I_j}(y)$, and that  $x$ \emph{wins the match against} $y$ if
\[ \widehat{f}_{I_j}(x)<\widehat{f}_{I_j}(y)
\quad\text{on more than } \frac{n}{2} \text{ blocks} \ j,\]
i.e.\ if $x$ defeats $y$ on a majority of the blocks.
\end{enumerate}
\vspace{0.5em}

\noindent
Denote by $\tilde{\mathcal{X}}_N^\ast\subseteq\mathcal{X}$ the set of \emph{champions}, i.e.\
\[ \tilde{\mathcal{X}}_N^\ast:=\left\{ x\in\mathcal{X} :
\begin{array}{l}
x \text{ wins the match against every}\\
y\in\mathcal{X} \text{ that satisfies } \|x-y\|\geq r
\end{array}\right\}. \]
The following proposition shows that elements in $\tilde{\mathcal{X}}_N^\ast$ satisfy the part of Theorem \ref{thm:main.convex.intro} pertaining to the estimation error.

\begin{proposition}[Estimation error]
\label{prop:estimation.error}
	In the setting of Theorem \ref{thm:main.convex.intro},
	with probability at least $1-2\exp(-c_{3}n)$, we have that $x^\ast\in\tilde{\mathcal{X}}_N^\ast$ and every $x\in \tilde{\mathcal{X}}_N^\ast$ satisfies $\|x-x^\ast\|\leq r$.
\end{proposition}

The rough idea of the proof of Proposition \ref{prop:estimation.error}---explained in details in Section \ref{sec:proof.main}--- is to study the Taylor expansion
\begin{align*}
&\widehat{f}_{I_j}(x) - \widehat{f}_{I_j}(x^\ast)\\
&= \frac{1}{m}\sum_{i\in I_j} \langle \nabla F(x^\ast ,\xi_i), x-x^\ast \rangle +\frac{1}{2}\frac{1}{m}\sum_{i\in I_j} \langle \nabla^2 F(z_i ,\xi_i) (x-x^\ast), x-x^\ast \rangle,
\end{align*}
where $z_i$ are mid-points between $x$ and $x^\ast$.
The expectation of the term involving the gradient is non-negative (due to optimality of $x^\ast$) and we shall show that on the sample it is not too negative; say larger than $-\frac{r^2}{16}$ for all points $x$ whose distance to $x^\ast$ is $r$.
We will further show that the term involving the Hessian is at least $\frac{r^2}{8}$ for such points $x$, and as a results, $\widehat{f}_{I_j}(x) - \widehat{f}_{I_j}(x^\ast)\geq \frac{r^2}{16} >0$.

Naturally, the heart of the matter is to show that these estimates hold uniformly---on the same event of high probability.

As we already explained, if $\mathcal{X}=\mathbb{R}^d$ or, more generally, if $x^\ast$ lies in the interior of $\mathcal{X}$,  the first order condition for optimality immediately implies that $f(x) \leq f(x^\ast) + c_{\mathrm{H}} r^2$ for every $x\in\mathcal{X}$ with $\|x-x^\ast\|\leq r$.
In particular, in that case, it follows from Proposition \ref{prop:estimation.error} that any choice $\widehat{x}_N^\ast\in\tilde{\mathcal{X}}_N^\ast$ satisfies the assertion of Theorem \ref{thm:main.convex.intro}.

However, in general, if one wishes to find $x\in \tilde{\mathcal{X}}_N^\ast$ with a small prediction error, one requires an additional procedure, which we describe now.

To simplify notation, assume without loss of generality that the set $\tilde{\mathcal{X}}_N^\ast$ has already been determined, and that we can run an additional second procedure, for which we are given a new (independent) sample $F(\cdot,\xi_i)_{i=N+1}^{2N}$.

Again partition $\{N+1,\dots,2N\}$ into $n$ disjoint blocks $I_j'$ of cardinality $|I_j'|=m$ with the same $n$ and $m$, and denote by $\widehat{f}_{I_j'}(\cdot)$ the empirical mean on the block $I_j'$.

\vspace{0.5em}
\begin{enumerate}[(Step 1)]
	\setcounter{enumi}{2}
	\item
	We say that $x\in \mathcal{X}$ \emph{wins its home match} against $y\in\mathcal{X}$ if
	\[ \widehat{f}_{I_j'}(x) \leq \widehat{f}_{I_j'}(y) + \frac{ c_{\mathrm{H}} r^2}{4}
	\quad\text{ 	on more than } \frac{n}{2} \text{ blocks } j.\] %\quad\text{where } 	f_{I_j}(\cdot) :=\frac{1}{m} \sum_{i\in I_j} F(\cdot,\xi_i) \]
%	on more than $\frac{n}{2}$ blocks $j$.
	\end{enumerate}
\vspace{0.5em}

\noindent
Denote by $\widehat{\mathcal{X}}_N^\ast$ the \emph{winners}, i.e.,
\[ \widehat{\mathcal{X}}_N^\ast:=
\left\{x\in \tilde{\mathcal{X}}_N^\ast :
\begin{array}{l}
x \text{ wins its home match }\\
\text{against every } y\in\tilde{\mathcal{X}}_N ^\ast
\end{array} \right\}.   \]
In light of Proposition \ref{prop:estimation.error}, the crucial advantage here is that, with high probability, matches are only carried out between competitors that are close to $x^\ast$.
The following proposition shows that any $\widehat{x}_N^\ast\in\widehat{\mathcal{X}}_N^\ast$ satisfies the requirements in Theorem \ref{thm:main.convex.intro}:

\begin{proposition}[Prediction error]
\label{prop:prediction.error}
	In the setting of Theorem \ref{thm:main.convex.intro},
	with probability at least $1-2\exp(-c_{3}n)$, we have that $x^\ast\in\widehat{\mathcal{X}}_N^\ast$ and every $x\in \widehat{\mathcal{X}}_N^\ast$ satisfies $f(x)<f(x^\ast) + 2c_{\mathrm{H}}r^2$.
\end{proposition}

\subsection{Related literature}
\label{sec:literature}

Stochastic optimization and the statistical properties of the sample average approximation method have been studied intensively for several decades;  it is therefore impossible to mention every single contribution.
Instead, we refer to Kim, Pasupathy, and Henderson \cite{kim2015guide}, Shapiro, Dentcheva, and Ruszczy{\'n}ski \cite{shapiro2014lectures}, Kleywegt, Shapiro, Homem-de-Mello \cite{kleywegt2002sample}, Shapiro \cite{shapiro2003monte},
 or Homem-de-Mello and Bayraksan \cite{homem2014monte}.
As we already mentioned, the statistical analysis in these works is always of asymptotic nature.
Let us also refer to Banholzer, Fliege, and Werner \cite{banholzer2017almost} for a recent study of asymptotic almost sure convergence rates and an up-to date review on the asymptotic convergence analysis, to Guigues, Kratschmer, and Shapiro \cite{guigues2018statistical} who provide a central limit theorem for general risk-averse problems based on coherent risk measures,  and to Bertsimas, Gupta, and Kallus \cite{bertsimas2018robust} who raise concerns that most statistical results for sample average approximation are of asymptotic nature.

Shifting to a non-asymptotic analysis of the performance of the SAA, the available literature gets considerably less diverse.
An early reference here is Pflug \cite{pflug1999stochastic,pflug2003stochastic}, who relies on Talagrand's deviation inequality for the supremum of Gaussian processes (which is obviously suitable only in very special, light-tailed scenarios).
Similar methods were adapted by R\"omisch \cite{romisch2003stability} and Vogel \cite{vogel2008universal} and (for the error of the value---i.e.\ $\min_{x\in\mathcal{X}} f(x)$) by Guigues, Juditsky, and Nemirovski in \cite{guigues2017non}.
Confidence intervals derived from optimization algorithms such as Stochastic Mirror Descent were analyzed by Lan, Nemirovski, and  Shapiro \cite{lan2012validation} for risk-neutral problems and by Guigues \cite{guigues2017multistep} for risk-averse
problems.
The two recent papers by Oliveira and Thompson \cite{oliveira2017sampleI,oliveira2017sampleII} which have been mentioned at the end of Section \ref{sec:main} contain results that are closest to ours.

The portfolio optimization problem is often listed as a prime example of a stochastic optimization problem, see e.g.\ \cite[Section 1.4]{shapiro2014lectures}.
As such, and with the exception of perhaps more applied studies like \cite{ghodrati2014monte,wang2018sparse,xu2012monte}, existing estimates on the problem have been derived as applications of general results---much like in our work.

There are other (optimization) problems in mathematical finance that have been analyzed via sampling, such as the estimation of risk measures.
We refer the reader to \ Weber \cite{weber2007distribution} (which relies on large deviation methods) for an early reference and to \cite{bartl2020non} for a more recent review. We believe that our methods can be applied to those type of problems as well, and defer it to future work.

We should mention that the statistical analysis of the SAA has close ties to the statistical analysis of problems in high-dimensional statistics, such as linear regression. But despite obvious similarities, the two fields have grown adrift. The current focus in statistical learning literature is on non-asymptotic statements that hold under increasingly relaxed assumptions.
Among the outcomes of this approach were  \cite{lugosi2016risk,mendelson2017aggregation,mendelson2019unrestricted}---alternatives to the sample average approximation in statistical learning problems that recover the Gaussian rates in completely heavy tailed situations. 
We believe that pursuing the same direction in the context of stochastic optimization would lead to intriguing questions. 
Indeed, there are many variants of (convex) stochastic optimization problems that are not covered in this article. 
For example, among these questions is \emph{chance constrained stochastic optimization}, where the optimization takes place only over actions $x\in\mathcal{X}$ satisfying probabilistic constraints of the form  $\E[G(x,\xi)]\leq 0$.
We are confident that our methods can be adapted to these settings as well, though we shall leave this for future work.

\section{Applications}
\label{sec:applications}

Before continuing with the portfolio optimization problem (in its general form), we present four applications of Theorem \ref{thm:main.convex.intro}.

\subsection{Multivariate mean estimation}

As a first application of Theorem \ref{thm:main.convex.intro}, let us return to the problem of mean estimation---this time for a square integrable $d$-dimensional random vector $\xi$. Mean estimation is a stochastic optimization problem because
\[\E[\xi] = \mathrm{argmin} \left\{ \E[ \|x-\xi\|_2^2 ] : x\in\mathbb{R}^d  \right\}. \]
This suggests that one should set
\[ F(x,\xi):= \frac{1}{2}\| x- \xi \|_2^2
\quad\text{for }x\in\mathcal{X}:=\mathbb{R}^d\]
so that $x^\ast=\E[\xi]$.
(The factor $\frac{1}{2}$ has only a normalizing purpose to make the computations below clearer.)
As an intimidate consequence of Theorem \ref{thm:main.hessian.deterministic.lower.bound} we obtain the following.

\begin{corollary}[Multivariate mean estimation]
\label{cor:mean.Euclidean}
	There are absolute constants $C_1,C_2$  such that the following holds.
	Let $r>0$ and let
	\[N\geq C_1 \frac{\mathop{\mathrm{trace}}(\C[\xi]) }{r^2} .\]
	Then there is a procedure $\widehat{x}_N^\ast$ %depending on $(F(\cdot,\xi_i))_{i=1,\dots,N}$
	such that, with probability at least
	\[ 1-2\exp\left(-C_2 N \min\left\{1, \frac{r^2}{ \lambda_{\mathrm{max}}(\C[\xi])  } \right\} \right), \]
	we have that
	\begin{align*}
	\left\| \widehat{x}_N^\ast - \E[\xi] \right\|_2 &\leq  r.
	\end{align*}
\end{corollary}

Let us once again mention that the question of finding an  estimator of the mean of a heavy tailed random vector that exhibits Gaussian rates remained open until very recently: it was settled in \cite[Theorem 1]{lugosi2019sub}; see also \cite{lugosi2019mean} for a recent survey.
Corollary \ref{cor:mean.Euclidean} recovers \cite[Theorem 1]{lugosi2019sub}.

\begin{proof}[Proof of Corollary \ref{cor:mean.Euclidean}]
	For every $x\in\mathcal{X}$, we have that
	\begin{align*}
	\nabla F(x,\xi)
	&=x-\xi, \\
	\nabla^2 F(x,\xi)
	&= \mathrm{Id}.
	\end{align*}
	In particular $\|\cdot\|=\|\cdot\|_2$ and Assumption \ref{ass:on.F.diff.convex.etc} is clearly satisfied.
	Actually, as the Hessian is deterministic and independent of the action $x$, we are in the setting  of Theorem \ref{thm:main.hessian.deterministic.lower.bound} and it remains to compute the parameters $N_{\mathrm{G}}(r)$ and $\sigma^2$.
	To that end, it suffices to note that $\nabla^2 f(x^\ast)^{-1}=\mathrm{Id}$ and $\C[ \nabla F(x^\ast,\xi)]=\C[ \xi]$; hence
	\begin{align*}
	N_{\mathrm{G}}(r)&= \frac{1}{r^2} \mathop{\mathrm{trace}}(\C[\xi]),\\
	\sigma^2&= \lambda_{\mathrm{max}}(\C[\xi]).
	\end{align*}	
	The proof therefore follows from Theorem \ref{thm:main.hessian.deterministic.lower.bound}.
\end{proof}

\subsection{Linear regression}

Linear regression is one of the fundamental problems studied in statistics.
Given a one-dimensional random variable $Y$ and a $d$-dimensional random vector $X$, the task is to find the best possible forecast of $Y$ based on linear combinations of $X$.
To be more precise, one seeks the minimizer of $x\mapsto  \E[( \langle X, x \rangle - Y)^2]$ over $x\in\mathbb{R}^d$ (or a subset thereof).
This problem clearly falls within the scope of this article by considering
\[ F(x,\xi):=\frac{1}{2}( \langle X,x \rangle - Y)^2
\quad\text{where } \xi:=(X,Y) \]
with $\mathcal{X}\subseteq \mathbb{R}^d$.
(The purpose of the factor $\frac{1}{2}$ is again only for convenience.)
In order to lighten notation, we shall impose a standard assumption on $X$: that it is centred and isotropic. 
The latter means that its covariance matrix is the identity.

\begin{assumption}%[Linear Regression]
\label{ass:lin.reg.norm.reg}
	The set $\mathcal{X}\subseteq\mathbb{R}^d$ is closed and convex, $X$ is centred and isotropic, and there is a constant $L_X$ such that
	\begin{align}
	\label{eq:lin.reg.norm.equiv.1}
		\E[\langle X,x\rangle^4]^{ \frac{1}{4} }
	&\leq L_X \E[\langle X,x\rangle^2]^{ \frac{1}{2} }
	\end{align}
	for every $x\in\mathbb{R}^d$.
	Further,
	\[\bar{\sigma}^2:=\E[(\langle X,x^\ast\rangle - Y)^4]^{\frac{1}{2}}\]
	is finite.
\end{assumption}

The assumption \eqref{eq:lin.reg.norm.equiv.1}, which means that the $L_4$ and $L_2$ norms of linear forms of $X$ are equivalent, is a typical assumption made in high-dimensional statistics and it is not too restrictive.

\begin{corollary}[Linear Regression]
\label{cor:lin.reg}
	If Assumption \ref{ass:lin.reg.norm.reg} is satisfied, there are constants $c_1$ and $c_2$ depending only on $L_X$ such that the following holds.
	Let $r>0$ and
	\[N\geq c_1 \max\left\{ d \log(2d) \, , \, \frac{d   \bar{\sigma}^2}{r^2} \right\} .\]
	Then there is a procedure $\widehat{x}_N^\ast$
	such that, with probability at least
	\[ 1-2\exp\left(-c_2 N \min\left\{1, \frac{r^2}{\bar{\sigma}^2  } \right\} \right), \]
	we have that
	\begin{align*}
	\| \widehat{x}_N^\ast - x^\ast \|_2
	&\leq  r,\\
	\E_{X,Y}[ (\langle X, \widehat{x}_N^\ast\rangle - Y)^2]
	& \leq \E[ (\langle X, x^\ast\rangle - Y)^2] + 2r^2.
	\end{align*}
\end{corollary}

Here, $\E_{X,Y}[\cdot]$ denotes the expectation taken only over $X$ and $Y$ (and, of course, not the sample $(X_i,Y_i)_{i=1}^N$ used for the computation of $\widehat{x}_N^\ast$).

Compared to the benchmark result on linear regression in a heavy tailed scenario \cite{lugosi2016risk}, Corollary \ref{cor:lin.reg} has an additional $\log(2d)$ term in the estimate on the sample size.
This term is merely an artifact of the generality of our main result. Its origin lies in the matrix-Bernstein inequality, which we use in the process of bounding the smallest singular value of a general random matrix ensemble (namely $\nabla^2 F(x^\ast,\xi)$); that extra factor is not needed in this example and can be easily removed.

\begin{proof}[Proof of Corollary \ref{cor:lin.reg}]
	For every $x\in\mathcal{X}$, we have that
	\begin{align*}
 	\nabla F(x,\xi)
 	&=( \langle X,x \rangle - Y)X , \\ %\text{ and}\\
	\nabla^2 F(x,\xi)
	&= X\otimes X.
	\end{align*}
	As $X$ is isotropic, this implies $\nabla^2 f(x^\ast)=\mathrm{Id}$ and in particular $\|\cdot\| =\|\cdot\|_2$.
	Also, as $(\langle X,x^\ast\rangle -Y)$ and $X$ are both in $L_4$ by assumption, the Cauchy-Schwartz  inequality shows that $\nabla F(x^\ast,\xi)$ is square integrable.
	In particular, Assumption \ref{ass:on.F.diff.convex.etc} is satisfied.
	Employing the norm equivalence of $X$ \eqref{eq:lin.reg.norm.equiv.1}, we obtain
	\begin{align*}
	\E[\langle \nabla^2 F(x^\ast,\xi) z,z\rangle^2]
	&=  \E[\langle X,z\rangle^4] \\
	&\leq  L_X^4 \|z\|^4
	\end{align*}
	for every $z\in\mathbb{R}^d$; thus Assumption \ref{ass:norm.equiv} is satisfied with $L=L_X^4$.
	Finally, as the Hessian is independent of the action $x$, Assumption \ref{ass:hessian.midpoints} is clearly satisfied with $\alpha=1$ , $K\equiv 0$ and an arbitrary $r_0$; in particular, $N_{\mathrm{H},\mathcal{E}}\leq d$.

	It remains to compute the parameters $N_{\mathrm{G}}(r),\sigma^2$, and $c_{\mathrm{H}}$. Using once again that the Hessian is independent of the action $x$, it is evident that $c_{\mathrm{H}}=1$. 
	Turning to $\sigma^2$ and $N_{\mathrm{G}}(r)$, recall that $\nabla^2 f(x^\ast)=\mathrm{Id}$, and let us estimate the largest eigenvalue and the trace of $\C[\nabla F(x^\ast,\xi)]$.

	For every $z\in\mathbb{R}^d$, the Cauchy-Schwartz inequality together with \eqref{eq:lin.reg.norm.equiv.1} imply
	\begin{align}
	\label{eq:lin.reg.cov}
	\begin{split}
	\langle \C [\nabla F(x^\ast,\xi)] z,z\rangle
	&=\V [(\langle X,x^\ast\rangle-Y)\langle X,z\rangle]  \\
	&\leq \E\left[ \left( (\langle X,x^\ast\rangle-Y)\langle X,z\rangle \right)^2 \right] \\
	&\leq \E[ (\langle X,x^\ast\rangle-Y)^4]^\frac{1}{2}\E[\langle X,z\rangle^4]^\frac{1}{2}\\
	&\leq \bar{\sigma}^2 L_X^2 \|z\|_2^2.
	\end{split}
	\end{align}
	This clearly implies $\sigma^2\leq  L_X^2\bar{\sigma}^2$ and, taking the standard Euclidean basis $z=e_i$ in  \eqref{eq:lin.reg.cov}, we conclude $\C[ \nabla F(x^\ast,\xi)]_{ii}\leq \bar{\sigma}^2 L_X^2$, hence
	\begin{align*}
	N_{\mathrm{G}}(r)
	&=\frac{1}{r^2} \sum_{i=1}^d 	\C[\nabla F(x^\ast,\xi)]_{ii}
	\leq \frac{ \bar{\sigma}^2 L_X^2 d  }{r^2} .
	\end{align*}
	The proof now follows from Theorem \ref{thm:main.convex.intro}.
\end{proof}

\subsection{Ridge regression}

A popular modification of linear regression is \emph{ridge regression} (also known as \emph{weight decay regression} and \emph{$\ell_2$-regularized regression} in the machine learning community). The idea is that one penalizes a large Euclidean norm of $x$ by setting
\[F(x,\xi):= (\langle X,x\rangle - Y)^2 + \lambda \|x\|_2^2
\]
where $\xi:=(X,Y)$, $x\in\mathcal{X}\subseteq\mathbb{R}^d$, and $\lambda>0$ is a tradeoff parameter.
 The aim of this sort of penalization is to counteract  \emph{over-fitting}.

In contrast to the estimate we obtain in linear regression, the factor $d\log(2d)$ in the minimal sample is not needed here:

\begin{corollary}
\label{cor:lin.reg.ridge}
	If Assumption \ref{ass:lin.reg.norm.reg} is satisfied, there are constants $c_1$ and $c_2$ depending only on $L_X$ and $\lambda$ such that the following holds.
	Let $r>0$ and set
	\[N\geq c_1 \frac{d   \bar{\sigma}^2}{r^2} .\]
	Then there is a procedure $\widehat{x}_N^\ast$ 	such that, with probability at least
	\[ 1-2\exp\left(-c_2 N \min\left\{1, \frac{r^2}{\bar{\sigma}^2  } \right\} \right), \]
	satisfies that
	\begin{align*}
	\| \widehat{x}_N^\ast - x^\ast \|_2
	&\leq  r,\\
	\E_{X,Y}[ (\langle X, \widehat{x}_N^\ast\rangle - Y)^2  ] + \lambda \|\widehat{x}_N^\ast\|_2^2
	& \leq \E[ (\langle X, x^\ast\rangle - Y)^2] + \lambda\|x^\ast\|_2^2 + 2r^2.
	\end{align*}
\end{corollary}

\begin{proof}
	For every $x\in\mathcal{X}$, we have that
	\begin{align*}
	\nabla^2 F(x,\xi)
	&=2X\otimes X + 2\lambda \mathrm{Id},\\
%	\quad\text{and}\quad
	\nabla^2f(x)
	&=(2+2\lambda) \mathrm{Id};
	\end{align*}
	hence we are in the setting of Theorem \ref{thm:main.hessian.deterministic.lower.bound}.
	All that remains is to estimate the parameters $N_{\mathrm{G}}(r)$ and $\sigma^2$, and this can be done exactly as in the proof of Corollary \ref{cor:lin.reg}.
\end{proof}

\subsection{Portfolio optimization}
\label{sec:PO}
As a final example, we address the portfolio optimization problem in its general form, that is, we apply Theorem \ref{thm:main.convex.intro} to the problem
\[ F(x,\xi):=\ell( V_x) \quad\text{where }
\begin{array}{l}
V_x:=-Y-\langle X,x\rangle,\\
\hspace{0.6em} \xi:=(X,Y),
\end{array}\]
for $x\in\mathcal{X}\subseteq\mathbb{R}^d$ that is closed and convex.
Moreover, $\ell\colon\mathbb{R}\to\mathbb{R}$ is strictly convex, increasing, bounded from below, and we assume that it is three times continuously differentiable.

We shall impose the following two assumptions on the zero-mean random vector $X$.
Firstly, assume that $X$ satisfies a (directional) $L_6$-$L_2$ norm equivalence, i.e.\ there is a constant $L_X$ such that
\begin{align}
\label{ass:portfolio.norm.equiv}
\E[\langle X,z\rangle^6]^{\frac{1}{6}}
\leq L_X  \E[\langle X,z\rangle^2]^{\frac{1}{2}}
<\infty
\end{align}
for all $z\in\mathbb{R}^d$.
In addition, we assume that the following no-arbitrage condition holds
\begin{align}
\label{eq:PO.NA}
\P[\langle X,z\rangle < 0]>0 \quad\text{for all }z\in\mathbb{R}^d\setminus\{0\}.
\end{align}

\begin{remark}
\label{rem:PO.NA}
	The classical no-arbitrage condition used in mathematical finance reads as follows: for every $z\in\mathbb{R}^d$, we have that
	\[ \P[\langle X,z\rangle< 0]=0
	\text{  implies that }
	\P[\langle X,z\rangle>0]=0;\]
	i.e.\ it is not possible to make profit without taking any risk.
	Under this condition, it is well-known that one can decompose $\mathbb{R}^d=V \oplus V^{\perp}$ into an orthogonal sum such that $\P[\langle X,v\rangle < 0]>0$ for all $v\in V\setminus\{0\}$ and $\P[\langle X,w\rangle=0]=1$ for all $w\in V^{\perp}$, see e.g.\ \cite[Section 1.3]{follmer2011stochastic}.
	In particular, replacing $\mathcal{X}$ by $\mathcal{X}\cap V$ viewed as a subset of $\mathbb{R}^{\mathrm{dim}(V)}$ does not affect the outcome of the portfolio optimization problem but guarantees that \eqref{eq:PO.NA} holds.
\end{remark}

Moreover, we shall assume that part (c) of Assumption \ref{ass:on.F.diff.convex.etc}, pertaining the integrability of $F,\nabla F, \nabla^2 F$, is satisfied.
By H\"older's inequality and \eqref{ass:portfolio.norm.equiv}, this is the case if, for example,  $\ell(V_x), \ell'(V_x)^4, \ell''(V_x)^2$ are integrable for every $x\in\mathcal{X}$.
Then $f$ is real-valued, and standard arguments building on the no-arbitrage condition, show that $f$ is strictly convex and coercive.
In particular, a unique optimal action $x^\ast\in\mathcal{X}$ exists; see e.g.\ \cite[Section 3.1]{follmer2011stochastic}.

Finally, denoting by $\mathcal{B}_1^\ast$ the ball of radius 1 w.r.t.\ the norm $\|\cdot\|$ centered at $x^\ast$ and restricted to $\mathcal{X}$, we assume that
\begin{align*}
	\bar{\sigma}^2
	&:=\E\left[ ( \tfrac{\ell'( V_{x^\ast} )^2}{\ell''(V_{x^\ast})} )^2 \right]^{\frac{1}{2}},\\
	v_1
	&:=\E[ |V_{x^\ast} |], \\
	v_2
	&:=\E[\ell''( V_{x^\ast} )^6]^{\frac{1}{6}},\\
	v_{K}
	&:=\E\left[ \sup\nolimits_{x\in\mathcal{B}_1^\ast} \ell'''( V_{x} )^2 \right]^{\frac{1}{2}},\\
	v_{\mathcal{E}_{\mathrm{H}}}
	&:=\sup\nolimits_{x\in\mathcal{B}_1^\ast}  \E\left[ \sup\nolimits_{t\in[0,1]} \ell'''( V_{x^\ast+t(x-x^\ast)})^2 \right]^{\frac{1}{2}}
\end{align*}
are all finite.

\begin{remark}
\label{rem:PO.ell'''}
	If, for instance, $\ell'''$ is non-negative and increasing, the terms $v_{K}$ and $v_{\mathcal{E}_\mathrm{H}}$ can be simplified as
	\begin{align*}
	\sup_{x\in\mathcal{B}_1^\ast} \ell'''( V_{x} )
	&=  \ell'''( V_{x^\ast} + \|X\|_\ast ),\\
	\sup_{t\in[0,1]} \ell'''( V_{x^\ast+t(x-x^\ast)})
	&\leq  \ell'''( V_{x^\ast} + |\langle X, x-x^\ast\rangle|)
	\end{align*}
	where $\|\cdot\|_\ast:=\sup_{x\in\mathbb{R}^d \text{ s.t.\ } \|x\|\leq 1} \langle x,\cdot\rangle$ denotes the dual norm of $\|\cdot\|$.
	(Note that if $\|\cdot\|$ is (equivalent to) the Euclidean norm, then its dual norm is (equivalent to) the Euclidean norm too.)
\end{remark}
	
Under these assumptions, we obtain the following:
	
\begin{corollary}[Portfolio optimization]
\label{cor:PO}
	There are constants $c_1, c_2, c_3,c_4$ depending only on $L_X, v_1, v_2$ such that the following holds.
	For  $r\in(0, \min\{1,\frac{c_1}{v_{\mathcal{E}_\mathrm{H}}}\})$ and
	\[N\geq c_1 \max\left\{ \frac{ d \bar{\sigma}^2}{r^2} , \, d \log (d ( v_K+2)) \, \right\}, \]
	there is a procedure $\widehat{x}_N^\ast$ such that, with probability at least
	\[ 1-2\exp\left(-c_2 N \min\left\{1, \frac{r^2}{\bar{\sigma}^2  } \right\} \right), \]
	we have that
	\begin{align*}
	\| \widehat{x}_N^\ast - x^\ast  \| &\leq   r, \\
	u(\widehat{x}_N^\ast) &\geq u(x^\ast)- c_4r^2.
	\end{align*}
\end{corollary}

We postpone the proof of Corollary \ref{cor:PO}; it will be presented in Section \ref{sec:PO.proof}, where we also show how to recover the estimate in the exponential portfolio optimization problem (i.e.\ Corollary \ref{cor:PO.exp}).

\section{On the smallest singular value of general random matrix ensembles}
\label{sec:mom.singular.value}

In the course of the analysis needed in the proof of Theorem \ref{thm:main.convex.intro}, a crucial ingredient is that sampling can exhibit that $\langle \nabla^2 F(x^\ast,\xi)(x-x^\ast),x-x^\ast\rangle$ is sufficiently large.
Put differently, it is essential to derive a suitable lower bound on the smallest singular value of the empirical random matrix of $\nabla^2 F(x^\ast,\xi)$.

Results of this type have been studied extensively \cite{adamczak2010quantitative,koltchinskii2015bounding,rudelson1999random,srivastava2013covariance,tikhomirov2018sample,yaskov2015sharp}, however (to the best of our knowledge) the focus was on random matrix ensembles that had some additional special structure, like iid rows/columns. Unfortunately such special structure need not exist in our setting.

\vspace{0.5em}
In this section, let $A$ be a real, square integrable, positive-semidefinite $(d\times d)$-random matrix and let $(A_i)_{i\geq 1}$ be independent copies of $A$.
(in the context of this article, the case that interests us is $A=\nabla^2 F(x^\ast,\xi)$.)
Denote its expectation by $\mathbb{A}:=\E[A]$, the semi-norm endowed by $\mathbb{A}$ is
\[ \|x\|:= \langle \mathbb{A} x,x\rangle^{\frac{1}{2}}
= \E[ \langle A x,x\rangle]^{\frac{1}{2}} \]
and the corresponding unit sphere is
\[S:=\{ x\in \mathbb{R}^d : \|x\|=1\}.\]

To simplify the presentation, assume that $\|\cdot\|$ is a true norm, i.e.\ $\|x\|=0$ implies that $x=0$.
Next, for any $(d\times d)$-matrix $B$, its operator norm is given by
\[\|B\|_{\mathrm{op}}
:= \max_{x,y\in S} \langle Bx, y\rangle.\]

Before stating the main result of this section (Theorem \ref{thm:singular.value}) let us begin with one of its outcomes.

\begin{corollary}
\label{cor:smallest.singular.value.norm.equiv}
	Assume that there is a constant $L>0$ such that
	\begin{align}	
	\label{eq:cor.smallest.singular.ass.norm.equiv}
	 \E[\langle Ax,x\rangle^2]^\frac{1}{4} \leq L
	\quad\text{for all }x\in S.
	\end{align}
	Then there are constants $c_1$ and $c_2$ that depend only on $L$ such that the following holds.
	Let $\gamma\in(0,1)$ and assume that
	\[ N \geq c_1 \frac{d }{\gamma^2} \log\left(\frac{2 d }{\gamma  } \right) .\]
	Then, with probability at least
	\[1-2\exp\left(- c_2 N \gamma^2 \right),\]
	we have that
	\[ \lambda_{\min}\left( \frac{1}{N} \sum_{i=1}^N A_i \right)
	\geq (1-\gamma) \lambda_{\min}(\E[A]). \]
\end{corollary}
\noindent
(Here, of course, $\lambda_{\min}$ denotes  the smallest singular value.)

\vspace{0.5em}
As the results of this section will be used in the proof of Theorem \ref{thm:main.convex.intro} for the construction of a median-of-means tournament, we also need a median-of-means version of Corollary \ref{cor:smallest.singular.value.norm.equiv}.
To that end, as before, let $N=nm$ for two integers $n$ and $m$, and consider a partition of $\{1,\dots,N\}$ into $n$ disjoint blocks $I_j$, each one of cardinality $|I_j|=m$.
Throughout this article, we make the \emph{notational convention} that the letter $j$ always refers blocks, i.e.\ $j$ is always an element of $\{1,\dots,n\}$.
In addition, $0<C,C_0,C_1,\dots$ denote absolute constants independent of all parameters.
They are allowed to change their values from line to line.

We often encounter the so-called \emph{Rademacher random variables} $(\varepsilon_i)_{i\geq 1}$, which are independent, symmetric random signs (i.e.\ $\P[\varepsilon_i=\pm 1]=\frac{1}{2}$) that are also independent of all the other random variables that appear in the analysis; in particular, they are independent of $(A_i)_{i=1}^N$.

We have already seen in the introduction that the smallest empirical singular value of a random matrix cannot be bounded (even with constant probability) without imposing some of assumption.
While an integrability assumption as in \eqref{eq:cor.smallest.singular.ass.norm.equiv} will do the job, the following definition from \cite{mendelson2017extending} contains the essence of what is actually needed.

\begin{definition}[Stable lower bound]
\label{def:stable.lower.bound}
	A set $H\subseteq L_2$ of real-valued functions satisfies a \emph{stable lower bound} with parameters $(m,\gamma,l,k)$ if the following holds:
	for every $h\in H$ and independent copies $(h_i)_{i\geq 1}$ of $h$, with probability at least $1-2\exp(-k)$, for all $J\subseteq \{1,\dots,m\}$ with $|J|\leq l$, we have that
\[	\frac{1}{m}\sum_{i\in\{1,\dots,m\}\setminus J} h_i^2 \geq (1-\gamma) \E[h^2]. \]
%	\[ \begin{array}{l}
%	\text{for all } J\subseteq \{1,\dots,m\} \\
%	\text{with } |J|\leq l \text{ it holds} \qquad
%	\end{array}
%	\frac{1}{m}\sum_{i\in\{1,\dots,m\}\setminus J} h_i^2 \geq (1-\gamma) \E[h^2]. \]
	We say that a (symmetric, positive semi-definite, $d\times d$) random matrix $A$ satisfies a stable lower bound with parameters $(m,\gamma,l,k)$ if $H:=\{ \langle Ax,x\rangle^{\frac{1}{2}}: x\in S \}$ does (with the same parameters).
\end{definition}

The stable lower bound can be seen as an extension of the \emph{small ball property}.
Recall that a set $H$ is said to satisfy a small ball property with parameters $(\kappa,\delta)$ if
\[ \P\left[ h^2\geq \kappa^2  \E[h^2] \right] \geq \delta
\quad\text{for all } h\in H. \]
This condition is used frequently in problems involving a quadratic term (see, e.g.~\cite{lecue2018regularization,mendelson2014learning,mendelson2020geometry}).

\begin{remark}
\label{rem:relation.SLB.SBP.norm.equiv}
	Let $H=\{ \langle Ax,x\rangle^{\frac{1}{2}}: x\in S \}$.
	Then the following hold.
	\begin{enumerate}[(i)]
	\item
	If $H$ satisfies the small ball property with constants $(\kappa,\delta)$, then $H$ satisfies a stable lower bound with parameters
	\[(m,\gamma,k,l)=\left(m, 1-\frac{\delta\kappa^2}{2},s_1\delta m, s_2\delta m\right)\]
	for every $m$, where $s_1,s_2>0$ are absolute constants.
	\item
	If $H$ is bounded in $L_p$ for some $p>2$, that is,
	\[ \E[|h|^p]^{\frac{1}{p}} \leq L \quad\text{for all } h\in H\]
	where $L$ is a fixed constant, then $H$ satisfies a stable lower bound with parameters
	\[ (m,\gamma,k,l)
	=\left(m,\gamma,s_1 m \gamma^{\frac{p}{p-2}} ,s_2 m \gamma^{ \max\{ \frac{p}{p-2} ,2\} } \right) \]
	where $s_1,s_2>0$ are constants depending only on $p$ and $L$.
	\end{enumerate}
\end{remark}

The first statement is an immediate consequence of a Binomial concentration inequality and second statement can be found in \cite[Section 2.1]{mendelson2017extending} together with a more thorough discussion and analysis of stable lower bounds.

The following theorem is the main result of this section. 
To formulate it, recall that for a random matrix $A$, $\E [A]$ is denoted by $\mathbb{A}$.

\begin{theorem}
\label{thm:singular.value}
	There are absolute constants $C_1$ and $C_2$ such that the following holds.
	Fix $\gamma,\tau\in(0,1)$  and assume that
	\begin{enumerate}[(a)]
	\item
	the random matrix $A$ satisfies a stable lower bound with parameters $(m,\frac{\gamma}{2},2l,k)$ were $k\geq \max\{4,2\log (\frac{4}{\tau})\}$,
	\item
	the sample size satisfies
	\[ N \geq C_1 \max\left\{ \| \E[ A \mathbb{A}^{-1} A] \|_{\mathrm{op}} , \frac{d m}{\tau k} \log\left(\frac{\log(3d) \E[ \| A \mathbb{A}^{-1} A \|_{\mathrm{op}} ] }{\gamma\tau \|\E[A \mathbb{A}^{-1} A]\|_{\mathrm{op}} } \right) \right\} .\]
	\end{enumerate}
	Then, with probability at least
	\[1-2\exp\left(-C_2 N \tau \min\left\{ \frac{l}{m},\frac{k}{m} \right\} \right),\]
	for every $x\in \mathbb{R}^d$ and every $J_j\subseteq I_j$ with $|J_j|\leq l$ for all $j$, we have that
	\[  \left| \left\{ j: \frac{1}{m}\sum_{i\in I_j\setminus J_j} \langle A_i x,x\rangle
	\geq (1-\gamma) \|x\|^2  \right\}\right| \geq (1-\tau) n .\]	
\end{theorem}

The formulation in Theorem \ref{thm:singular.value} is tailer-made for the median-of-means analysis which is presented in the next section (as part of the proof of Theorem \ref{thm:main.convex.intro}).
It ensures a stability property that is stronger than a mere lower bound on the quadratic form (and therefore, a lower estimate on the smallest singular value). Rather, it gives a useful lower bound even if a proportion of the sample is arbitrarily modified or removed.
This will be useful in the proof of Theorem \ref{thm:main.convex.intro} when passing from the Hessian evaluated at the optimizer to the Hessian evaluated at mid-points.
In addition, Theorem \ref{thm:singular.value} provides an almost isometric lower bound (i.e.\ when $\gamma$ is close to zero) while the analysis for Theorem \ref{thm:main.convex.intro} actually only requires an isomorphic lower bound.

Finally, let us formulate the following immediate consequence of Theorem \ref{thm:singular.value}.

\begin{corollary}
\label{cor:smallest.singular.value}
	There are absolute constants $C_1$ and $C_2$ such that the following holds.
	Fix $\gamma\in(0,1)$ and assume that
	\begin{enumerate}[(a)]
	\item
	the random matrix $A$ satisfies a stable lower bound with parameters $(N,\frac{\gamma}{2},2l,k)$ for some $k\geq 4$,
	\item
	the sample size satisfies
	\[ N \geq C_1 \max\left\{ \| \E[ A \mathbb{A}^{-1} A ] \|_{\mathrm{op}} , \frac{d N}{ k} \log\left(\frac{\log(3d) \E[ \| A \mathbb{A}^{-1} A \|_{\mathrm{op}} ]  }{\gamma  \| \E[A \mathbb{A}^{-1} A] \|_{\mathrm{op}} } \right) \right\} .\]
	\end{enumerate}
	Then, with probability at least
	\[1-2\exp\left(-C_2 N \min\left\{ \frac{l}{N},\frac{k}{N} \right\} \right),\]
	we have that
	\[ \lambda_{\min}\left( \frac{1}{N} \sum_{i=1}^N A_i \right)
	\geq (1-\gamma) \lambda_{\min}(\E[A]). \]
\end{corollary}
\begin{proof}
	Applying Theorem \ref{thm:singular.value} with $n=1$ (hence $m=N$), $\tau=0.6$, and $J_j=\emptyset$ yields the following:
	with probability at least $1-2\exp(-C_2 N \min\{ \frac{l}{N},\frac{k}{N} \} )$, for every $x\in\mathbb{R}^d$, we have that
	\begin{align*}
	\left\langle \frac{1}{N} \sum_{i=1}^N A_i x,x \right\rangle
	&\geq (1-\gamma) \langle \E[A]x,x\rangle.
	\end{align*}
	A twofold application of the extremal expression $\lambda_{\min}(\cdot)=\min_{x\in\mathbb{R}^d \text{ s.t.\ }\|x\|_2=1} \langle \,\cdot\, x,x\rangle$ of the smallest singular value completes the proof.
\end{proof}

In order to recover Corollary \ref{cor:smallest.singular.value.norm.equiv} from Corollary \ref{cor:smallest.singular.value} (and later also to apply Theorem \ref{thm:singular.value} in the proof of Theorem \ref{thm:main.convex.intro}) we need two simple observations, showing that under a $L_4-L_2$ norm equivalence, the estimate on the sample size in Corollary \ref{cor:smallest.singular.value} and Theorem \ref{thm:singular.value} can be simplified.

\begin{lemma}
\label{lem:operator.norm.relation}
	We have that
	\begin{align}
	\label{eq:op.norm.simplify}
	1
	\leq \| \E[ A\mathbb{A}^{-1}A ] \|_{\mathrm{op}}
	\leq \E[\| A\mathbb{A}^{-1}A \|_{\mathrm{op}}]
	\leq d  \| \E[ A\mathbb{A}^{-1}A ] \|_{\mathrm{op}}  .
	\end{align}
\end{lemma}
\begin{proof}
	We start with the final inequality in \eqref{eq:op.norm.simplify}.
	Recall that $\mathbb{A}=\E[A]$ and note that the relation $\|\cdot\|=\|\mathbb{A}^{\frac{1}{2}}\cdot\|_2$ transfers to the operator norm in the following sense:
	for any $(d\times d)$-matrix $B$, we have that
	\begin{align}
	\label{eq:operatorn.norm.relation}
	\begin{split}
	\|B\|_{\mathrm{op}}
	&= \|\mathbb{A}^{-\frac{1}{2}} B \mathbb{A}^{-\frac{1}{2}} \|_{\mathrm{op}_2} \quad\text{where}\\
	\|B\|_{\mathrm{op}_2}
	&:=\max_{x,y\in\mathbb{R}^d \text{ s.t.\ } \|x\|_2,\|y\|_2\leq 1} \langle Bx,y\rangle.
	\end{split}
	\end{align}
	Indeed, $\{x \in\mathbb{R}^d : \langle \mathbb{A}x,x \rangle \leq 1\}$ is the ellipsoid $\mathbb{A}^{-\frac{1}{2}} B_2^d$ (where $B_2^d$ is the unit ball w.r.t.\ the Euclidean norm).

	The norm $\|\cdot\|_{\mathrm{op}_2}$ is the usual spectral norm which, for positive semidefinite matrices, equals the largest singular value $\lambda_{\max}(\cdot)$.
	
	Setting $F:=\mathbb{A}^{-\frac{1}{2}} A \mathbb{A}^{-\frac{1}{2}}$ which is positive semidefinite, we thus have
	\begin{align*}
	\E[\| A\mathbb{A}^{-1}A \|_{\mathrm{op}}]
	&=\E[ \| F^2\|_{\mathrm{op}_2} ] \\
	&=\E[ \lambda_{\max}(F^2) ]
	\leq \sum_{i=1}^d \E[ (F^2)_{ii} ],
	\end{align*}
	where the last inequality follows by bounding the largest singular value by the trace.
	On the other hand, for every $i=1,\dots,d$, taking the Euclidean unit vector $x=y=e_i$ in \eqref{eq:operatorn.norm.relation} shows that
	\[ \| \E[ A\mathbb{A}^{-1}A ]\|_{\mathrm{op}}
	=\| \E[ F^2 ]\|_{\mathrm{op}_2}
	\geq \E[(F^2)_{ii}]\]
	and hence the last inequality of the lemma follows.
	
	The second inequality in \eqref{eq:op.norm.simplify} is trivial and the first inequality follows from \eqref{eq:operatorn.norm.relation} and Jensen's inequality for matrices (which states that $\E[F^2]\succeq\E[F]^2=\mathrm{Id}$).
	This completes the proof.
\end{proof}

\begin{lemma}
\label{lem:growth.op.squared.A}
	Assume that there is a constant $L$ such that $\E[\langle Ax,x\rangle^2]^{\frac{1}{4}}\leq L$ for all $x\in S$.
	Then
	\[ \| \E[ A\mathbb{A}^{-1}A ] \|_{\mathrm{op}} \leq d L^4 .\]
\end{lemma}
\begin{proof}
	Recall that $F:=\mathbb{A}^{-\frac{1}{2}} A \mathbb{A}^{-\frac{1}{2}}$ and that, by \eqref{eq:operatorn.norm.relation},
	\begin{align*} 
	\| \E[ A\mathbb{A}^{-1}A ] \|_{\mathrm{op}}	
	= \| \E[ F^2 ] \|_{\mathrm{op}_2}
	= \max_{x\in\mathbb{R}^d \text{ s.t.\ } \|x\|_2\leq 1} \E[ \langle F^2x,x\rangle] .
	\end{align*}
	Moreover, for every $x\in\mathbb{R}^d$, we have
	\begin{align*}
	\langle F^2x,x\rangle
	&= \langle F F^{\frac{1}{2}}x, F^{\frac{1}{2}} x\rangle \\
	&\leq  \| F \|_{\mathrm{op}_2} \|F^{\frac{1}{2}}x\|_2^2,
	\end{align*}
	which, in combination with the Cauchy-Schwartz inequality, yields that
	\[ \| \E[ A\mathbb{A}^{-1}A ] \|_{\mathrm{op}}	
	\leq \max_{x\in\mathbb{R}^d \text{ s.t.\ } \|x\|_2\leq 1} \E[ \| F \|_{\mathrm{op}_2}^2 ]^{\frac{1}{2}} \E[ \|F^{\frac{1}{2}}x\|_2^4 ]^{\frac{1}{2}}.\]
	At this point, recall that $\{x\in\mathbb{R}^d : \langle \mathbb{A}x,x \rangle \leq 1\}$ is the ellipsoid $\mathbb{A}^{-\frac{1}{2}} B_2^d$, and note that an equivalent formulation of \eqref{eq:cor.smallest.singular.ass.norm.equiv} is
	\begin{align*}
	\E[  \|F^{\frac{1}{2}}x\|_2^4]^{\frac{1}{2}}
	&=\E[\langle Fx,x\rangle^2]^{\frac{1}{2}} \\
	&\leq L^2 \E[\langle Fx,x\rangle]  
	\end{align*}
	for every $x\in\mathbb{R}^d$, and that $\E[\langle Fx,x\rangle]=1$ for every $x\in\mathbb{R}^d$ with $\|x\|_2=1$.
	Hence, bounding $\| F \|_{\mathrm{op}_2}=\lambda_{\max}(F)$ by the trace of $F$,
	\begin{align*}
	\E[ \| F \|_{\mathrm{op}_2}^2 ]^{\frac{1}{2}}
	&\leq \sum_{i=1}^d \E[ (F_{ii})^2 ]^{\frac{1}{2}} \\
	&= \sum_{i=1}^d \E[ \langle Fe_i,e_i\rangle^2 ]^{\frac{1}{2}}
	\leq   d L^2.
	\end{align*}
	In conclusion $ \| \E[ A\mathbb{A}^{-1}A ] \|_{\mathrm{op}}	\leq dL^4$, as claimed.
\end{proof}

\begin{proof}[Proof of Corollary \ref{cor:smallest.singular.value.norm.equiv}]
	By Remark \ref{rem:relation.SLB.SBP.norm.equiv} (ii), the random matrix $A$ satisfies a stable lower bound with parameters $(N,\frac{\gamma}{2}, s_1 N\gamma^2, s_2 N \gamma^2)$ for constants $s_1$ and $s_2$ depending only on $L$.
	The proof therefore follows from Corollary \ref{cor:smallest.singular.value} together with Lemma \ref{lem:operator.norm.relation} and Lemma \ref{lem:growth.op.squared.A}.
\end{proof}

Throughout this article, we often aim to prove statements that are supposed to hold with high probability, uniformly over (uncountable) sets; for example that for all $x\in S$, most coordinates of $(\langle A_i x,x\rangle)_{i=1}^N$ are suitably large.
The proofs of such statements follow (in principle) a recurring scheme:
in a first step, we prove that a slightly stronger variant of the statement holds with high probability for a single element (i.e.\ Lemma \ref{lem:slb.uniform} below).
By a trivial union bound, this allows to extend the validity of the statement to a finite set of high cardinality, and we shall choose such a set with an extra feature: it approximates the original set in a suitable sense (i.e.\ Lemma \ref{lem:matrix.covering} below).
It then remains to show that the oscillations caused by passing from the approximating set to the whole set do not distort the outcome by too much (i.e.\ Lemma \ref{lem:net.to.uniform} below).
\vskip0.3cm
From now on, we shall continue under the same assumptions used in the formulation of Theorem \ref{thm:singular.value} without mentioning these assumptions at every instance.

\begin{lemma}
\label{lem:slb.uniform}
	There is an absolute constant $C$ such that the following holds.
	For every $x\in S$, with probability at least $1-2\exp(- C N \tau \frac{k}{m})$, for all choices $J_j\subseteq I_j$ with $|J_j|\leq 2l$, we have that
	\begin{align*}
	\left| \left\{ j : \frac{1}{m}\sum_{i\in I_j\setminus J_j} \langle A_ix,x\rangle
	\geq 1-\frac{\gamma}{2} \right\} \right| \geq \left(1-\frac{\tau}{2}\right)n.
	\end{align*}
	In particular, for a set $\bar{S}\subseteq S$ satisfying $\log (|\bar{S}|)\leq \frac{1}{2} C \tau N \frac{k}{m}$, the above statement holds uniformly over $x\in \bar{S}$.
\end{lemma}
\begin{proof}
	The proof follows from an application of Bennett's inequality and is essentially  the same as the proof of \cite[Lemma 4.3]{mendelson2017extending}.
	For completeness, we sketch the argument.
	Fix some $x\in S$ and, for every $j$, set
	\[ \delta(j):=
	\begin{cases}
	0 & \text{if } \frac{1}{m}\sum_{i\in I_j\setminus J_j} \langle A_ix,x\rangle \geq 1-\frac{\gamma}{2} \text{ for all } J_j\subseteq I_j \text{ with } |J_j|\leq 2l,\\
	1 & \text{otherwise.}
	\end{cases} \]
	By definition of the stable lower bound, we have that
	\[\delta:=\P[\delta(j)=1]
	\leq 2\exp(-k)
	\leq \frac{3}{4}.\]
	If $\delta=0$, there is nothing to prove, so assume otherwise.
	Now make use of Bennett's inequality \cite[Theorem 2.9]{boucheron2013concentration}:
	for every $u\geq 2$ with probability at least $1-2\exp(-C\delta n  u \log (u))$, we have that
	\[|\{ j : \delta(j)=1\} | \leq u \delta n. \]
	Apply this to
	\[ u:=\frac{\tau}{2\delta}
	\geq \exp\left(\frac{k}{2}\right)
	\geq 2\]
	(where the second inequality follows as $\tau\geq \frac{1}{4}\exp(-\frac{k}{2})$ by assumption) and observe that
	\begin{align*}
	C\delta n  u \log (u)
	&=\frac{1}{2} C \tau n\log \left( \frac{\tau}{2\delta} \right) \\
	&\geq \frac{C \tau n k}{4}
	=\frac{C \tau N k }{4m}.
	\end{align*}
	This completes the first part of the proof.
	The ``in particular'' part is a consequence of the union bound.
\end{proof}

In a next step we choose a subset of $S$ of high cardinality which ``covers'' $S$ w.r.t.\ the natural metric.

\begin{definition}
Let $(S,d)$ be a metric space and let $\rho>0$. A set $\bar{S} \subset S$ is a $\rho$-cover of $S$ with respect to the metric $d$ if for every $s \in S$ there is $\bar{s} \in \bar{S}$ such that $d(s,\bar{s}) \leq \rho$.
\end{definition}

For now, let $C$ be the absolute constant from Lemma \ref{lem:slb.uniform} and set $C_1$ to be the constant from assumption (b)
in Theorem \ref{thm:singular.value}. Recall that we have the freedom to choose $C_1$ as we see fit, and set $C_0$ to be a constant specified in what follows.

\begin{lemma}
\label{lem:matrix.covering}
	Let
	\[
	\rho:= \frac{C_0 \gamma\tau \|\E[A \mathbb{A}^{-1} A]\|_{\mathrm{op}}}{ \log (3d) \E[\|A \mathbb{A}^{-1} A\|_{\mathrm{op}}] }.
	\]
	Then there is a $\rho$-cover $\bar{S}\subseteq S$ (w.r.t.\ the norm $\|\cdot\|$) of log-cardinality at most $\frac{1}{2} C \tau N \frac{k}{m}$.
\end{lemma}
\begin{proof}
	By a simple volumetric argument (see e.g.\ \cite[Exercise 2.2.14]{talagrand2014upper}) there is a set
	\[\bar{S}_2\subseteq S_2 := \{ x\in\mathbb{R}^d : \|x\|_2=1\}\]
	such that, for every $x\in S_2$ there is $y=y(x)\in \bar{S}_2$ with $\|x-y\|_2\leq \rho$ with cardinality
	\begin{align*}
	\log ( |\bar{S}_2| )
	&\leq d\log\left(\frac{6}{\rho}\right)\\
	&\leq C_3 d\log\left(\frac{\log(3d) \E[\|A \mathbb{A}^{-1} A\|_{\mathrm{op}}]}{\gamma\tau \|\E[A \mathbb{A}^{-1} A]\|_{\mathrm{op}}} \right),
	\end{align*}
	where $C_3$ depends only on $C_0$.
	By assumption (b) on the sample size in Theorem \ref{thm:singular.value}
	we conclude that $\log( |\bar{S}_2| ) \leq \frac{1}{2}C \tau N \frac{k}{m}$ once the constant $C_1$ (in assumption (b)) is chosen sufficiently large.
	Finally, the relation $\|\cdot\|=\|\mathbb{A}^{\frac{1}{2}}\cdot\|_2$ readily implies that the set
	\[\bar{S}:=\{ \mathbb{A}^{-\frac{1}{2}} y : y\in \bar{S}_2 \}\]
	satisfies the statement of the lemma.
\end{proof}

The final step in the proof consist of showing that the transition from a $\rho$-cover $\bar{S}$ to the whole $S$ does not distort the wanted outcome by too much.
To that end, we fix from now on the $\rho$-cover $\bar{S}$ of Lemma \ref{lem:matrix.covering} and denote by $y=y(x)\in \bar{S}$ the element satisfying $\|x-y\|\leq \rho$.
Also, for every $x\in S$, set
\[\Delta(x)
:= \langle A x , x\rangle -  \langle A y, y\rangle. \]

\begin{remark}
The following two preliminary lemmas are stated here to maintain a chronological order within the proofs. However, it might be helpful to skip to (the proof of) Lemma \ref{lem:net.to.uniform} where their role is clarified and return here afterwards.
\end{remark}

\begin{lemma}
\label{lem:matrix.Delta}
	We have that
	\[ \E[|\Delta(x)|]\leq 3 \rho
	\quad\text{for  every } x\in S.\]
\end{lemma}
\begin{proof}
	Fix some $x\in S$.
	As $A$ is symmetric and positive semidefinite, it follows from the triangle inequality that
	\[\left| \langle Ax,x\rangle^{\frac{1}{2}} - \langle Ay,y\rangle^{\frac{1}{2}}  \right|
	\leq \langle A (x-y),x-y\rangle^{\frac{1}{2}}  . \]
	Combined with the fact that $|a^2-b^2|\leq 2 |a-b| \max\{a,b\}$ for $a,b\geq 0$, we have
	\begin{align*}
	|\Delta(x)|
	&\leq 2 \langle A (x-y), x-y\rangle^{\frac{1}{2}}   \max\left\{ \langle Ax,x\rangle^{\frac{1}{2}} , \langle Ay,y\rangle^{\frac{1}{2}}  \right\},
	\end{align*}
	and by the Cauchy-Schwartz inequality
	\begin{align*}
	\E[|\Delta(x)|]^2
	&\leq  4 \E[ \langle A(x-y),x-y\rangle]  \left( \E[\langle Ax,x\rangle] + \E[\langle Ay,y\rangle]\right)\\
	&= 4  \| x-y\|^2   (\|x\|^2+\|y\|^2) \\
	&= 8 \| x-y\|^2,
	\end{align*}
	where the second equality follows by definition of the norm $\|\cdot\|$ and the final one holds because $x,y\in S$.
\end{proof}

\begin{lemma}
\label{lem:matrix.rademacher}
	Let $C_T$ be the absolute constant in Talagrand's concentration inequality \cite{talagrand1994sharper}.
	Moreover, set $b:= \gamma\frac{m}{l}$, let $(A_i)_{i=1}^N$ be independent copies of $A$ and set $(\varepsilon_i)_{i\geq 1}$ to be independent Rademacher random variables that are independent of $(A_i)_{i=1}^N$.
	Then we have that
	\[ \E\left[ \sup_{x\in S} \left| \frac{1}{N} \sum_{i=1}^N \varepsilon_i (|\Delta_{i}(x)|\wedge b)\right|\right]
	\leq \frac{\gamma\tau}{ C_T 16} \]
	once $C_0$ (the absolute constant of Lemma \ref{lem:matrix.covering}) is small enough.
\end{lemma}
\begin{proof}
	As a preliminary step, we invoke the contraction inequality for Bernoulli processes \cite[Corollary 3.17]{ledoux2013probability}
	%\cite[Theorem 11.6]{boucheron2013concentration}, 
	conditionally on $(A_i)_{i=1}^N$, and applied to the $1$-Lipschitz map $t\mapsto |t|\wedge b$ which passes through the origin. It follows that
	\begin{align*}
	\E\left[ \sup_{x\in S} \left| \frac{1}{N} \sum_{i=1}^N \varepsilon_i (|\Delta_{i}(x)| \wedge  b)\right|\right]
	&\leq 2\E\left[ \sup_{x\in S} \left| \frac{1}{N} \sum_{i=1}^N \varepsilon_i \Delta_{i}(x)\right|\right].
	\end{align*}
	Rewriting
	\[\Delta_i(x)=\langle A_i (x+y), x-y\rangle\]
	we have that
	\begin{align*}
	\left| \frac{1}{N} \sum_{i=1}^N \varepsilon_i \Delta_{i}(x) \right|
	&=\left| \left\langle \left( \frac{1}{N} \sum_{i=1}^N \varepsilon_i A_{i} \right)(x+y), x-y\right\rangle \right|  \\
	&\leq 2\rho  \left\|\frac{1}{N} \sum_{i=1}^N \varepsilon_i A_i \right\|_{\mathrm{op}} ,
	\end{align*}
	because $\|x+y\|\leq 2$ and $\|x-y\| \leq \rho$.
	Next, set $F_i:= \varepsilon_i \mathbb{A}^{-\frac{1}{2}}A_i \mathbb{A}^{-\frac{1}{2}}$ for every $i$ and recall the relation between $\|\cdot\|_{\mathrm{op}}$ and the spectral norm $\|\cdot\|_{\mathrm{op}_2}$ stated in \eqref{eq:operatorn.norm.relation}. The Matrix-Bernstein inequality \cite[Theorem I]{tropp2016expected}
	implies that
	\begin{align*}	
	\E\left[ \left\| \sum_{i=1}^N \varepsilon_i A_i \right\|_{\mathrm{op}} \right]
	&=\E\left[ \left\| \sum_{i=1}^N F_i \right\|_{\mathrm{op}_2} \right]  \\
	&\leq  \left(  C(d) N \| \E[F^2] \|_{\mathrm{op}_2}  \right)^{\frac{1}{2}} +  C(d) \E\left[\max_{1\leq i \leq N} \|F_i\|_{\mathrm{op}_2}^2 \right]^{\frac{1}{2}}
	\end{align*}
	where
	\begin{align*}
	C(d)
	&=4(1+2\log(2d))
	\leq 22\log(2d).
	\end{align*}
	Further, estimating the maximum by the sum and using that $\|F\|_{\mathrm{op}_2}^2=\|F^2\|_{\mathrm{op}_2}$, we trivially have
	\begin{align*}
	\E\left[\max_{1\leq i \leq N} \|F_i\|_{\mathrm{op}_2}^2\right]
	&\leq N \E[\|F\|_{\mathrm{op}_2}^2] \\
	&=N \E[\| A\mathbb{A}^{-1}A \|_{\mathrm{op}}] .
	\end{align*}

	Putting everything together, we therefore obtain
	\begin{align*}
	&\E\left[ \sup_{x\in S} \left| \frac{1}{N} \sum_{i=1}^N \varepsilon_i (|\Delta_{i}(x)| \wedge  b)\right|\right] \\
	&\leq  4 \rho C(d) \left( \left( \frac{ \| \E[ A\mathbb{A}^{-1}A ] \|_{\mathrm{op}} }{N}\right)^{\frac{1}{2}} + \left( \frac{ \E[\| A\mathbb{A}^{-1}A \|_{\mathrm{op}}] }{ N}\right)^{\frac{1}{2}} \right)\\
	&\leq  4 \rho C(d) \left( 1 + \left( \frac{ \E[\| A\mathbb{A}^{-1}A \|_{\mathrm{op}}] }{  \| \E[ A\mathbb{A}^{-1}A ] \|_{\mathrm{op}} }\right)^{\frac{1}{2}} \right),
	\end{align*}
	where the second inequality follows from assumption (b) in Theorem \ref{thm:singular.value}: that on the sample size satisfies $N \geq \| \E[  A\mathbb{A}^{-1}A  ] \|_{\mathrm{op}}$.
	By Lemma \ref{lem:operator.norm.relation}, the expectation of the operator norm is always larger than the operator norm of the expectation, hence
	\begin{align*}
	\E\left[ \sup_{x\in S} \left| \frac{1}{N} \sum_{i=1}^N \varepsilon_i (|\Delta_{i}(x)| \wedge  b)\right|\right]
	&\leq  8 \rho C(d)  \frac{ \E[\| A\mathbb{A}^{-1}A \|_{\mathrm{op}}] }{  \| \E[ A\mathbb{A}^{-1}A ] \|_{\mathrm{op}} } .
	\end{align*}
	Recalling the value of $\rho$ from Lemma \ref{lem:matrix.covering} shows that the latter term is at most $\frac{\gamma\tau}{C_T16}$.
	This completes the proof.
\end{proof}

\begin{lemma}
\label{lem:net.to.uniform}
	There exists an absolute constant $C$ such that the following holds.
	For every $x\in S$ and every $j$, let
	\[ J_j^\ast(x):=\{ \text{largest } l \text{ coordinates of }  (|\Delta_i(x)|)_{i\in I_j}\}\subseteq I_j.\]
	%$J_j^\ast(x):= \subseteq I_j$ be the (random) subset of the $l$ largest coordinates of $(|\Delta_i(x)|)_{i\in I_j}$.
	Then, with probability at least $1-2\exp(- C_2 N  \tau \frac{l}{m} )$, we have that
	\begin{align}
	\label{eq:lem.talagrand.random.matrix}
	 \sup_{x\in S} \left|\left\{ j : \frac{1}{m} \sum_{i\in I_j\setminus J_j^\ast(x)} |\Delta_{i}(x)|  > \frac{\gamma }{2}  \right\} \right|  \leq \frac{\tau n}{2}.
	 \end{align}
\end{lemma}

\begin{proof}
	Set $b:=\gamma \frac{m}{l}$.
	We claim that, for every $x\in S$ and every $j$,
	\begin{align}
	\label{eq:sum.and.sum.min}
	\begin{split}
	\frac{1}{m} \sum_{i\in I_j\setminus J_j^\ast(x)}  |\Delta_{i}(x)|
	&> \frac{ \gamma  }{ 2 }
	\quad\text{implies}\quad \\
	\frac{1}{m} \sum_{i\in I_j} (|\Delta_{i}(x)|\wedge  b)
	&> \frac{ \gamma  }{ 2 }.
	\end{split}
	\end{align}
	Indeed, if $|\Delta_{i}(x)|\leq b$ for $i\in I_j\setminus J_j^\ast(x)$, the second sum  in \eqref{eq:sum.and.sum.min} is trivially at least as big as the first one.
	Otherwise, if there is $i_0\in I_j\setminus J_j^\ast(x)$ for which $|\Delta_{i_0}(x)|>b$, then by definition of $J_j^\ast(x)$, there are least $l$ coordinates $i\in I_j$ for which $|\Delta_{i}(x)|>b$.
	In particular, the second sum in \eqref{eq:sum.and.sum.min} is at least $\frac{1}{m}l b=\gamma $, and \eqref{eq:sum.and.sum.min} holds.
Therefore, it suffices to show that $R\leq \frac{1}{4}\gamma\tau $ holds with high probability, where
	\[R:=\sup_{x\in S}  \frac{1}{N} \sum_{i=1}^N (|\Delta_{i}(x)|\wedge b).\]
		
	To that end, Talagrand's concentration inequality for bounded empirical processes \cite{talagrand1994sharper} (see also \cite{boucheron2013concentration}): 
	there is an absolute constant $C_T$ such that
	\begin{align*}
	\P\left[ R \leq R_1 + C_T\left( R_2 + R_3 + R_4\right)\right]
	\geq 1-2\exp(-u)
	\end{align*}
	for every $u\geq 0$, where
	\begin{align*}
	R_1 &:= \sup_{x\in S} \E[ |\Delta(x)| \wedge b], \\
	R_2 &:= \sup_{x\in S} \E[( |\Delta(x)|\wedge  b )^2]^{\frac{1}{2}}  \left( \frac{u}{N} \right)^{\frac{1}{2}}, \\
	R_3 &:= \sup_{x\in S} \| |\Delta(x)| \wedge b \|_{L^\infty}  \frac{u}{N}, \\
	R_4 &:= \E\left[ \sup_{x\in S} \left| \frac{1}{N} \sum_{i=1}^N \varepsilon_i (|\Delta_{i}(x)|\wedge b)\right|\right].
	\end{align*}
	Thus, to conclude the proof, let us show that for $u= C N \tau \frac{l}{m}$, the sum $R_1+C_T (R_2+R_3+R_4)$ is smaller than $\frac{1}{4}\gamma\tau $.	
		
	We now proceed to bound $R_1,\dots,R_4$.
	First, recall that $C_0$ is the absolute constant of Lemma \ref{lem:matrix.covering} which we are still able to choose as small as we want, and note that
	\[ \rho
	= \frac{C_0 \gamma\tau \|\E[A\mathbb{A}^{-1}A]\|_{\mathrm{op}}}{ \log (3d) \E[\| A\mathbb{A}^{-1}A\|_{\mathrm{op}}] }
	\leq  \frac{C_0 \gamma\tau }{\log(3)}
	\leq C_0 \gamma\tau,\]
	where the first inequality holds by Lemma \ref{lem:operator.norm.relation} and the second one by absorbing $\frac{1}{\log(3)}$ into $C_0$.
	
	By Lemma \ref{lem:matrix.Delta} we have that $\E[|\Delta(x)|]\leq 3\rho$ for every $x\in S$; thus
	\[R_1\leq  3 \rho
	\leq \frac{ \gamma \tau }{16} \]
	as soon as $C_0 < \frac{1}{48}$. 

	For the terms $R_2$ and $R_3$, which involve $u= C N \tau \frac{l}{m} $, note that
	\[\E[( |\Delta(x)|\wedge  b )^2] \leq  3 \rho b
	\quad\text{for every }x\in S.\]
	Indeed, this follows from the trivial estimate $(|\Delta(x)|\wedge b)^2\leq  |\Delta(x)| b$ and Lemma \ref{lem:matrix.Delta} once again.
	Recalling that $b=\gamma\frac{m}{l}$, we therefore have
	\begin{align*}
	R_2
	&\leq  \left( \frac{3  \rho b u}{N} \right)^{\frac{1}{2}} \\
	&\leq \left( 3 C_0 C \gamma^2\tau^2\right)^{\frac{1}{2}}
	\leq \frac{ \gamma\tau}{C_T 16}	
	\end{align*}
	once $C_0$ is small enough. Moreover,
	\[  R_3
	\leq  \frac{bu}{N}
	= C \gamma\tau
	\leq \frac{\gamma\tau}{ C_T 16}\]
	provided that $C$ is small enough.
	
	Finally, by Lemma  \ref{lem:matrix.rademacher}, we have $R_4\leq \frac{\gamma\tau}{C_T 16}$.
	This completes the proof
\end{proof}

\begin{proof}[Proof of Theorem \ref{thm:singular.value}]
	The statement of Theorem \ref{thm:singular.value} is clearly homogeneous in $x\in\mathbb{R}^d$, hence it suffices to restrict to $x\in S$.
	The proof follows from a combination of Lemma \ref{lem:slb.uniform} and Lemma \ref{lem:net.to.uniform}.
	Indeed, using the notation of Lemma \ref{lem:net.to.uniform}, for every $x\in S$, $y=y(x)$, and every $J_j\subseteq I_j$ with $|J_j|\leq l$, we write
	\begin{align}
	\nonumber
	\frac{1}{m}\sum_{i\in I_j\setminus J_j} \langle A_i x,x\rangle
	&\geq \frac{1}{m}\sum_{i\in I_j\setminus (J_j\cup J_j^\ast(x))} \langle A_i x,x\rangle\\
	\label{eq:sum.x.y.delta}
	&\geq  \frac{1}{m}\sum_{i\in I_j\setminus (J_j\cup J_j^\ast(x))} \langle A_i y,y\rangle
	- \frac{1}{m}\sum_{i\in I_j\setminus (J_j\cup J_j^\ast(x))} |\Delta_i(x)|.
	\end{align}
	
	By Lemma \ref{lem:slb.uniform}, with probability at least $1-2\exp(-C\tau N \frac{k}{m})$, for every $x$, $y$, and sets $J_j$ as above, we have that
	\[ \frac{1}{m}\sum_{i\in I_j\setminus (J_j\cup J_j^\ast(x))} \langle A_i y,y\rangle \geq 1-\frac{\gamma}{2}
	\quad\text{on at least } \left(1-\frac{\tau}{2}\right)n \text{ blocks }.\]
	Next, by Lemma \ref{lem:net.to.uniform}, with probability at least $1-2\exp(-C\tau N \frac{l}{m})$, for every $x$, $y$,  and sets $J_j$ as above, we have that
	\[\frac{1}{m}\sum_{i\in I_j\setminus (J_j\cup J_j^\ast(x))} |\Delta_i(x)| \leq \frac{\gamma}{2}
	\quad\text{on at least } \left(1-\frac{\tau}{2}\right)n \text{ blocks }.\]
	Taking the intersection of the two high probability events yields the claim.
\end{proof}

\section{Proofs of the main results}
\label{sec:proof.main}

In addition to the \emph{notational conventions} already explained in Section \ref{sec:mom.singular.value}, set $c,c_0,c_1,\dots$ to be constants that may depend on $L$ (the parameter appearing in Assumption \ref{ass:norm.equiv}).
As before, these constants may change their values from line to line.
Moreover, $0<\theta<\tau<\frac{1}{4}$ are constants that may depend on $L$ as well.
For the sake of a clearer presentation, rather than stating the explicit values of $\theta$ and $\tau$ now, we collect constraints on their values along the way.

Next recall that
\[ n=\theta N \min\left\{1,\frac{r^2 }{\sigma^2} \right\}
\quad\text{and}\quad
m=\frac{N}{n},\]
where we assume without loss of generality that $m$ and $n$ are integers. Thus, $N=nm$ and $m\geq \frac{1}{\theta}$.
Finally, set
\begin{align*}
\mathcal{B}_r^\ast &:=\{x\in\mathcal{X} : \|x-x^\ast\| \leq r \},\\
\mathcal{S}_r^\ast &:=\{x\in\mathcal{X} : \|x-x^\ast\| =r \}
\end{align*}
to be the ball and sphere of radius $r$ around $x^\ast$ restricted to $\mathcal{X}$, respectively.
Recall the constant $r_0$ of Assumption \ref{ass:hessian.midpoints} and assume throughout that $r\leq r_0$.

\vspace{0.5em}
The proof of Theorem \ref{thm:main.convex.intro} relies on the following decomposition:
for every $j$ and every $x\in\mathcal{X}$, a Taylor expansion implies that
\begin{align*}
&\widehat{f}_{I_j}(x) - \widehat{f}_{I_j}(x^\ast)\\
&= \frac{1}{m}\sum_{i\in I_j} \langle \nabla F(x^\ast ,\xi_i), x-x^\ast \rangle +\frac{1}{2}\frac{1}{m}\sum_{i\in I_j} \langle \nabla^2 F(z_i ,\xi_i) (x-x^\ast), x-x^\ast \rangle\\
&=: M_{x,x^\ast}(j) + \frac{1}{2}Q_{x,x^\ast}(j),
\end{align*}
where $z_i$ are midpoints between $x$ and $x^\ast$ (and each $z_i$ may depend on $\xi_i$).
For obvious reasons we call $Q_{x,x^\ast}$ the \emph{quadratic term} and $M_{x,x^\ast}$ the \emph{multiplier term}.

We start with the proof of the estimation error, formulated in Proposition \ref{prop:estimation.error}.
With the decomposition of $\widehat{f}_{I_j}(x) - \widehat{f}_{I_j}(x^\ast)$ into a multiplier and a quadratic term at hand, the \emph{strategy of the proof} that $x^\ast$ defeats any competitor $x\in \mathcal{S}_r^\ast$ on the $j$-th block (i.e.\ that $\widehat{f}_{I_j}(x) - \widehat{f}_{I_j}(x^\ast)>0$), consists of showing that the quadratic term is likely to be positive (of order $r^2$) and the multiplier term is likely not to be too negative:

\begin{lemma}
\label{prop:actually.the.theorem}
	There is a constant $c$ such that the following holds.
	With probability at least $1-2\exp(-c \tau^2 n)$, for every $x\in\mathcal{S}_r^\ast$, we have that
	\begin{align}
	\label{eq:quad.term}
	\left| \left\{ j : \frac{1}{2} Q_{x,x^\ast}(j) \geq  \frac{ r^2 }{8}  \right\} \right|
	&\geq (1-\tau)n ,\\
	\label{eq:multi.term}
	\left| \left\{ j : M_{x,x^\ast}(j) \geq -\frac{ r^2 }{16}  \right\} \right|
	&\geq (1-\tau)n.
	\end{align}
\end{lemma}

Let us show that Lemma \ref{prop:actually.the.theorem} implies Proposition \ref{prop:estimation.error}:

\begin{proof}[Proof of Proposition \ref{prop:estimation.error}]
	On the high probability event from Lemma \ref{prop:actually.the.theorem}, for every $x\in\mathcal{S}_r^\ast$, we have that
	\[\widehat{f}_{I_j}(x) - \widehat{f}_{I_j}(x^\ast) \geq \frac{ r^2}{16}  >0
	\quad\text{on at least } (1-2\tau)n \text{ blocks } j.\]
	Therefore, as $\tau<\frac{1}{4}$, on that event, $x^\ast$ wins the match against every $x\in \mathcal{S}_r^\ast$.

	The extension to all $x\in\mathcal{X}$ with $\|x-x^\ast\| \geq r$ is a simple consequence of convexity.
	Indeed, let $x\in\mathcal{X}$ with $\|x-x^\ast\| \geq r$ and set
	\[ y:=x^\ast + \frac{r}{\|x-x^\ast\| } (x-x^\ast) \in \mathcal{S}_r^\ast. \]
	Note now that (for every sample) $\widehat{f}_{I_j}(\cdot) - \widehat{f}_{I_j}(x^\ast)$ is a convex function which equals zero at $x^\ast$.
	Hence, if this function is strictly positive in $y$, then, by convexity, it is also strictly positive on
	\[\{x^\ast + t(y-x^\ast) : t\geq 1\}\cap\mathcal{X},\]
	which is the subset of the ray that originates from $x^\ast$ and passed through $y$, consisting of the points that are ``beyond'' $y$.
	Taking $t=\frac{1}{r}\|x-x^\ast\| $, we see that $x^\ast$ defeats $x$ (at least) on the same blocks on which it defeats $y$.
	In conclusion, on the high probability event of the lemma, $x^\ast$ wins the match against $x$.
\end{proof}

\begin{remark}
By Assumption \ref{ass:on.F.diff.convex.etc}, the functions $F$, $\nabla F$ and $\nabla^2 F$ are so-called Carath{\'e}odory functions and therefore are jointly measurable. % (see, e.g.\ \cite[Lemma 4.51]{aliprantis} for the simple proof of this fact).
Moreover, by Assumption \ref{ass:hessian.midpoints} and the dominated convergence theorem, one can readily verify that $f$ is twice continuously differentiable near $x^\ast$ with
\[\nabla f(x)=\E[\nabla F(x,\xi)]
\quad\text{and}\quad
\nabla^2 f(x)=\E[\nabla^2 F(x,\xi)]\]
for all $x\in\mathcal{X}$ with $\|x-x^\ast\|<r_0$.
In particular, from this we get that $\| \cdot \|=\E[\langle \nabla^2 F(x^\ast,\xi)\cdot,\cdot\rangle]^{\frac{1}{2}}$.
\end{remark}

\subsection{Estimation error, the quadratic term}
This subsection contains the proof of \eqref{eq:quad.term} from Lemma \ref{prop:actually.the.theorem}: we show that the quadratic term is likely to be at least of order $r^2$. The proof relies on the results of Section \ref{sec:mom.singular.value} and the strategy is the following.
In a first step, we ignore the fact that the Hessian in the definition of $Q_{x,x^\ast}$ is evaluated at a midpoint between $x^\ast$ and $x$, considering instead the Hessian evaluated at the optimizer $x^\ast$. 
We employ the median-of-mean-type lower bound on the smallest singular value of the random matrix $\nabla^2 F(x^\ast,\xi)$ established in Section \ref{sec:mom.singular.value},  which is summarized in the following lemma.

\begin{lemma}
\label{lem:hessian.at.optimzer}
	There are constants $s_1,c>0$ depending only on $L$ such that the following holds.
	With probability at least $1-2\exp(-c \tau N)$, for every $x\in\mathcal{S}_r^\ast$ and every choices of subsets $J_j \subseteq I_j$ with $|J_j|\leq s_1m$, we have that
	\begin{align}
	\label{eq:hessian.slb.at.optimzer}
	\left| \left\{ j : \frac{1}{m}\sum_{i\in I_j\setminus J_j} \langle  \nabla^2 F(x^\ast,\xi_i)(x-x^\ast), x-x^\ast \rangle \geq  \frac{ r^2}{2}  \right\} \right|
	&\geq \left( 1-\frac{\tau}{2}\right) n .
	\end{align}
\end{lemma}
\begin{proof}
	We apply Theorem \ref{thm:singular.value} with $A:=\nabla^2 F(x^\ast ,\xi)$. By Remark \ref{rem:relation.SLB.SBP.norm.equiv}, the random matrix $A$ satisfies a stable lower bound with parameters
	\[\left(m,\frac{\gamma}{2},2l,k\right)
	=\left(m,\frac{1}{4},2s_1m,s_2m\right)\]
	for constants $s_1,s_2$ that depend only on $L$.
	Moreover, once $\theta$ is sufficiently small, we have that
	\[ k
	=s_2m
	\geq \frac{s_2}{\theta}
	\geq \max\left\{4,2\log\left(\frac{4}{\tau}\right)\right\}, \]
	showing that assumption (a) in Theorem \ref{thm:singular.value} is satisfied. Lemma \ref{lem:growth.op.squared.A} and Lemma \ref{lem:operator.norm.relation} imply that
	\begin{align*}
	&\max\left\{ \| \E[ A \mathbb{A}^{-1} A] \|_{\mathrm{op}} , \frac{d m}{\tau k} \log\left(\frac{\log(3d) \E[ \| A \mathbb{A}^{-1} A \|_{\mathrm{op}} ] }{\gamma\tau \|\E[A \mathbb{A}^{-1} A]\|_{\mathrm{op}} } \right) \right\} \\
	&\leq \max\left\{ c_0 d , \frac{d s_2}{\tau } \log\left(\frac{\log(3d) 2 d }{\tau  }\right) \right\} \\
	&\leq c_0 d\log(2d)
	\end{align*}
	for a constant $c_0$ that depends only $L$ and $\tau$. And,
	as $\tau$ depends only on $L$, $c_0$ actually only depends only on $L$.

	Recall that Theorem \ref{thm:main.convex.intro} has the requirement that $N\geq c_2 d\log(2d)$ for a constant $c_2$ that may depend on $L$ and which we are free to choose to be as large as we went.
	Doing so shows that assumption (b) of Theorem \ref{thm:singular.value} holds as well.
	
Setting $c:=C_2 \min\{s_1,s_2\}$ where $C_2$ is the constant from Theorem \ref{thm:singular.value}, it follows from that theorem that \eqref{eq:hessian.slb.at.optimzer} holds with probability at least $1-2\exp(-c \tau N)$.
\end{proof}

From now on, we fix the constant $s_1$ from Lemma \ref{lem:hessian.at.optimzer}.
As $s_1$ depends only on $L$, all constants $\theta,\tau, c,c_1,\dots$ which are allowed to depend on $L$ may also depend on $s_1$.

In a next step, we show that replacing the Hessian at a midpoint with the Hessian at $x^\ast$ does not come at a high cost. Clearly, in Lemma \ref{lem:hessian.at.optimzer} we may arbitrarily modify / delete $s_1m$ coordinates from each block $j$, and we shall argue in the following that the errors
\[\mathcal{E}_{\mathrm{H},i}(x)
=\sup_{ t\in[0,1] } \left| \left\langle \left(\nabla^2 F(x^\ast + t(x-x^\ast) ,\xi_i) -\nabla^2 F(x^\ast ,\xi_i) \right) (x-x^\ast), x-x^\ast \right\rangle \right| \]
are well behaved on the remaining blocks.
As before, the proof has two components.
In the first one, we analyze what happens for a single $x\in\mathcal{S}_r^\ast$.
Here the error is governed by the probability that $\mathcal{E}_{\mathrm{H}}(x)$ is large, where we  recall that $r< r_0$ and by Assumption \ref{ass:hessian.midpoints}
\[\sup_{x\in\mathcal{S}_r^\ast} \P\left[ \mathcal{E}_{\mathrm{H}}(x) \geq \frac{ r^2 }{8}\right]\leq c_1
\]
for a constant $c_1$ which we are free to choose depending on $L$---hence, $c_1$ may depend on $s_1$.

\begin{lemma}
\label{lem:hessian.midpoints.single.x.single.block}
	There is a constant $c$ such that, for every $x\in\mathcal{S}_r^\ast$ and every $j$, with probability at least $1-2\exp(-c m )$, we have that
	\begin{align*}
	\left| \left\{ i\in I_j : 	\mathcal{E}_{\mathrm{H},i}(x) \leq  \frac{ r^2}{8}  \right\}  \right|
	& \geq \left(1-\frac{s_1}{2}\right) m .
	\end{align*}
\end{lemma}
\begin{proof}
	Fix some $x\in\mathcal{S}_r^\ast$.
	Setting $c_1:=\frac{s_1}{4}$, Assumption \ref{ass:hessian.midpoints} implies that
	\[\delta:=\P\left[\mathcal{E}_{\mathrm{H}}(x)\geq \frac{ r^2}{8}\right]
	\leq \frac{s_1}{4}.\]
	If $\delta=0$ there is nothing to prove, so assume otherwise and apply the Binomial concentration inequality \cite[Corollary 2.11]{boucheron2013concentration}:
	there is an absolute constant $C>0$ such that, for every $u\geq 0$, with probability at least $1-2\exp(-Cm\delta\min\{u,u^2\})$, we have that
	\[ \left|\left\{i\in I_j :\mathcal{E}_{\mathrm{H},i}(x) \geq\frac{r^2}{8} \right\} \right| \leq m\delta(1+u) . \]
	Applying this to $u:=\frac{s_1}{2\delta}-1\geq 1$ completes the proof (with $c=\frac{Cs_1}{4}$).
\end{proof}

Next, let us show that most blocks have many indices $i$ for which the errors $\mathcal{E}_{\mathrm{H},i}(x)$ are small.

\begin{lemma}
\label{lem:hessian.midpoints.single.x.all.blocks}
	There is a constant $c$ such that, for every $x\in\mathcal{S}_r^\ast$, with probability at least $1-2\exp(-c\tau N)$, we have that
	\begin{align*}
	\left| \left\{ j : \left| \left\{ i\in I_j : 	\mathcal{E}_{\mathrm{H},i}(x) \leq  \frac{ r^2}{8}   \right\}  \right| \geq \left(1-\frac{s_1}{2}\right)m \right\} \right|
	&\geq \left(1-\frac{\tau}{4}\right) n.
	\end{align*}
	In particular, the statement holds uniformly over sets $\bar{\mathcal{S}}_r^\ast\subseteq\mathcal{S}_r^\ast$ of cardinality at most $\log ( |\bar{\mathcal{S}}_r^\ast| ) \leq \frac{1}{2}c\tau N$.
\end{lemma}
\begin{proof}
	Fix $x\in\mathcal{S}_r^\ast$ and, for every $j$, set
	\[ \delta(j):=% 1_{ |\{ i\in I_j : \mathcal{E}_{\mathrm{H},i}(x) \leq  \frac{1}{8}r^2  \}  | \geq (1-\frac{1}{2}s_1)m }.
	\begin{cases}
	1, &\text{if } |\{ i\in I_j : \mathcal{E}_{\mathrm{H},i}(x) \leq  \frac{1}{8}r^2  \}  | \geq (1-\frac{1}{2}s_1)m \\
	0, &\text{otherwise}.
	\end{cases}
	\]
	By Lemma \ref{lem:hessian.midpoints.single.x.single.block} we have that
	\[\delta
	:=\P[\delta(j)=1]
	\leq 2\exp(-c_0m)\]
	for a constant $c_0$.
	
	Now recall that $m\geq \frac{1}{\theta}$ and therefore $\delta\leq \frac{3}{4}$ once $\theta$ is small enough (i.e.\  $\theta\leq \frac{1}{c_0}\log(\frac{8}{3})$ suffices).
	By Bennett's inequality, exactly as in the proof of Lemma \ref{lem:slb.uniform}, we have that for every $u\geq 2$, with probability at least $1-2\exp(-C\delta n  u\log (u) )$,
	\[ |\{j : \delta(j)=1\}| \leq u\delta n. \]
	To complete the proof, set $u:=\frac{\tau}{4\delta}$ so that $u\geq \exp(\frac{c_0m}{2})\geq 2$, again once $\theta$ is small enough.
\end{proof}

The final step is to extend the outcome of the lemma from the net $\bar{\mathcal{S}}_r^\ast$ to the whole $\mathcal{S}_r^\ast$. To that end, recall that by Assumption \ref{ass:hessian.midpoints}
\[ \| \nabla^2 F(x,\xi) - \nabla^2 F(y,\xi) \|_{\mathrm{op}}
\leq \|x-y\|^\alpha  K(\xi) \]
for all $x,y\in\mathcal{B}_r^\ast$.

\begin{remark}
In what follows we shall, from time to time, divide by $\E[K(\xi)]$; therefore, we shall assume without loss of generality that $\E[K(\xi)]>0$. Note that if $\E[K(\xi)]=0$ then the Hessian is constant in $\mathcal{B}_r^\ast$ and there is nothing to prove: $\mathcal{E}_{\mathrm{H}}(x)=0$ for every $x\in\mathcal{B}_r^\ast$.
\end{remark}

\begin{lemma}
\label{lem:quadratic.net}
	Let $c_0$ be the constant of Lemma \ref{lem:hessian.midpoints.single.x.all.blocks} and let $C_0$ be an absolute constant to be specified later. 	Set
	\[
	\rho:= \left( \frac{C_0 \tau s_1 }{  r_0^\alpha \E[K(\xi)]} \right)^{\frac{1}{\alpha}}.
	\]
	Then $\mathcal{S}_r^\ast$ contains a $\rho r$-cover with respect to the norm $\| \cdot \|$, which is denoted by  $\bar{\mathcal{S}}_r^\ast$, and
	\[
	\log ( |\bar{\mathcal{S}}_r^\ast| )\leq \frac{1}{2}c_0\tau N.
	\]
\end{lemma}

\begin{proof}
	If $\rho\geq 1$ there is nothing to prove; thus we may assume without loss of generality that $0<\rho<1$.
	By a standard volumetric argument (see e.g.\ \cite[Exercise 2.2.14]{talagrand2014upper}) there is a $\rho r$ cover
	$\bar{\mathcal{S}}_r^\ast\subseteq \mathcal{S}_r^\ast$, and
	\begin{align*}
	\log (  |\bar{\mathcal{S}}_r^\ast|)
	&\leq d \log \left( \frac{12}{ \rho } \right) \\
	&= \frac{d}{\alpha} \log \left(  \frac{r_0^\alpha \E[K(\xi)] }{ C_0 \tau s_1}   \right).
	\end{align*}
	By assumption we have
	\[ N\geq c_2 N_{\mathrm{H},\mathcal{E}}
	\equiv c_{2} \frac{d}{\alpha} \log \left( r_0^\alpha \E[K(\xi)] +2\right)  \]
for a constant $c_{2}$ which we are free to choose large enough. Then $\log ( |\bar{\mathcal{S}}_r^\ast| )\leq \frac{1}{2}c_0\tau N $, as claimed.
\end{proof}

\begin{lemma}
\label{lem:hessian.midpoints.uniform}	
	Let $\bar{\mathcal{S}}_r^\ast$ be as in Lemma \ref{lem:quadratic.net}, and for every $x \in \mathcal{S}_r^\ast$ set $y=y(x)$ to be the nearest point to $x$ in $\bar{\mathcal{S}}_r^\ast$.

	There is an absolute constant $C$ such that the following holds.
	With probability at least $1-2\exp(-C\tau^2n)$, for every $x\in\mathcal{S}_r^\ast$, we have that
	\[ \left| \left\{ j : \left| \left\{ i\in I_j : |\mathcal{E}_{\mathrm{H},i}(x)- \mathcal{E}_{\mathrm{H},i}(y)| \geq \frac{r^2}{8} \right\}\right| \leq \frac{s_1m}{2} \right\} \right| \geq \left(1-\frac{\tau}{4}\right)n .\]
\end{lemma}
\begin{proof}
	Let
	\[ \Psi := \sup_{x\in \mathcal{S}_r^\ast} \frac{1}{n}\sum_{j=1}^n 1_{ | \{ i\in I_j : |  \mathcal{E}_{\mathrm{H},i}(x) - \mathcal{E}_{\mathrm{H},i}(y) | \geq \frac{1}{8}r^2 \} | \geq \frac{1}{2}s_1m }\]
	and observe that it suffices to show that $\P[\Psi\geq \frac{\tau}{4}]\leq 2\exp(-C\tau^2n)$. 	
	By the bounded differences inequality \cite[Theorem 6.2]{boucheron2013concentration}, we have that for every $u\geq 0$,
	\[\P[\Psi\geq \E[\Psi]+u]\leq 2\exp(-Cnu^2), \]
	and setting $u:=\frac{\tau}{8}$, all that remains to show is that $\E[\Psi]\leq \frac{\tau}{8}$.
	
	Note that
	\begin{align*}
	1_{ | \{ i\in I_j : |  \mathcal{E}_{\mathrm{H},i}(x) - \mathcal{E}_{\mathrm{H},i}(y) | \geq \frac{1}{8}r^2 \} | \geq \frac{1}{2}s_1m }	
	&\leq \frac{16}{r^2 s_1} \frac{1}{m}\sum_{i\in I_j} |  \mathcal{E}_{\mathrm{H},i}(x) - \mathcal{E}_{\mathrm{H},i}(y) |	
	\end{align*}
	for every $j$; hence,
	\begin{align}
	\label{eq:error.hessian.midpoints.psi.estimate}
	\Psi
	&\leq  \frac{16}{r^2 s_1} \sup_{x\in \mathcal{S}_r^\ast}\frac{1}{N}\sum_{i=1}^N  |  \mathcal{E}_{\mathrm{H},i}(x) - \mathcal{E}_{\mathrm{H},i}(y) |.
	\end{align}
	To control the difference of the $\mathcal{E}$'s, for simpler notation, for $(t,z)\in [0,1]\times\mathcal{S}_r^\ast$ set
	\[ A^{tz}_i:=\nabla^2 F(x^\ast +t(z-x^\ast) ,\xi_i) -\nabla^2 F(x^\ast ,\xi_i), \]
	and observe that
	\[ \mathcal{E}_{\mathrm{H},i}(z)
	=\sup_{t\in[0,1]} |\langle A^{tz}_i(z-x^\ast),z-x^\ast\rangle| \]
	for every $z\in\mathcal{S}_r^\ast$.
	Now, for every $t\in[0,1]$ and every $x\in\mathcal{S}_r^\ast$ (and $y=y(x)$),
	\begin{align*}
	&\left| \langle A^{tx}_i(x-x^\ast),x-x^\ast\rangle  - \langle A^{ty}_i(y-x^\ast),y-x^\ast\rangle  \right| \\
	&= \left| \langle A^{tx}_i(x-x^\ast+y-x^\ast),x-y \rangle - \langle (A^{ty}_i-A^{tx}_i)(y-x^\ast),y-x^\ast\rangle \right|  \\
	&\leq  2 \rho r^2 \| A^{tx}_i \|_{\mathrm{op}} + r^2 \| A^{ty}_i-A^{tx}_i \|_{\mathrm{op}},
	\end{align*}	
	where the last inequality holds by definition of the operator norm and by noting that $\|x-y\|\leq \rho r $,  $\|x-x^\ast+y-x^\ast\|\leq 2r$, and $\|y-x^\ast\|\leq r$.
	Invoking the subadditivity of ``$\sup_t(\cdot)$'' and the triangle inequality,
	\begin{align}
	\nonumber
 	|\mathcal{E}_{\mathrm{H}_i}(x) - \mathcal{E}_{\mathrm{H}_i}(x) |
	&\leq 2 \rho r^2\sup_{t\in[0,1]} \| A^{tx}_i \|_{\mathrm{op}}  + r^2 \sup_{t\in[0,1]}\| A^{ty}_i - A^{tx}_i \|_{\mathrm{op}} \\
	\nonumber	
	&\leq 2 \rho r^2 r^\alpha K(\xi_i) + r^{2}(r\rho)^\alpha  K(\xi_i) \\
	\label{eq:E.H.control.hoelder}
	&\leq 3 r^2(\rho r)^\alpha  K(\xi_i)
	\end{align}
	where the second inequality holds by Assumption \ref{ass:hessian.midpoints} on the continuity of the Hessian and as we may assume without loss of generality that $\rho\leq 1$ (and hence $\rho\leq \rho^\alpha$).
	
	Plugging \eqref{eq:E.H.control.hoelder} into \eqref{eq:error.hessian.midpoints.psi.estimate} implies that
	\begin{align*}
	\Psi
	&\leq  \frac{16}{  r^2 s_1}  3 r^2(\rho r)^\alpha \frac{1}{N}\sum_{i=1}^N  K(\xi_i)
	\end{align*}	
	and therefore
	\begin{align*}
	\E[\Psi ]
	&\leq \frac{48 \rho^\alpha r_0^\alpha\E[ K(\xi)]  }{s_1}.
	\end{align*}	
	Setting
	\[ \rho= \left( \frac{C_0 \tau s_1 }{  r_0^\alpha \E[K(\xi)]} \right)^{\frac{1}{\alpha}}\]
 just as in Lemma  \ref{lem:quadratic.net}, and recalling that we are free to choose $C_0$ small enough, it follows that $\E[\Psi ]\leq \frac{\tau}{8}$, as required.
\end{proof}

We are now ready for the 	

\begin{proof}[Proof of Proposition \ref{prop:actually.the.theorem}, quadratic part]
	For every $x\in\mathcal{S}_r^\ast$ and $j$, set
	\[ J_j^\ast(x):=\{ \text{largest $s_1m$ coordinates of } (\mathcal{E}_{\mathrm{H},i}(x))_{i\in I_j}\}\subseteq I_j.\]
	For every $j$, recalling the definition of $\mathcal{E}_{\mathrm{H}}(x)$ and the fact that $\nabla^2 F(z,\xi)$ is positive semidefinite, it follows that
	\begin{align*}
	Q_{x,x^\ast}(j)
	&\geq \frac{1}{m}\sum_{i\in I_j\setminus J_j^\ast(x)} \inf_{t\in[0,1]} \left\langle  \nabla^2 F(x^\ast + t(x-x^\ast),\xi_i)(x-x^\ast),x-x^\ast\right\rangle \\
	&\geq  \frac{1}{m}\sum_{i\in I_j\setminus J_j^\ast(x)} \left\langle  \nabla^2 F(x^\ast ,\xi_i)(x-x^\ast),x-x^\ast \right\rangle - \frac{1}{m}\sum_{i\in I_j\setminus J_j^\ast(x)} \mathcal{E}_{\mathrm{H},i}(x)	\\
	&=: A_x(j) + B_x(j)
	\end{align*}
	As $|J_j^\ast(x)|\leq s_1m$ for every $x\in\mathcal{S}_r^\ast$ by definition, Lemma \ref{lem:hessian.at.optimzer} implies that,  with probability at least $1-2\exp(-c\tau n)$, for every $x\in\mathcal{S}_r^\ast$, we have that
	\[A_x(j)\geq \frac{r^2}{2}
	\quad\text{on more than } \left(1-\frac{\tau}{2}\right)n \text{ of the blocks } j.\]
	Moreover, combining Lemma \ref{lem:hessian.midpoints.single.x.all.blocks} and Lemma \ref{lem:hessian.midpoints.uniform} implies that,  with probability at least $1-2\exp(-c\tau^2 n)$, for every $x\in\mathcal{S}_r^\ast$, we have that
	\begin{align*}
	\left| \left\{ j : \left| \left\{ i\in I_j : 	\mathcal{E}_{\mathrm{H},i}(x) \leq  \frac{r^2}{4}   \right\}  \right| \geq (1-s_1)m \right\} \right|
	&\geq \left(1-\frac{\tau}{2}\right) n,
	\end{align*}
	and on that event clearly
	\[B_x(j)\geq -\frac{r^2}{4}
		\quad\text{on more than } \left(1-\frac{\tau}{2}\right)n \text{ of the blocks } j.\]
	In particular, combining the estimates on $A_x(j)$ and $B_x(j)$ gives: with probability at least $1-2\exp(-c\tau^2 n)$, for every $x\in\mathcal{S}_r^\ast$, we have that
	\[Q_{x,x^\ast}(j) \geq \frac{r^2}{4}
	\quad\text{on more than } (1-\tau)n \text{ of the blocks } j.\]
	This completes the proof.
\end{proof}

\subsection{Estimation error, the multiplier term}
This subsection contains the proof of \eqref{eq:multi.term} from Lemma \ref{prop:actually.the.theorem}, stating that the multiplier term
\[
M_{x,x^\ast}(j)=\frac{1}{m}\sum_{i\in I_j} \langle \nabla F(x^\ast ,\xi_i), x-x^\ast \rangle
\]
is likely to be at most of order $r^2$.
To ease notation, set
\begin{align*}
	\mathbb{H}
	&:= \nabla^2 f(x^\ast)
	=\E[\nabla^2 F(x^\ast,\xi)],\\
	\mathbb{G}
	&:= \C[\nabla F(x^\ast,\xi)],
\end{align*}
\i.e., $\mathbb{H}$ is the Hessian of $f$ at $x^\ast$ and $\mathbb{G}$ is the covariance matrix of the gradient of $F(\cdot,\xi)$ at $x^\ast$.
In particular, a straightforward computation shows that
\begin{align}
\label{eq:N.G.expression}
N_{\mathrm{G}}(r)
&=\frac{1}{r^2}\mathrm{trace}(\mathbb{H}^{-1}\mathbb{G})
=\frac{1}{r^2}\mathrm{trace}( \mathbb{H}^{-\frac{1}{2}} \mathbb{G} \mathbb{H}^{-\frac{1}{2}} ),\\
\label{eq:sigma.expression}
\begin{split}
\sigma^2
&= \lambda_{\max}( \mathbb{H}^{-1} \mathbb{G} )  \\
&= \lambda_{\max}( \mathbb{H}^{-\frac{1}{2}} \mathbb{G} \mathbb{H}^{-\frac{1}{2}} )
= \| \mathbb{G} \|_{\mathrm{op}} .
\end{split}
\end{align}
As before, we analyze what happens for a single $x\in\mathcal{S}_r^\ast$; the high probability estimate we obtain allows us to control a net in $\mathcal{S}_r^\ast$; we then show that passing from the net to the entire set $\mathcal{S}_r^\ast$ does not distort the outcome by too much.

\begin{lemma}
\label{lem:muliplier.single.x}
	There is an absolute constant $C$ such that, for every $x\in\mathcal{S}_r^\ast$, with probability at least $1-2\exp(-C\tau^2n)$, we have that
	\[ \left| \left\{ j : M_{x,x^\ast}(j) \geq -\frac{  r^2}{32}  \right\} \right|
	\geq \left(1-\frac{\tau}{2}\right) n. \]
	In particular, the statement holds uniformly over a set $\bar{\mathcal{S}}_r^\ast\subseteq \mathcal{S}_r^\ast$ of cardinality at most $\log ( \frac{1}{2}|\bar{\mathcal{S}}_r^\ast|) \leq \frac{1}{2} C\tau^2n $.
\end{lemma}
\begin{proof}
	Fix some $x\in \mathcal{S}_r^\ast$ and define for $1 \leq i \leq N$,
	\begin{align*}
	U_{i}
	&:=\langle \nabla F(x^\ast,\xi_i) - \E[\nabla F(x^\ast,\xi)], x-x^\ast\rangle\\
	&= \langle \nabla F(x^\ast,\xi_i) , x-x^\ast\rangle  -  \langle \nabla f(x^\ast), x-x^\ast\rangle.
	\end{align*}
	If $x^\ast$ lies in the interior of $\mathcal{X}$, the first order condition for optimality implies that $\langle \nabla f(x^\ast), x-x^\ast\rangle$ equals zero.
	In general, the first order condition implies that this term is non-negative.
	In either case, we get that
	\begin{align*}
	\left| \frac{1}{m}\sum_{i\in I_j} U_i \right|
	&\leq \frac{  r^2}{32}
	\quad\text{implies}\quad
	M_{x,x^\ast}(j) \geq -\frac{ r^2}{32}
	\end{align*}
	and we shall show that the former happens on most blocks.
	
	To that end, consider first a single block $j$.
	As the $U_i$'s are i.i.d.\ zero mean random variables, Markov's inequality together with the Cauchy-Schwartz inequality implies that
	\begin{align}
	\label{eq:multiplier.single.markov}
	\begin{split}
	\P\left[ \left| \frac{1}{m}\sum_{i\in I_j} U_i\right| \geq \frac{ r^2}{32} \right]
	&\leq \frac{32}{ r^2 } \E\left[ \left| \frac{1}{m} \sum_{i\in I_j} U_i \right| \right]\\
	&\leq \frac{32}{ r^2 }  \left( \frac{ \E[| U_1|^2] }{ m } \right)^{\frac{1}{2}}.
	\end{split}
	\end{align}
	By the definition of $\sigma$ (or rather, by the alternative expression in \eqref{eq:sigma.expression}) and as $\|x-x^\ast\|=r$, we have that
	\begin{align}
	\label{eq:multiplier.single.sigma.max}
	\begin{split}
	\E[|U_1|^2]
	&= \langle \mathbb{G} (x-x^\ast) ,x-x^\ast \rangle \\
	&\leq r^2 \sigma^2.
	\end{split}
	\end{align}
	Combining \eqref{eq:multiplier.single.markov} and \eqref{eq:multiplier.single.sigma.max} and using that $m\geq \frac{\sigma^2}{\theta r^2}$, we conclude that
	\begin{align*}
	\P\left[ \left| \frac{1}{m}\sum_{i\in I_j} U_i\right| \geq \frac{ r^2}{32} \right]
	&\leq \frac{ 32 r \sigma  }{ r^2 \sqrt{m} } \\
	&\leq  32 \sqrt{\theta}
	\leq \frac{\tau}{4}
	\end{align*}
	as soon as $\theta$ is sufficiently small.

	The claim now follows from a Binomial estimate---just as in the proof of Lemma \ref{lem:hessian.midpoints.single.x.single.block}:
	the probability that a single block $j$ has the wanted property is at least $1-\frac{\tau}{4}$ (by the above); therefore, the probability that the number of desirable blocks $j$ is smaller than the mean (which is at least $(1-\frac{\tau}{4})n$) by more than $\frac{\tau}{4}n$ is at most $2\exp(-Cn\tau^2)$.
\end{proof}

\begin{lemma}
\label{lem:multiplier.good.spearated.set}
	Let $C_0$ be the absolute constant of Lemma \ref{lem:muliplier.single.x}.
	Then there is an absolute constant $C_1$ and a set $\bar{\mathcal{S}}_r^\ast\subseteq \mathcal{S}_r^\ast$ with cardinality $\log ( \frac{1}{2}|\bar{\mathcal{S}}_r^\ast|) \leq \frac{1}{2}C_0n\tau^2$ such that the following holds. 
	Let
	\[
	\rho:=\left( \frac{ \mathrm{trace}(\mathbb{H}^{-1} \mathbb{G} ) }{ C_1 \tau^2n }\right)^{\frac{1}{2}}.
	\]
	For every $x\in\mathcal{S}_r^\ast$ there is $y=y(x)\in \bar{\mathcal{S}}_r^\ast$ with
	\begin{align}
	\label{eq:multiplier.good.set.covaraince}
	\langle \mathbb{G} (x-y),x-y\rangle^{\frac{1}{2}}
	&\leq 2 \rho r \quad\text{and } , \\
	\label{eq:multiplier.good.set.gradient}
	\langle \nabla f(x^\ast),x-y\rangle
	&\geq 0.
	\end{align}
\end{lemma}
\begin{proof}
	As a first step, we ignore \eqref{eq:multiplier.good.set.gradient} and construct a set $\tilde{\mathcal{S}}_r^\ast$ with log-cardinality satisfying $\log ( \frac{1}{2} |\tilde{\mathcal{S}}_r^\ast|) \leq \frac{1}{2} C_0n\tau^2$ such that \eqref{eq:multiplier.good.set.covaraince} holds with $\rho r$ instead of $ 2\rho r$.
	To that end, observe that covering the sphere $\{ x\in\mathbb{R}^d :  \langle \mathbb{H}x,x\rangle =1\}$ w.r.t.\ to the norm endowed by $\mathbb{G}$ is equivalent to covering the Euclidean sphere  $\{ x\in\mathbb{R}^d :  \langle x,x\rangle =1\}$ w.r.t.\ the norm endowed by $\mathbb{H}^{-\frac{1}{2}}\mathbb{G}\mathbb{H}^{-\frac{1}{2}}$.
	Hence, denoting by $\mathcal{G}$ the standard Gaussian vector in $\mathbb{R}^d$, the dual Sudakov inequality 
	%(see e.g.\ \cite[Proposition 1]{lugosi2019sub}) 
	(see e.g.\ \cite[Theorem 3.18]{ledoux2013probability})
	guarantees the existence of a $\rho r$ cover  of $\tilde{\mathcal{S}}_r^\ast\subseteq\mathcal{S}_r^\ast$ with respect to the norm $\langle \mathbb{G}\cdot,\cdot\rangle^{\frac{1}{2}}$ such that
	\[ \log \left( \frac{|\tilde{\mathcal{S}}_r^\ast|}{2} \right)
	\leq  C_2 \left( \frac{ \E[ \langle \mathbb{H}^{-\frac{1}{2}}\mathbb{G}\mathbb{H}^{-\frac{1}{2}}  \mathcal{G},\mathcal{G}\rangle^{\frac{1}{2}}] }{ \rho } \right)^2;  \]
	thus, for every $x\in\mathcal{S}^\ast_r$ there is $y=y(x)\in\tilde{\mathcal{S}}_r^\ast$ with $\langle \mathbb{G}(y-x),y-x\rangle^{\frac{1}{2}} \leq \rho r$.	
	Observe that
	\begin{align*}
	\E[ \langle \mathbb{H}^{-\frac{1}{2}} \mathbb{G} \mathbb{H}^{-\frac{1}{2}} \mathcal{G}, \mathcal{G} \rangle^{\frac{1}{2}}]^2
	&= \E[ \| (\mathbb{H}^{-\frac{1}{2}} \mathbb{G} \mathbb{H}^{-\frac{1}{2}})^{\frac{1}{2}} \mathcal{G}\|_2 ]^2 \\
	&\leq \E[ \| (\mathbb{H}^{-\frac{1}{2}} \mathbb{G} \mathbb{H}^{-\frac{1}{2}})^{\frac{1}{2}} \mathcal{G} \|_2^2 ] \\
	&= \mathop{\mathrm{trace}}(\mathbb{H}^{-\frac{1}{2}} \mathbb{G} \mathbb{H}^{-\frac{1}{2}})
	=\mathrm{trace}(\mathbb{H}^{-1} \mathbb{G} ),
	\end{align*}
	(where the last equality was already observed in \eqref{eq:N.G.expression}).
	Thus,
	\begin{align*}
	\begin{split}
	\log  \left( \frac{|\tilde{\mathcal{S}}_r^\ast|}{2} \right)
	&\leq  \frac{C_2 \mathrm{trace}(\mathbb{H}^{-1} \mathbb{G} ) }{ \rho^2 } \\
	&= C_2 C_1  \tau^2 n
	\leq \frac{1 }{2} C_0  \tau^2 n,
	\end{split}
	\end{align*}
	once $C_1$ is sufficiency small.
	
	Next, we modify $\tilde{\mathcal{S}}_r^\ast$, ensuring that both equations \eqref{eq:multiplier.good.set.covaraince} and \eqref{eq:multiplier.good.set.gradient} hold:
	for every $z\in \tilde{\mathcal{S}}_r^\ast$, pick some
	\[y(z) \in \mathop{\mathrm{argmin}} \left\{ \langle \nabla f(x^\ast),y\rangle : y \in \mathcal{S}_r^\ast \text{ s.t.\ } \langle \mathbb{G}(y-z),(y-z)\rangle^{\frac{1}{2}} \leq \rho r  \right\};\]
thus, $y(z)$ is the minimizer of $\langle \nabla f(x^\ast),y\rangle$ with the $\rho r$-ball (with respect to the norm $\langle\mathbb{G}\cdot,\cdot\rangle^{\frac{1}{2}}$) centered at $z$.

	It is straightforward to verify that the set
	\[\bar{\mathcal{S}}_r^\ast:=\{ y(z) : z\in \tilde{\mathcal{S}}_r^\ast\} \]
	satisfies the statement of the lemma.
\end{proof}

\begin{lemma}
\label{lem:multiplier.uniform.x}
	Let $\bar{\mathcal{S}}_r^\ast$ and $y=y(x)$ be as in Lemma \ref{lem:multiplier.good.spearated.set}.
	There is an absolute constant $C$ such that the following holds.
	With probability at least $1-2\exp(-C\tau^2n)$, for every $x\in\mathcal{S}_r^\ast$,  we have that
	\[ \left| \left\{ j :  M_{x,x^\ast}(j)\geq M_{y,x^\ast}(j) - \frac{ r^2}{32} \right\} \right| 
	\geq \left(1-\frac{\tau}{2}\right)n. \]
\end{lemma}
\begin{proof}
	For every $x\in\mathcal{X}$ and $j$, set
	\begin{align*}
	\Delta_{x}(j)
 	&:= M_{x,x^\ast}(j) -  M_{y,x^\ast}(j) \\
 	\bar{\Delta}_x(j)
	&:= \Delta_x(j) - \E[ \Delta_x(j)]\\
	&=\frac{1}{m}\sum_{i\in I_j} \langle \nabla F(x^\ast,\xi_i) - \E[\nabla F(x^\ast,\xi)], x-y \rangle.
	\end{align*}
	Recalling that $\E[\Delta_{x}(j)]\geq 0$ by Lemma \ref{lem:multiplier.good.spearated.set} and setting
	\[ \Psi := \sup_{x\in \mathcal{S}_r^\ast} \frac{1}{n}\sum_{j=1}^n 1_{ \{ | \bar{\Delta}_x(j)|\geq \frac{1}{32}  r^2 \} },\]
	the statement of the lemma therefore follows if $\Psi\leq \frac{\tau}{2}$ holds with probability at least $1-2\exp(-C\tau^2n)$.
	
	To that end, we once more rely on the bounded difference inequality \cite[Theorem 6.2]{boucheron2013concentration}: for every $u\geq 0$, we have that
	\[\P[\Psi\geq \E[\Psi]+u]\leq 2\exp(-Cnu^2). \]
	Setting $u:=\frac{\tau}{4}$, all that is left is to show that $\E[\Psi]\leq \frac{\tau}{4}$.
		
	First, observe that $1_{|a|\geq b} \leq \frac{1}{b} |a|$. Hence,
	\begin{align*}
	\Psi
	&\leq \frac{32}{ r^2} \sup_{x\in \mathcal{S}_r^\ast} \frac{1}{n}\sum_{j=1}^n | \bar{\Delta}_{x}(j)|\\
	&\leq \frac{32}{ r^2} \left( \sup_{x\in\mathcal{S}_r^\ast} \E[ | \bar{\Delta}_{x}(1) | ] + \sup_{x\in \mathcal{S}_r^\ast} \frac{1}{n}\sum_{j=1}^n \left| |\bar{\Delta}_{x}(j)| - \E[|\bar{\Delta}_{x}(j)|] \right| \right)\\
	&=:\frac{32}{ r^2} \left( R_1 + R_2\right).
	\end{align*}
	In particular $\E[\Psi]\leq \frac{32}{ r^2} ( R_1 + \E[R_2])$.
	
	To estimate $R_1$, recall that $\bar{\Delta}_x(j)$ is a sum of independent, zero mean random variables. Thus,
	\begin{align*}
	\E[| \bar{\Delta}_{x}(1)| ]
	&\leq \left( \frac{ \langle \mathbb{G}(x-y),x-y\rangle }{m}\right)^{\frac{1}{2}} \\
	&\leq  2 r \left( \frac{\mathrm{trace}(\mathbb{H}^{-1}\mathbb{G}) }{ C_1 \tau^2 n m } \right)^{\frac{1}{2}}
	\end{align*}
	for every $x\in\mathcal{S}_r^\ast$, where the second inequality follows from Lemma \ref{lem:multiplier.good.spearated.set}.
	Recalling that, by our assumptions,
	\[ n m
	=N
	\geq c_2 N_{\mathrm{G}}(r)
	\equiv \frac{c_2 \mathrm{trace}(\mathbb{H}^{-1}\mathbb{G}) }{r^2}, \]
	we conclude that $\E[| \bar{\Delta}_{x}(1)| ]\leq  \frac{2r^2}{ \tau \sqrt{ C_1 c_{2}} }$.
	Therefore,
	\[R_1\leq \frac{2r^2}{\tau\sqrt{C_1 c_{2}} } .\]

	Turning to $R_2$, by symmetrization, the contraction theorem for Rademacher processes, and de-symmetrization (see e.g.\ \cite[Section 11.3]{boucheron2013concentration}), it is evident that
	\begin{align*}
	\E[R_2]
	&\leq 2\E\left[ \sup_{x\in \mathcal{S}_r^\ast} \left| \frac{1}{n}\sum_{j=1}^n \varepsilon_j |\bar{\Delta}_{x}(j)| \right| \right]\\
	&\leq 4 \E\left[ \sup_{x\in \mathcal{S}_r^\ast } \left| \frac{1}{n}\sum_{j=1}^n \varepsilon_j \bar{\Delta}_{x}(j) \right| \right] \\
	&\leq 8 \E \left[ \sup_{x\in \mathcal{S}_r^\ast } \left| \frac{1}{n}\sum_{j=1}^n \bar{\Delta}_{x}(j) \right| \right].
	\end{align*}
	Moreover, by the definition of $\bar{\Delta}_x(j)$,
	\begin{align}
	\label{eq:multiplier.bounde.diff}
	\E[R_2]
	&\leq 8 \E\left[ \sup_{x\in \mathcal{S}_r^\ast } \left| \frac{1}{N}\sum_{i=1}^N  \left\langle \nabla F(x^\ast,\xi_i) - \E[\nabla F(x^\ast,\xi)], x-y \right\rangle  \right| \right].
	\end{align}
	Recall that $\|\cdot\|=\|\mathbb{H}^{\frac{1}{2}} \cdot\|_2$ and that $\|x-y\|\leq 2r$.
	Therefore, for any $z\in\mathbb{R}^d$,
	\begin{align*}
	\sup_{x\in \mathcal{S}_r^\ast } | \langle z, x-y \rangle|
	&= \sup_{x\in \mathcal{S}_r^\ast } | \langle  \mathbb{H}^{-\frac{1}{2}}z, \mathbb{H}^{\frac{1}{2}} ( x-y) \rangle| \\
	&\leq 2r \| \mathbb{H}^{-\frac{1}{2}} z\|_2.
	\end{align*}
	Together with linearity of $z\mapsto \langle z,x-y\rangle$ and \eqref{eq:multiplier.bounde.diff},
	\begin{align*}
	\E[R_2]
	&\leq 16r \E\left[ \left\| \frac{1}{N}\sum_{i=1}^N  \mathbb{H}^{-\frac{1}{2}}  \left( \nabla F(x^\ast,\xi_i) - \E[\nabla F(x^\ast,\xi)] \right) \right\|_2 \right]\\
	&\leq 16r \left( \frac{ \mathrm{trace}(\C[\mathbb{H}^{-\frac{1}{2}}\nabla F(x^\ast,\xi)]) } {N} \right)^{\frac{1}{2}}.
	\end{align*}
	It remains to observe that
	\begin{align*}
	\mathrm{trace}(\C[\mathbb{H}^{-\frac{1}{2}} \nabla F(x^\ast,\xi)])
	&= \mathrm{trace}(\mathbb{H}^{-\frac{1}{2}} \mathbb{G} \mathbb{H}^{-\frac{1}{2}})\\
	&= \mathrm{trace}(\mathbb{H}^{-1} \mathbb{G}) .
	\end{align*}
	Since $N\geq  c_2N_{\mathrm{G}}(r)\equiv \frac{c_2 }{r^2}\mathrm{trace}(\mathbb{H}^{-1}\mathbb{G})$, it is evident that
	\[ \E[R_2]\leq \frac{16r^2 }{\sqrt{c_{2}}},\]
	and combining the two estimates,
	\begin{align*}
	\E[\Psi]
	&\leq \frac{32}{r^2}\left( \frac{2r^2}{\tau\sqrt{C_1 c_{2}} } + \frac{16r^2 }{\sqrt{c_{2}}} \right)
	\leq \frac{\tau}{4},
	\end{align*}
	where the last inequality holds as soon as $c_{2}$ is large enough.
\end{proof}

\begin{proof}[Proof of Proposition \ref{prop:actually.the.theorem}, multiplier part]
	Let $\bar{\mathcal{S}}_r^\ast$ be the cover defined in Lemma \ref{lem:multiplier.good.spearated.set}.
	By Lemma \ref{lem:muliplier.single.x}, with probability at least $1-2\exp(-c\tau^2 n)$, for every $y\in \bar{\mathcal{S}}_r^\ast$, we have that
	\[M_{y,x^\ast}(j)\geq - \frac{ r^2}{32}	\
	\quad\text{on more than } \left(1-\frac{\tau}{2}\right)n \text{ blocks } j.\]
	Moreover, by Lemma \ref{lem:multiplier.uniform.x}, with probability at least $1-2\exp(-c\tau^2 n)$, for every $x\in \mathcal{S}_r^\ast$ there is $y\in \bar{\mathcal{S}}_r^\ast$ such that
	\[ M_{x,x^\ast}(j)\geq  M_{y,x^\ast}(i)- \frac{r^2}{32}\
	\quad\text{on more than } \left(1-\frac{\tau}{2}\right)n \text{ blocks } j.\]
	Combining the two estimates, it follows that, with probability at least $1-2\exp(-c\tau^2 n)$, for every $x\in \mathcal{S}_r^\ast$,
	\[M_{x,x^\ast}(j)\geq -\frac{r^2}{16}
	\quad\text{on more than } (1-\tau)n \text{ blocks } j,\]
	which is exactly what we wanted to show.
\end{proof}

\subsection{Prediction error}

In this section we shall prove Proposition \ref{prop:prediction.error}, dealing with the prediction error.
To that end, let
\[\mathcal{U}_r^\ast:= \{ x\in \mathcal{B}_r^\ast : f(x) \geq f(x^\ast) + 2 c_{\mathrm{H}} r^2\}\]
be the set of all $x\in\mathcal{B}_r^\ast$ that are in an ``unfavorable position". Let us stress again that if $x^\ast$ lies in the interior of $\mathcal{X}$, then $\mathcal{U}_r^\ast$ is empty and the estimate on the prediction error holds automatically. We therefore assume that $\mathcal{U}_r^\ast$ is not empty.

The proof of Proposition \ref{prop:prediction.error} relies on the convexity of $F$:
for any $x,y\in\mathcal{X}$ and $j$, we have that
\begin{align*}
\widehat{f}_{I_j'}(x) - \widehat{f}_{I_j'}(y)
&\geq \frac{1}{m}\sum_{i\in I_j'} \langle \nabla F(y,\xi_i),x-y\rangle
=: M_{x,y}'(j) .
\end{align*}
In particular, this implies that

\vspace{0.5em}
\begin{enumerate}[(i)]
	\item
	$x^\ast$ wins its home match against $x$ if
	\[ M_{x,x^\ast}'(j) \geq -\frac{c_{\mathrm{H}} r^2}{4} \text{ on more than }\frac{n}{2} \text{ blocks } j, \]
	\item
	$x$ does not win its home match against $x^\ast$ if
	\[ M_{x,x^\ast}'(j) > \frac{ c_{\mathrm{H}} r^2}{4} \text{ on more than }\frac{n}{2} \text{ blocks } j, \]
\end{enumerate}
\vspace{0.5em}

Thus, Proposition \ref{prop:prediction.error} is a consequence of the following lemma, which we shall prove in this section.

\begin{lemma}
\label{prop:actually.the.theorem.prediction}
	There is a constant $c$ such that the following holds.
	With probability at least $1-2\exp(-c \tau^2 n)$, for every $x\in\mathcal{B}_r^\ast$ and every $y\in\mathcal{U}_r^\ast$ we have that
	\begin{align}
	\label{eq:prediction.normal}
	\left| \left\{ j : M_{x,x^\ast}'(j) \geq -\frac{c_{\mathrm{H}} r^2}{4}  \right\} \right|
	&\geq (1-2\tau)n ,\\
	\label{eq:prediction.unfavorable}
	\left| \left\{ j : M_{y,x^\ast}'(j) > \frac{c_{\mathrm{H}} r^2}{4}  \right\} \right|
	&\geq (1-2\tau)n.
	\end{align}
\end{lemma}

Just as in the analysis of the estimation error, Proposition \ref{prop:prediction.error} is an easy consequence of Lemma \ref{prop:actually.the.theorem.prediction}:

\begin{proof}[Proof of Proposition \ref{prop:prediction.error}]
	By Proposition \ref{prop:estimation.error} we have that with probability at least $1-2\exp(-c\tau^2n)$,  $x^\ast\in \tilde{\mathcal{X}}_N^\ast$ and $\tilde{\mathcal{X}}_N^\ast\subseteq\mathcal{B}_r^\ast$.
	In the following we argue conditionally on that high probability event, and recall that $\tau<\frac{1}{4}$. By Lemma \ref{prop:actually.the.theorem.prediction}, with probability at least $1-2\exp(-c \tau^2 n)$, we have that $x^\ast$ wins its home match against every competitor in $\mathcal{B}_r^\ast$ (in particular, against every competitor in $\tilde{\mathcal{X}}_N^\ast$). Moreover, on the same event, every element in $\tilde{\mathcal{X}}_N^\ast$ that is in an unfavorable position loses its home match against $x^\ast$. Thus, $x^\ast\in \widehat{\mathcal{X}}_N^\ast$ and $\widehat{\mathcal{X}}_N^\ast\subseteq \mathcal{B}_r^\ast\setminus \mathcal{U}_r^\ast$, which is exactly what we wanted to show.
\end{proof}

The first part of Lemma \ref{prop:actually.the.theorem.prediction} (namely \eqref{eq:prediction.normal}) is an immediate consequence of Lemma \ref{prop:actually.the.theorem}.
Thus, in the following, we focus on the second part of Lemma \ref{prop:actually.the.theorem.prediction} (namely \eqref{eq:prediction.unfavorable}), dealing with $\mathcal{U}_r^\ast$. Our starting point is the following observation:

\begin{lemma}
\label{lem:unfavorable.position}
	Let $x\in\mathcal{U}_r^\ast$.
	Then we have that
	\[\langle \nabla f(x^\ast),x-x^\ast \rangle \geq c_{\mathrm{H}} r^2 .\]
\end{lemma}
\begin{proof}
Since $x \in \mathcal{U}_r^\ast$ we have that $2 c_{\mathrm{H}} r^2
	\leq f(x) - f(x^\ast)$. A Taylor expansion around $x^\ast$ shows that there is a midpoint $z$ such that
	\begin{align*}
	f(x) - f(x^\ast)
	&= \langle \nabla f(x^\ast),x-x^\ast\rangle + \frac{1}{2} \langle \nabla^2f(z) (x-x^\ast) ,x-x^\ast \rangle \\
	&\leq\langle \nabla f(x^\ast),x-x^\ast\rangle + c_{\mathrm{H}} r^2.
	\end{align*}
The proof clearly follows. 	
\end{proof}

 Thanks to Lemma \ref{lem:unfavorable.position}, the proof of the second part of Lemma \ref{prop:actually.the.theorem.prediction} (namely \eqref{eq:prediction.unfavorable}) follows the same path as the proof of the bound on the multiplier term in the context of the estimation error. Thus, we shall only sketch the argument for the sake of completeness.

\begin{lemma}
\label{lem:muliplier.unfavorable.single.x}
	There is an absolute constant $C$ such that, for every $x\in\mathcal{U}_r^\ast$, with probability at least $1-2\exp(-C\tau^2 n)$, we have that
	\begin{align}
	\label{eq:multi.unfavorable.single}
	\left| \left\{  j :  M_{x,x^\ast}'(j) \geq \frac{c_{\mathrm{H}} r^2}{2}\right\} \right| \geq ( 1-\tau) n.
	\end{align}
	In particular, the statement holds uniformly for sets $\bar{\mathcal{U}}_r^\ast\subseteq\mathcal{U}_r^\ast$ whose cardinality satisfies $\log(\frac{1}{2} |\bar{\mathcal{U}}_r^\ast|)\leq \frac{1}{2}C\tau^2n$.
\end{lemma}
\begin{proof}
	Fix $x\in\mathcal{U}_r^\ast$.
	By Lemma \ref{lem:unfavorable.position}, we have
	\[  \E[ \langle \nabla F(x^\ast,\xi),x-x^\ast\rangle ]
	\geq c_{\mathrm{H}}r^2 .\]
	Thus, exactly as in the proof of Lemma \ref{lem:muliplier.single.x}, we conclude that
	\[ \P\left[ M_{x,x^\ast}'(j) \leq \frac{c_{\mathrm{H}} r^2}{2} \right]
	\leq \frac{2\sqrt{\theta}}{ c_{\mathrm{H}} }
	\leq 2\sqrt{\theta}
	\leq \frac{\tau}{2}. \]
	(as long as $\theta$ is small enough, and the second inequality holds because $c_{\mathrm{H}} \geq 1$).
	A Binomial estimate (just as in the proof of Lemma \ref{lem:muliplier.single.x}) can be used to show that \eqref{eq:multi.unfavorable.single} holds with probability at least $1-2\exp(-C\tau^2 n)$.
\end{proof}

\begin{lemma}
\label{lem:multiplier.unfavorable.good.spearated.set}
	Let $C_0$ be the absolute constant of Lemma \ref{lem:muliplier.unfavorable.single.x}.
	Then there is an absolute constant $C_1$ and a set $\bar{\mathcal{U}}_r^\ast\subseteq\mathcal{U}_r^\ast$ whose cardinality satisfies $\log ( \frac{1}{2}|\bar{\mathcal{U}}_r^\ast| ) \leq \frac{1}{2}C_0n\tau^2$, such that the following holds.
	
	Let
	\[
	\rho:=\left(\frac{\mathrm{trace}(\mathbb{H}^{-1}\mathbb{G})}{ C_1 \tau^2 n}\right)^{\frac{1}{2}}
	\]
	For every $x\in\mathcal{U}_r^\ast$ there is $y=y(x)\in \bar{\mathcal{U}}_r^\ast$ with
	\begin{align}
	\label{eq:multiplier.unfavorable.good.set.covaraince}
	\langle \mathbb{G} (x-y),x-y\rangle^{\frac{1}{2}}
	&\leq 2 \rho r \quad\text{and}  , \\
	\label{eq:multiplier.unfavorable.good.set.gradient}
	\langle \nabla f(x^\ast),x-y\rangle
	&\geq 0.
	\end{align}
\end{lemma}
\begin{proof}
Just as in Lemma \ref{lem:multiplier.good.spearated.set}, we can construct a set $\tilde{\mathcal{B}}_r^\ast\subseteq\mathcal{B}_r^\ast$ satisfying \eqref{eq:multiplier.unfavorable.good.set.covaraince} with $2\rho r$ replaced by $\rho r$.
	The modification of that set is again similar:
	for $z\in\tilde{\mathcal{B}}_r^\ast$, define
	\[ y(z) \in\mathrm{argmin}\left\{  \langle \nabla f(x^\ast), y \rangle : y\in\mathcal{U}_r^\ast \text{ s.t.\ } \langle \mathbb{G}(y-z),y-z\rangle^\frac{1}{2}\leq \rho r \right\}. \]
	with the convention $y(z) := y_0$ for some fixed $y_0\in\mathcal{U}_r^\ast$ if the above set is empty.
		Then
	\[\bar{\mathcal{U}}_r^\ast:=\{ y(z): z\in\tilde{\mathcal{B}}_r^\ast\}\]
	 satisfies the statement of the lemma.
\end{proof}

Finally, fix the set $\bar{\mathcal{U}}_r^\ast$ of Lemma \ref{lem:multiplier.unfavorable.good.spearated.set}  and for $x\in\mathcal{U}_r^\ast$ denote by $y=y(x)\in\bar{\mathcal{U}}_r^\ast$ the best approximation in the cover constructed in Lemma \ref{lem:multiplier.unfavorable.good.spearated.set}.

\begin{lemma}
	\label{lem:multiplier.unfavorable.uniform.x}
	There is an absolute constant $C$ such that the following holds.
	With probability at least $1-2\exp(-C\tau^2n)$, for every $x\in\mathcal{U}_r^\ast$, we have that
	\[ \left|  \left\{ j : M_{x,x^\ast}'(j) \geq M_{y,x^\ast}'(j) - \frac{c_{\mathrm{H}} r^2}{8} \right\} \right| 
	\geq (1-\tau)n.\]
\end{lemma}
\begin{proof}
	Recall that $c_{\mathrm{H}}\geq 1$ by its definition. The claim follows (without any modification) just as in the proof of Lemma \ref{lem:multiplier.uniform.x}.
\end{proof}

\begin{proof}[Proof of Lemma \ref{prop:actually.the.theorem.prediction}]
	Let us prove \eqref{eq:prediction.normal}. Note that Lemma \ref{prop:actually.the.theorem} (without any modifications in the proof) yields the following:
	with probability at least $1-2\exp(-c\tau^2 n)$, for all $x\in \mathcal{B}_r^\ast$, we have that
	\[M_{x,x^\ast}'(j)
	\geq -\frac{c_{\mathrm{H}} r^2}{4} \text{ on more than }(1-\tau)n \text{ blocks } j. \]
	In particular, this holds for $\mathcal{U}_r^\ast\subseteq\mathcal{B}_r^\ast$.
	
	As for \eqref{eq:prediction.unfavorable}, a combination of Lemma \ref{lem:muliplier.unfavorable.single.x} and Lemma \ref{lem:multiplier.unfavorable.uniform.x} shows that with probability at least $1-2\exp(-c\tau^2 n)$, for every $x\in\mathcal{U}_r^\ast$, we have that
	\[ M_{x,x^\ast}'(j) > \frac{c_{\mathrm{H}} r^2}{4}
	\quad\text{on more than } (1-2\tau)n \text{ blocks } j.\]
	This completes the proof.
\end{proof}

\subsection{Proof under a deterministic lower bound of the Hessian}
\label{sec:proof.main.deterministic.lower.bound}

To conclude this section, let us prove Theorem \ref{thm:main.hessian.deterministic.lower.bound}.
There are only very few modifications needed in the proof of our main results, as we explain in what follows.

In the proof of Theorem \ref{thm:main.convex.intro}, the only place where the requirement 
\[N\geq c_2\max\{ d\log(2d), N_{\mathrm{H},\mathcal{E}}\}\]
was used, was for the statement of Lemma \ref{prop:actually.the.theorem} pertaining to the quadratic term, namely \eqref{eq:quad.term}.
However, Assumption \ref{ass:Deterministic.lower.bound.Hessian} clearly implies that, with probability 1, for every $x\in\mathcal{S}_r^\ast$, we have that
\[ \left| \left\{ j : \frac{1}{2}Q_{x,x^\ast}(j) \geq \frac{r^2\varepsilon}{2} \right\} \right| = n. \]
Thus, the only modification that is needed, is to prove the part of Lemma \ref{prop:actually.the.theorem} pertaining to the multiplier term (namely \eqref{eq:multi.term}) with $\frac{\varepsilon r^2}{4}$ instead of $\frac{r^2}{16}$.
Inspecting the proof, one readily sees that this is possible once the constants $\theta,\tau,c,c_1,\dots$ are allowed to depend on $\varepsilon$.

\addtocontents{toc}{\protect\setcounter{tocdepth}{1}}
%\addtocontents{toc}{\protect\setcounter{tocdepth}{2}}

\section{Proofs for the portfolio optimization problem}
\label{sec:PO.proof}

\subsection{The proof of Corollary \ref{cor:PO}}

The proof builds on several lemmas stated below, making the heuristic computations explained in the introduction rigorous.
Throughout, we work under the assumptions made in Section \ref{sec:PO}.
Note that
\begin{align*}
\nabla F(x,\xi)
&=-\ell'(V_x)  X, \\
\nabla^2 F(x,\xi)
&=\ell''(V_x)  X\otimes X
\end{align*}
where we recall $V_{x}=-Y-\langle X,x\rangle$  for $x\in\mathcal{X}$.
Finally, denote by
\[ \|\cdot\|_X
:= \E[ \langle X,\cdot\rangle^2]^{\frac{1}{2}}
= \langle \C[X]\cdot,\cdot\rangle^{\frac{1}{2}}
= \| \C[X]^{\frac{1}{2}} \cdot\|_2\]
the norm endowed by $X$.
Note that $\|\cdot\|_X$ is indeed a norm by the no-arbitrage condition.

\begin{lemma}
\label{lem:PO.norm.equals.normS}
	There is a constant $c_1>0$ depending on $L_X,v_1$ and a constant $c_2>0$ depending on  $L_X, v_1,v_2$  such that
	\[ c_1 \| \cdot \|_X
	\leq \|\cdot\|
	\leq c_2 \|\cdot\|_X. \]
	In particular, $\|\cdot\|$ is a true norm.
\end{lemma}
\begin{proof}
	We start with the second inequality.
	H\"older's inequality (with exponent $\frac{3}{2}$ and conjugate exponent $3$) implies that
	\begin{align*}
	\|z\|^2
	&=\E[ \ell''( V_{x^\ast} )  \langle X,z\rangle^2 ] \\	
	&\leq \E[ \ell''( V_{x^\ast} )^{\frac{3}{2}}]^{\frac{2}{3}} \E[\langle X,z\rangle^6]^{\frac{1}{3}}
	\end{align*}
	for every $z\in\mathbb{R}^d$.
	The first term is bounded by $v_2$ and  norm equivalence of $X$ (see \eqref{ass:portfolio.norm.equiv}) implies that the second term is bounded by $L_X^2 \|z\|_X^2$.

	We continue with the first inequality.
	The Paley-Zygmund inequality together with norm equivalence of $X$ implies that
	\begin{align}
	\label{eq:PO.paley.zygmund}
	\P\left[ \langle X,z\rangle^2 \geq \frac{1}{2} \|z\|_X^2 \right]
	&\geq \frac{( 1- \frac{1}{2})^2 \E[\langle X,z\rangle^2]^2 } { \E[\langle X,z\rangle^4]}
	\geq \frac{1}{4 L^4_X}
	\end{align}
	for every $z\in\mathbb{R}^d\setminus\{0\}$.
	Moreover, setting
	\[ \varepsilon:=\min\{  \ell''(u) : |u|\leq 8 L_X^4 \E[ |V_{x^\ast}| ] \} , \]
	an application of Markov's inequality shows that
	\[ 	\P[ \ell''( V_{x^\ast}) < \varepsilon]
	\leq \frac{1}{8 L_X^4}.\]
	In combination with \eqref{eq:PO.paley.zygmund}, we conclude that
	\begin{align*}
	\|z\|^2
	&\equiv \E[ \ell''( V_{x^\ast} )  \langle X,z\rangle^2 ]
	\geq \frac{\varepsilon \|z\|_X^2}{16 L_X^4} .
	\end{align*}
	
	Finally, as noted previously, the no-arbitrage condition \eqref{eq:PO.NA} immediately implies that $\|\cdot\|_X$ is a true norm.
	Hence, by the above, $\|\cdot\|$ is a true norm as well.
	This completes the proof.
\end{proof}

\begin{lemma}
\label{lem:PO.on.norm.equiv}
	Assumption \ref{ass:norm.equiv} is satisfied with a constant $L$ depending on $L_X, v_1, v_2$.
\end{lemma}
\begin{proof}
	An application of H\"older's inequality (with exponents $3,\frac{3}{2}$) implies that
	\begin{align*}
	\E[\langle \nabla^2 F(x^\ast,\xi)z,z \rangle^2]
	&=\E[\ell''( V_{x^\ast} )^2 \langle X,z\rangle^4] \\
	&\leq  \E[\ell''( V_{x^\ast} )^6]^{\frac{1}{3}} \E[\langle X,z\rangle^6]^{\frac{2}{3}}
	\end{align*}
	for every $z\in\mathbb{R}^d$.
	The first term is $v_2^2$ by definition and, by norm equivalence of $X$ (see \eqref{ass:portfolio.norm.equiv}) and Lemma \ref{lem:PO.norm.equals.normS}, the second term is bounded uniformly in $\{z\in\mathbb{R}^d : \|z\|\leq 1\}$ by a constant depending on $L_X,v_1$.
\end{proof}

\begin{lemma}
\label{lem:PO.sigma.NG}
	Let $L$ be the parameter from Assumption \ref{ass:norm.equiv}.
	Then
	\begin{align*}
	\sigma^2
	\leq \sqrt{L}  \bar{\sigma}^2
	\quad\text{and}\quad
	N_{\mathrm{G}}(r)
	\leq \sqrt{L}  \frac{\bar{\sigma}^2 d}{r^2}.
\end{align*}
\end{lemma}
\begin{proof}
	We make the preliminary claim that
	\[\C[ \nabla F(x^\ast,\xi)]
	\preceq \bar{\sigma}^2 \sqrt{L} \nabla^2 f(x^\ast).\]
	Indeed, for every $z\in\mathbb{R}^d$,
	\begin{align*}
	\langle \C[ \nabla F(x^\ast,\xi)]z,z\rangle
	&\leq \E[ \ell'(V_{x^\ast})^2 \langle X,z\rangle^2] \\
	&= \E \left[ \frac{\ell'(V_{x^\ast})^2}{\ell''(V_{x^\ast})}   \ell''(V_{x^\ast}) \langle X,z\rangle^2\right] \\
	&\leq \bar{\sigma}^2 \E[(\ell''(V_{x^\ast})\langle X,z\rangle^2)^2]^{\frac{1}{2}},
	\end{align*}	
	where the last step follows from H\"older's inequality.
	Moreover, by Assumption \ref{ass:norm.equiv},
	\begin{align*}
	\E[(\ell''(V_{x^\ast})\langle X,z\rangle^2)^2]^{\frac{1}{2}}
	&=\E[\langle \nabla^2 F(x^\ast,\xi)z,z\rangle^2]^{\frac{1}{2}}\\
	&\leq \sqrt{L}\|z\|^2 \\
	&=\sqrt{L} \langle \nabla^2 f(x^\ast)z,z\rangle
	\end{align*}
	hence the preliminary claim follows.
	
	As both the largest singular value and the trace are monotone w.r.t.\ the positive semidefinite order, the statement of the lemma follows.	
\end{proof}

We need the following simple auxiliary lemma.
Recall that $\|\cdot\|_\ast$ denotes the dual norm of $\|\cdot\|$.

\begin{lemma}
\label{lem:PO.expectation.norm}
	There is a constant $c$ depending on $L_X,v_1$ such that $\E[\|X\|_\ast^6]\leq c d^3$.
\end{lemma}
\begin{proof}
	By Lemma \ref{lem:PO.norm.equals.normS} we have $\|\cdot\|_X\leq c_1\|\cdot \|$ for a constant $c_1$ depending on $L_X,v_1$.
	This immediately implies that $\|\cdot \|_\ast\leq \frac{1}{c_1} \|\cdot\|_{X,\ast}$ where $\|\cdot\|_{X,\ast}$ is the dual norm of $\|\cdot\|_X$.
	As the norm $\|\cdot\|_X$ is endowed by $\C[X]$, its dual norm $\|\cdot\|_{X,\ast}$ is endowed by  $\C[X]^{-1}$, whence
	\[ \|X\|_{X,\ast} = \|Y\|_2 \quad\text{for } Y:=\C[X]^{-\frac{1}{2}} X.\]
	Applying H\"older's inequality (in its version for three random variables, with exponents $3,3,3$) shows that
	\[ \E[\|Y\|_2^6]
	=\sum_{i,j,k=1}^d \E[Y_i^2Y_j^2Y_k^2]
	\leq \sum_{i,j,k=1}^d \E[Y_i^6]^{\frac{1}{3}}\E[Y_j^6]^{\frac{1}{3}} \E[Y_k^6]^{\frac{1}{3}}.\]
	It remains to note that $Y$ satisfies the same norm equivalence as $X$ does, and therefore, denoting by $e_i$ the $i$-th standard Euclidean unit vector, 
	\begin{align*}
	\E[Y_i^6]^{\frac{1}{3}}
	&=\E[\langle Y,e_i\rangle^6]^{\frac{1}{3}} \\
	&\leq L_X^2 \E[\langle Y,e_i\rangle^2] \\
	&=L_X^2 \C[Y]_{ii}
	=L_X^2.
	\end{align*}
	Combining everything completes the proof.
\end{proof}

\begin{lemma}
\label{lem:PO.hessian.continuity}
	There is a constant $c>0$ depending on $L_X, v_1$ such that, for every $x,y\in\mathcal{B}_1^\ast$, we have that
	\begin{align}
	\label{eq:PO.hessian.directional}
	\P\left[ \mathcal{E}_{\mathrm{H}}(x) \leq \frac{ \|x-x^\ast\|^2}{8} \right]
	&\leq c v_{\mathcal{E}_{\mathrm{H}}} \|x-x^\ast\|, \\
	\label{eq:PO.hessian.difference.OP}
	\left\|  \nabla^2 F(x,\xi) - \nabla^2 F (y,\xi) \right\|_{\mathrm{op}}
	&\leq  \|x-y\|  K(\xi)
	\end{align}
	where $K(\xi)$ satisfies $\E[K(\xi)]\leq  c  v_{K} d^{\frac{3}{2}}$.
\end{lemma}
\begin{proof}
	As a preliminary observation, note that a Taylor expansion gives
	\begin{align}
	\label{eq:PO.midpoints.taylor}
	\begin{split}
	&\langle (\nabla^2 F(x ,\xi) -\nabla^2 F(y ,\xi) ) z, z\rangle \\
	&=\ell'''(V_{x+t(y-x)}) \langle X, x-y\rangle \langle X,z\rangle^2
	\end{split}
	\end{align}
	for every $z\in\mathbb{R}^d$, where $t\in[0,1]$ is some number depending on $x,y,\xi$.
	
	We start by proving \eqref{eq:PO.hessian.directional}.
	For every $x\in\mathcal{B}_1^\ast$, by \eqref{eq:PO.midpoints.taylor} and definition of $\mathcal{E}_{\mathrm{H}}$, we have that
	\begin{align*}
	\mathcal{E}_{\mathrm{H}}(x)
	&\equiv \sup_{t\in[0,1]} \left|\left\langle \left( \nabla^2 F(x^\ast+t(x-x^\ast) ,\xi) -\nabla^2 F(x^\ast ,\xi) \right) (x-x^\ast), x-x^\ast \right\rangle \right| \\
	& \leq \sup_{t\in[0,1]} |\ell'''(V_{x^\ast+t(x-x^\ast)}) | |\langle X, x-x^\ast\rangle|^3.
	\end{align*}
	In particular, applying H\"older's inequality, the norm equivalence of $X$ from \eqref{ass:portfolio.norm.equiv}, and Lemma \ref{lem:PO.norm.equals.normS}, we obtain
	\begin{align*}
	\E[\mathcal{E}_{\mathrm{H}}(x)]
	&\leq v_{\mathcal{E}_{\mathrm{H}}} \E[\langle X, x-x^\ast\rangle^6]^{\frac{1}{2}} \\
	&\leq  v_{\mathcal{E}_{\mathrm{H}}}  L_X^3 \|x-x^\ast\|_X^3 \\
	&\leq c_1 v_{\mathcal{E}_{\mathrm{H}}} \|x-x^\ast\|^3
	\end{align*}
	for a constant $c_1$ depending on $L_X,v_1$.
	The claim \eqref{eq:PO.hessian.directional} therefore follows from Markov's inequality.
	
	We now prove \eqref{eq:PO.hessian.difference.OP}.
	The definition of the operator norm together with the inequality $|\langle X, x-y\rangle|\leq \|X\|_\ast \|x-y\|$ (which holds by definition of the dual norm $\|\cdot \|_\ast$) shows that \eqref{eq:PO.midpoints.taylor} implies
	\begin{align*}
	&\| \nabla^2 F(x ,\xi) -\nabla^2 F(y ,\xi) \|_{\mathrm{op}} \\
	&\leq \sup_{\tilde{x},\tilde{y} \in\mathcal{B}_1^\ast \text{ and }t\in[0,1]} |\ell'''( V_{\tilde{x}+t(\tilde{y}-\tilde{x})})|  \|x-y\|  \|X\|_\ast^3   \\
	&=: K(\xi)   \| x-y\|.
	\end{align*}
	This proves \eqref{eq:PO.hessian.difference.OP}.
	
	It remains to control the expectation of $K(\xi)$.
	For every $\tilde{x},\tilde{y}\in\mathcal{B}_1^\ast$ and $t\in[0,1]$ we have $\tilde{x} +t(\tilde{y}-\tilde{x})\in\mathcal{B}_1^\ast$  by convexity of $\mathcal{X}$, whence $K(\xi)\leq \sup_{x\in \mathcal{B}_1^\ast } |\ell'''(V_{x})| \|X\|_\ast^3$.
	Therefore H\"older's inequality and the definition of $v_K$ imply that
	\begin{align*}
	\E[K(\xi)]
	&\leq v_{K} \E[ \|X\|_\ast^6]^{\frac{1}{2}}.
	\end{align*}
	By Lemma \ref{lem:PO.expectation.norm}, the last term is bounded by $c_2d^{\frac{3}{2}}$ for a constant $c_2$ depending on $L_X,v_1$.
	This completes the proof.
\end{proof}

\begin{lemma}
\label{lem:PO.cH}
	There is a constant $c$ depending on $L_X,v_1$ such that $c_{\mathrm{H}}\leq c v_{\mathcal{E}_{\mathrm{H}}}$.	
\end{lemma}
\begin{proof}
	We need to show that $\| \E[\nabla^2 F(x,\xi)] \|_{\mathrm{op}}\leq cv_{\mathcal{E}_{\mathrm{H}}}$ for every $x\in\mathcal{B}_1^\ast$.
	Fix such $x$.
	From the Taylor expansion \eqref{eq:PO.midpoints.taylor} we get
	\[\nabla^2 F(x,\xi)
	=\nabla^2 F(x^\ast,\xi)  + \ell'''(V_{x^\ast+t(x-x^\ast)}) \langle X,x-x^\ast\rangle X\otimes X\]
	for some $t\in[0,1]$ which depends on $x$ and $\xi$.
	
	The expectation of the first term equals $\nabla^2f(x^\ast)$, and the operator norm of the latter equals 1.
	To estimate the operator norm of the expectation of the second term, let $z\in\mathbb{R}^d$ with $\|z\|\leq 1$.
	Then H\"older's inequality (in its version for three random variables, with exponents $2,6,3$) implies
	\begin{align*}
	&\E[\ell'''(V_{x^\ast+t(x-x^\ast)}) \langle X,x-x^\ast\rangle \langle X,z\rangle^2] \\
	&\leq v_{\mathcal{E}_\mathrm{H}}   \E[\langle X,x-x^\ast\rangle^6]^{\frac{1}{6}}  \E[\langle X,z\rangle^6]^{\frac{1}{3}} \\
	&\leq v_{\mathcal{E}_\mathrm{H}}  L_X \|x-x^\ast\|_X \cdot  L_X^2 \|z\|_X^2,
	\end{align*}
	where the last inequality follows from norm equivalence of $X$ from \eqref{ass:portfolio.norm.equiv}.
	Finally, recalling that $\|x-x^\ast\|\leq 1$ and $\|z\|\leq 1$, the proof is completed by an application of Lemma \ref{lem:PO.norm.equals.normS} which states that $\|\cdot\|_X\leq c_1 \|\cdot\|$ for a constant depending on $L_X,v_1$.
\end{proof}

We are now ready for the

\begin{proof}[Proof of Corollary \ref{cor:PO}]
	Regarding Assumption \ref{ass:on.F.diff.convex.etc}:
	convexity and differentiability are clearly satisfied, and integrability holds by assumption.
	Moreover, by Lemma \ref{lem:PO.norm.equals.normS}, $\|\cdot\|$ is a true norm.
	
	Assumption \ref{ass:norm.equiv} follows from Lemma \ref{lem:PO.on.norm.equiv}.
	
	Lemma \ref{lem:PO.hessian.continuity} shows that Assumption \ref{ass:hessian.midpoints} is satisfied for $\alpha=1$ and
	\[ r_0:= \min\left\{1,\frac{c_0}{ v_{\mathcal{E}_{\mathrm{H}}}}\right\},\]
	where $c_0$ is a constant depending on the constant $c$ of Lemma \ref{lem:PO.hessian.continuity} and the parameter $L$ of Assumption \ref{ass:norm.equiv}; hence $c_0$ depends on $L_X,v_1,v_2$.
	
	Moreover, Lemma \ref{lem:PO.hessian.continuity} also gives $N_{\mathcal{E},\mathrm{H}}\leq c_1 d\log(d v_K+2)$ for a constant $c_1$ depending on $L_X,v_1$.
	
	The parameters $N_{\mathrm{G}}(r)$, $\sigma^2$, and $c_{\mathrm{H}}$ are bounded in Lemma \ref{lem:PO.sigma.NG} and Lemma \ref{lem:PO.cH}, respectively.
	Combining everything completes the proof.	
\end{proof}

\subsection{The proof of Corollary \ref{cor:PO.exp}}

First note that the $\bar{\sigma}$ of the Corollary \ref{cor:PO.exp} and the $\bar{\sigma}$ of Corollary \ref{cor:PO} coincide.
Moreover, as $X$ is Gaussian with non-degenerate covariance matrix, the no-arbitrage condition \eqref{eq:PO.NA} readily follows.
Further recall the well-known Gaussian norm equivalence
\[ \E[\langle X,z\rangle^p]^{\frac{1}{p}}
\leq C \sqrt{p}  \E[\langle X,z\rangle^2]^{\frac{1}{2}} \]
for every $z\in\mathbb{R}^d$ and every $p\geq 2$, where $C$ is an absolute constant.
In particular \eqref{ass:portfolio.norm.equiv} holds and our assumption that $U(2Y)$ is integrable implies part (c) of Assumption \ref{ass:on.F.diff.convex.etc}.
Along the same lines as Gaussian norm equivalence, the following holds.

\begin{lemma}
\label{lem:PO.exp.expectations}
	There is a constant $c$ depending on $L$ and $v_1\equiv \E[|V_{x^\ast}|]$ such that
	\begin{align*}
	\E[\|X\|_\ast^{12}]
	&\leq c d^6, \\
	\E[\exp(4\|X\|_\ast )]
	&\leq \exp(c \sqrt{d}),\\
	\E[ \langle X,x-x^\ast\rangle^{12}]
	&\leq c \|x-x^\ast\|^{12},\\
	\E[\exp(4|\langle X, x-x^\ast\rangle| )]
	&\leq c
	\end{align*}
	for every $x\in\mathcal{B}_1^\ast$.
\end{lemma}
\begin{proof}
	The two statements involving $\|X\|_\ast$ follow from similar arguments as given in Lemma \ref{lem:PO.expectation.norm}, noting that $\C[X]^{-\frac{1}{2}} X$ is standard Gaussian.
	The two statements involving $\langle X,x-x^\ast\rangle$ follow from Gaussian norm equivalence and the bound $\|\cdot\|_X\leq c\| \cdot \|$ from Lemma \ref{lem:PO.norm.equals.normS} for a constant $c$ depending on $L,v_1$.
\end{proof}

\begin{proof}[Proof of Corollary \ref{cor:PO.exp}]
	The only modifications needed pertain to Lemma \ref{lem:PO.hessian.continuity} and Lemma \ref{lem:PO.cH}, where terms can be simplified due to special features of the exponential function.
	
	We start with Lemma \ref{lem:PO.hessian.continuity}.
	As $\ell'''=\exp$ is increasing and as $\exp(a+b)=\exp(a)\exp(b)$ for $a,b\in\mathbb{R}$, we may use Remark \ref{rem:PO.ell'''} to get
	\begin{align*}
	\mathcal{E}_{\mathrm{H}}(x)
	&\leq  \exp( V_{x^\ast})   \exp(|\langle X, x-x^\ast\rangle| )    |\langle X,x-x^\ast\rangle|^3,\\
	\|\nabla^2 F(x,\xi)-\nabla^2 F(y,\xi)\|_{\mathrm{op}}
	&\leq \exp( V_{x^\ast})   \exp(\|X\|_\ast )   \|X\|_\ast^3 \|x-y\|\\
		&=:K(\xi) \|x-y\|.
	\end{align*}
	It remains to bound the expectation all terms.
	To that end, applying H\"older's inequality (in its version for three random variable, with exponents $2,4,4$) and Lemma \ref{lem:PO.exp.expectations} gives
	\begin{align}
	\label{eq:PO.exp.EE}
	\begin{split}
	\E[ \mathcal{E}_{\mathrm{H}}(x)]
	 &\leq \E[\exp( 2V_{x^\ast})]^{\frac{1}{2}} \E[\exp(4|\langle X, x-x^\ast\rangle| )]^{\frac{1}{4}} \E[ \langle X,x-x^\ast\rangle^{12}]^{\frac{1}{4}}\\
	&\leq \bar{\sigma}^2 c_1 \|x-x^\ast\|^3,
	\end{split}\\
	\label{eq:PO.exp.EK}
	\begin{split}
	\E[K(\xi)]
	&\leq \E[\exp( 2V_{x^\ast})]^{\frac{1}{2}}  \E[\exp(4\|X\|_\ast )]^{\frac{1}{4}} \E[\|X\|_\ast^{12}]^{\frac{1}{4}} \\
	&\leq \bar{\sigma}^2 \exp(c_1\sqrt{d}),
	\end{split}
	\end{align}
	for a constant $c_1$ depending on $L,v_1$.
	
	In particular, \eqref{eq:PO.exp.EE} shows that Assumption \ref{ass:hessian.midpoints} is satisfied for
	\[ r_0:= \min\left\{1,\frac{c_2}{\bar{\sigma}^2} \right\} \]
	where $c_2$ is a constant depending on $L,v_1$.
	In combination with \eqref{eq:PO.exp.EK}, this implies that
	\begin{align*}
	N_{\mathcal{E},\mathrm{H}}
	&\leq d\log( r_0 \bar{\sigma}^2 \exp(c_1 \sqrt{d}))  \\
	&\leq d \log(c_2 \exp(c_1 \sqrt{d}))
	\leq  c_3 d^{\frac{3}{2}}
	\end{align*}
	for a constant $c_3$ depending on $L,v_1$.
	
	Similar (but simpler) arguments show that $c_{\mathrm{H}}$ can be bounded in terms of $L,v_1$.	
	This completes the proof.
\end{proof}

\section{Concluding remarks}
\label{sec:concluding.remark}

\begin{remark}[Dependence of the procedure on the parameters]
\label{rem:parameter.dependence.optimal.action}
	All the parameters (e.g., $\sigma^2$, $\|\cdot\|$, $L$, $N_{\mathrm{G}}$, etc.\ ) that are required in the formulation of Theorem \ref{thm:main.convex.intro} (the only parameters needed in the definition of the procedure $\widehat{x}_N^\ast$ are $c_{\mathrm{H}}$, $L$, $\sigma$ and $r$) depend on the unknown optimal action $x^\ast$.

	While an a priori knowledge of the parameters seems unrealistic, there are various ways around this problem. 
	It should be stressed that any type of estimate on these parameters suffices to ensure the procedure performs well. For example, finding some $\hat{\sigma}^\prime$ such that $\frac{1}{2}\hat{\sigma}^\prime \leq \sigma \leq 2\hat{\sigma}^\prime$ is enough for our purposes, and estimating $\sigma$ within a constant multiplicative factor is a considerably simpler task than the ones we have to deal with in the analysis of the procedure $\widehat{x}_N^\ast$.
	
Alternatively, one can replace the parameters with the (local) worst-case scenario; for example, instead of $\sigma^2$, to consider	
\[ \bar{\sigma}^2:=\sup_{x \text{ close to }x^\ast } \sigma^2(x)\]
where $\sigma^2(x)$ is defined just as $\sigma^2$ but with gradient and Hessian evaluated at $x$ rather than at $x^\ast$.
Under mild smoothness assumptions on the function $f$, $\bar{\sigma}^2$ and $\sigma^2$ are the same order of magnitude; in particular, replacing $\sigma^2$ by $\bar{\sigma}^2$ will only result in a change of constants.
And to that end it suffices to find crude estimates of $x^\ast$, which is a relatively simple task.

Finally, it is much simpler to test whether a solution is a good one than producing a candidate. Therefore, one can increase the sample size and test the candidates that are produced. Once the sample size passes the critical threshold from Theorem \ref{thm:main.convex.intro}, a good candidate will be identified.

All of these are standard methods and there are plenty of other alternatives to tackle such issues. We shall not pursue these aspects further in this article.
\end{remark}

\begin{remark}[On the integrability of the Hessian] %Assumption \ref{ass:norm.equiv}]
\label{rem:hessian.SLB}
	In the course of the proof of Theorem \ref{thm:main.convex.intro}, the only place Assumption \ref{ass:norm.equiv} was used was in Lemma \ref{lem:hessian.at.optimzer}, and there it was used twice:
	firstly, by Remark \ref{rem:relation.SLB.SBP.norm.equiv}, Assumption \ref{ass:norm.equiv} guarantees the existence of three constants $s_0,s_1,s_2>0$ depending only on $L$, such that, for every $m\geq 1$, the random matrix
	\begin{align}
	\label{eq:hessian.SLB}
	\nabla^2 F(x^\ast,\xi)
	\begin{array}{l}
	\text{satisfies a stable lower bound with}\\
	\text{parameters }	(m,s_0,2s_1 m,s_2 m ).
	\end{array}
	\end{align}
	(In fact, one can choose $s_0=\frac{1}{4}$ as we did for notational purposes).
	On the other hand, by Lemma \ref{lem:operator.norm.relation} and Lemma \ref{lem:growth.op.squared.A}, Assumption \ref{ass:norm.equiv} implies that the minimal sample size from Theorem \ref{thm:singular.value},
	\[ N_{\mathrm{H}}:= \max\left\{ \| \E[ A \mathbb{A}^{-1} A] \|_{\mathrm{op}} ,  \log\left( \log(3d)  \frac{\E[ \| A \mathbb{A}^{-1} A \|_{\mathrm{op}} ] }{ \|\E[A \mathbb{A}^{-1} A]\|_{\mathrm{op}} } \right) \right\} \]
	(where $A:=\nabla^2 F(x^\ast,\xi)$ and $\mathbb{A}:=\nabla^2f(x^\ast)$) can be bounded by $c d\log(2d)$ for a constant $c$ depending only on $L$.
		
	In particular, we see that Theorem \ref{thm:main.convex.intro} remains valid if Assumption \ref{ass:norm.equiv} is replaced by assumption \eqref{eq:hessian.SLB} together with the requirement that the sample size exceeds $c_2 N_{\mathrm{H}}$
	(in that case the constants $c_1,c_2,c_3$ appearing in Theorem \ref{thm:main.convex.intro} depend on $s_0,s_1,s_2$).
\end{remark}

\vspace{1em}
\noindent
\textsc{Acknowledgements:}
Daniel Bartl is grateful for financial support through the Vienna Science and Technology Fund (WWTF) project MA16-021 and the Austrian Science Fund (FWF) project P28661.

Part of this work was conducted while Shahar Mendelson was visiting the Faculty of Mathematics, University of Vienna, and the Erwin Schr\"{o}dinger Institute, Vienna. He would also like to thank Jungo Connectivity for its support.

The authors would like to thank two anonymous referees for valuable comments.

\bibliographystyle{abbrv}
% \bibliography{bib_saa.bib}

%-----------------------   bibliography   ----------------------------------------
\end{document}